\documentclass[reqno]{amsart}
\usepackage{url}
\usepackage{amssymb,amsthm,amsfonts,amstext}
\usepackage{amsmath}
\usepackage{mathptmx}       
\usepackage{helvet}         
\usepackage{courier}        
\usepackage{graphicx}
\usepackage{enumerate}
\usepackage[active]{srcltx}
\usepackage{mathrsfs}

\usepackage{tikz}
\usetikzlibrary{backgrounds}
\usetikzlibrary{patterns,fadings}
\usetikzlibrary{arrows,decorations.pathmorphing}
\usetikzlibrary{calc}
\definecolor{light-gray}{gray}{0.95}
\usepackage{graphicx}
\usepackage{epsfig}
\usetikzlibrary{shapes}

\usepackage{color}

\newcommand\ve{\varepsilon}

\newtheorem{theorem}{Theorem}[section]
\newtheorem{lemma}[theorem]{Lemma}
\newtheorem{proposition}[theorem]{Proposition}

\numberwithin{equation}{section}

\newcommand{\mc}[1]{{\mathcal #1}}

\newcommand{\bb}[1]{{\mathbb #1}}

\renewcommand{\epsilon}{\varepsilon}

\newcommand{\RR}{\mathbb R}

\newcommand{\bM}{\mathbf{M}}

\let\ve=\varepsilon

\let\ve=\varepsilon

\let\D=\Delta

\begin{document}

\author{C\'edric Bernardin}

\address{\noindent
Universit\'e de Lyon and CNRS, UMPA, UMR-CNRS 5669, ENS-Lyon,
46, all\'ee d'Italie, 69364 Lyon Cedex 07 - France.
\newline e-mail: \rm \texttt{cedric.bernardin@ens-lyon.fr}
}

\author{Patr\'icia Gon\c{c}alves}

\address{\noindent
 Departamento de Matem\'atica, PUC-RIO, Rua Marqu\^es de S\~ao Vicente, no. 225, 22453-900, Rio de Janeiro, Rj-Brazil and CMAT, Centro de Matem\'atica da Universidade do Minho, Campus de Gualtar, 4710-057 Braga, Portugal.\newline e-mail: \rm \texttt{patricia@mat.puc-rio.br and patg@math.uminho.pt}}

\author{Claudio Landim}

\address{\noindent IMPA, Estrada Dona Castorina 110, CEP 22460 Rio de
  Janeiro, Brasil and CNRS UMR 6085, Universit\'e de Rouen, Avenue de
  l'Universit\'e, BP.12, Technop\^ole du Madril\-let, F76801
  Saint-\'Etienne-du-Rouvray, France.  \newline e-mail: \rm
  \texttt{landim@impa.br} }
\title[]{Entropy of non-equilibrium stationary measures of boundary driven TASEP}

\noindent

\begin{abstract}
  We examine the entropy of non-equilibrium stationary states  of boundary driven totally asymmetric simple exclusion processes. As a consequence, we obtain that the Gibbs-Shannon entropy of the non equilibrium stationary state converges to the Gibbs-Shannon entropy of the local equilibrium state. Moreover, we prove that its fluctuations are Gaussian, except when the mean displacement of particles produced by the bulk dynamics agrees  with the particle flux induced by the density reservoirs in the maximal phase regime.
\end{abstract}

\maketitle
\thispagestyle{empty}

\section{Introduction}

Nonequilibrium stationary states (NESS) maintained by systems in
contact with infinite reservoirs at the boundaries have attracted much
attention in these last years.  In analogy with the usual Boltzmann
entropy for equilibrium stationary states, we introduced in \cite{BL}
the entropy function of NESS and we computed it explicitly in the case
of the boundary driven symmetric simple exclusion process. In the
present paper we extend this work to the boundary driven totally
asymmetric simple exclusion process (TASEP) and we show that the
entropy function detects phase transitions.

The boundary driven asymmetric simple exclusion process is defined as follows. Let $p=1-q \in [0,1] \ne 1/2$ and $0< \rho_- < \rho_+ <1$. The microstates are described by the vectors $\eta=(\eta_{-N}, \ldots, \eta_{N})\in \Omega_N:=\{0,1\}^{\{-N, \ldots, N\}}$ where for $x\in\{-N,\cdots, N\}$, $\eta_x =1$ if the site $x$ is occupied and $\eta_x=0$ if the site $x$ is empty. In the bulk of the system, each particle, independently from the others, performs a nearest-neighbor asymmetric random walk, where jumps to the right (resp. left) neighboring site occur at rate $p$ (resp. rate $q$), with the convention that each time a particle attempts to jump to a site already occupied, the jump is suppressed in order to respect the exclusion constrain. At the two boundaries the dynamics is modified to mimic the coupling with reservoirs of particles: if the site $-N$ is empty (resp. occupied), a particle is injected at rate $\alpha$ (resp. removed at rate $\gamma$); similarly, if the site $N$ is empty, a particle is injected at rate $\delta$ (resp. removed at rate $\beta$). For any sites $x \ne y$, we denote by $\sigma^{x,y} \eta$ (resp. $\sigma^x (\eta)$) the configuration obtained from $ \eta \in \Omega_N$ by the exchange of the occupation variables $\eta_x$ and $\eta_y$ (resp. by the change of $\eta_x$ into $1-\eta_x$). The boundary driven (nearest neighbor) asymmetric simple exclusion process is the Markov process on $\Omega$ whose generator $L$ is given by
\begin{equation*}
L = L_{0} + L_{-} +L_{+},
\end{equation*}
where $L_{0},L_{-},L_{+}$ act on functions $f:\Omega_N \to \RR$ as follows
\begin{equation*}
\begin{split}
&(L_{0} f)(\eta) = \sum_{x=-N}^{N-1} \left\{p \eta_x (1-\eta_{x+1}) + q \eta_{x+1} (1- \eta_x) \right\}\left[ f (\sigma^{x,x+1} \eta) -f(\eta)\right],\\
&(L_{-} f)(\eta)= c_{-}(\eta_{-N}) \left[ f(\sigma^{-N} \eta) -f (\eta)\right], \quad (L_{+} f)(\eta)= c_{+}(\eta_{N}) \left[ f(\sigma^{N} \eta) -f (\eta)\right]
\end{split}
\end{equation*}
with $c_{\pm} : \Omega_N \to [0,+\infty)$ given by
$$c_{-} (\eta) = \alpha (1-\eta_{-N}) + \gamma \eta_{-N} , \quad c_+ (\eta) =\delta (1-\eta_N) + \beta \eta_N.$$

The density of the left (resp. right) reservoir is denoted by $\rho_-$ (resp. $\rho_+$) and can be explicitly computed as a function of $p,q,\alpha, \gamma$ (resp. $p,q,\beta,\delta$). For simplicity we will focus only on the totally asymmetric simple exclusion process (TASEP) which corresponds to $p=0$ or $p=1$. Furthermore, if $p=1-q=1$ we take $\alpha=\rho_-$, $\beta=1-\rho_+$, $\gamma =\delta=0$. If $p=1-q=0$, we take $\delta=\rho_+$, $\gamma= 1-\rho_-$ and $\alpha=\beta=0$. Since $\rho_- < \rho_+$ the reservoirs induce a flux of particles from the right to the left. On the other hand the bulk dynamics produces a mean displacement of the particles with a drift equal to $(p-q)$.  For $p=0$  both effects cooperate to push the particles to the left and we call the corresponding system  the {\textit{cooperative}} TASEP. If $p=1$ the two effects push the particles in opposite directions and we call the corresponding system the {\textit{competitive}} TASEP.

The unique non-equilibrium stationary state of the boundary driven TASEP  is denoted by $\mu_{ss,N}$. In the case $\rho_- =\rho_+=\rho \in (0,1)$, $\mu_{ss,N}$ is given by the Bernoulli product measure $\nu_{\rho}$ on $\Omega_N$. In the non-equilibrium situation, the steady state has a lot of non-trivial interesting properties.  The phase diagram for the average density ${\bar \rho}$ is well known and one can distinguish three phases: the high-density phase (HD) for which ${\bar \rho}=\rho_+$, the low density phase (LD) for which ${\bar \rho}=\rho_-$ and the maximal current phase (MC) where ${\bar \rho}=1/2$, see \cite{DLS3}. The transition lines between these phases are second order phase transitions except for the boundary $\rho_- + \rho_+ =1$ in the competitive case where the transition is of first order. On this line, the typical configurations are shocks between LD phase with density $\rho_-$ at the left of the shock and HD phase with density $\rho_+$ at the right of the shock. The position of the shock is uniformly distributed along the system and the average profile ${\bar \rho} (x)$ is given by ${\bar \rho}(x)=\rho_- {\bf 1}\{x \le 0\} + {\rho_+} {\bf 1}\{x \ge 0\}$. This is summarized in Figure \ref{fig:pddens}.
\begin{center}
\begin{figure}[h!]
\label{fig:pddens}
\begin{tabular}{ccc}
\begin{tikzpicture}[scale=0.3]
\draw[->,>=latex] (0,0) -- (13,0);
\draw[-,>=latex,dashed] (0,10) -- (10,10);
\draw[-,>=latex,dashed] (0,5) -- (5,5) -- (5,10);
\draw[->,>=latex] (0,0) -- (0,13);
\draw (0,12) node[left]{$\rho_+$};
\draw (12,0) node[below]{$\rho_-$};
\draw (0,0) node[below]{$0$};
\draw (0,0) node[left]{$0$};
\draw (10,0) node[below]{$1$};
\draw (0,10) node[left]{$1$};
\fill (0,0) -- (10,0) -- (10,10) -- cycle;
\draw (1.25, 3.75) node[right] {${\rho_+}$};
\draw (1.5,7.5) node[right]{$\frac{1}{2}$};
\draw (6.25, 8.75) node[right] {${\rho_-}$};
\end{tikzpicture}
&
\phantom{aaa}
&
\begin{tikzpicture}[scale=0.3]
\draw[->,>=latex] (0,0) -- (13,0);
\draw[-,>=latex,dashed] (0,10) -- (10,10);
\draw[-,>=latex,dashed] (0,10) -- (5,5);
\draw[->,>=latex] (0,0) -- (0,13);
\draw (0,12) node[left]{$\rho_+$};
\draw (12,0) node[below]{$\rho_-$};
\draw (0,0) node[below]{$0$};
\draw (0,0) node[left]{$0$};
\draw (10,0) node[below]{$1$};
\draw (0,10) node[left]{$1$};
\fill (0,0) -- (10,0) -- (10,10) -- cycle;
\draw (2.5, 5) node[left] {${\rho_-}$};
\draw (5,7.5) node[right] {${\rho_+}$};

\end{tikzpicture}
\end{tabular}
\caption{The phase diagram for the cooperative TASEP (left) and the competitive TASEP (right).}
\end{figure}
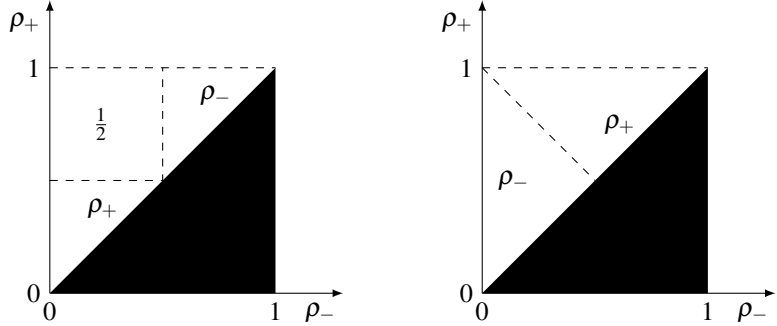
\end{center}
The entropy function of $\mu_{ss,N}$ introduced in \cite{BL} is the function $S^{p,q}_{\rho_-, \rho_+}: [0,+\infty) \to [0,\log(2)] \cup \{-\infty\}$ defined by
\begin{equation*}
S (E)= \lim_{\delta \to 0} \lim_{N \to + \infty} \frac{1}{2N +1} \log \Big(\sum_{\eta \in \Omega_N} \, {\bf 1} \Big\{\left| \frac{1}{2N +1} \log( \mu_{ss,N} (\eta)) + E\right| \le \delta \Big\}\Big)
\end{equation*}
if the limit exists.

Observe that $J_{\rho_-, \rho_+}^{p,q} (E) = E- S_{\rho_-, \rho_+}^{p,q} (E) \ge 0$ coincides with the large deviations function of the random variables
$Y_N (\eta):=-\cfrac{1}{2N +1} \log(\mu_{ss,N} (\eta))$ under the probability measure $\mu_{ss,N}$. Therefore, the (concave) Legendre transform of the entropy function $S:=
S_{\rho_-, \rho_+}^{p,q}$,
\begin{equation}
\label{eq:pressure-def}
P(\theta):=P_{\rho_-, \rho_+}^{p,q} (\theta) =\inf_{E \ge 0} \{\theta E -S (E) \},
\end{equation}
that we call the {\textit{ pressure}}, is by the Laplace-Varadhan theorem simply related to the cumulant generating function of the random variables $\{Y_N\}_N$, i.e.
\begin{equation}
\label{eq:pressure-cumul}
\begin{split}
P (\theta)&= - \lim_{N \to + \infty} \cfrac{1}{2N+1} \log \left( \int e^{(2N+1) (1-\theta) Y_N(\eta)}\mu_{ss,N} (d\eta)\right)\\
&= - \lim_{N \to + \infty} \cfrac{1}{2N+1} \log \left( \sum_{\eta \in \Omega_N} (\mu_{ss,N} (\eta) )^{\theta}\right).
\end{split}
\end{equation}

In the equilibrium case $\rho_- =\rho_+=\rho \in (0,1)$, denoting by
\begin{equation}\label{chemical potential}
\varphi = \log(\frac{\rho}{1-\rho}) \in \RR
\end{equation}
 the corresponding chemical potential, it is easy to show that the entropy function is given by
\begin{equation}
\label{eq:entropysimple}
S_{\rho,\rho} (E) = -s \left( \frac{-E +\log (1+e^{\varphi})}{\varphi}\right)
\end{equation}
where $s(\theta)= \theta \log (\theta) + (1-\theta) \log (1-\theta)$. The pressure $P(\varphi, \theta)$ is then given by
\begin{equation}
 \label{Pfunction}
P(\varphi,\theta):=\theta \log (1 +e^{\varphi}) - \log (1+e^{\varphi \theta}).
\end{equation}

In the non-equilibrium case $\rho_- \ne \rho_+$, since $\mu_{ss,N}$ has not a simple form, the computation of the entropy function is much more difficult.  It has
been proved in \cite{BL} that if a strong form of local equilibrium holds (see Section \ref{sec:ls} for a precise definition), then
the entropy function $S:=S^{p,q}_{\rho_-, \rho_+}$ can be expressed in a variational form involving the non-equilibrium free energy $V:=V_{\rho_-, \rho_+}^{p,q}$ and
 the Gibbs-Shannon entropy ${\bb S}$:
\begin{equation}
\label{eq:1.9}
S(E)= \sup_{\rho \in {\mc M}} \left\{ {\bb S} (\rho) \; ; \; V (\rho) +{\bb S} (\rho) =E \right\},
\end{equation}
where the set of density profiles $\mathcal{M}$ is defined in \eqref{set: profiles} and the Gibbs-Shannon entropy of the profile $\rho \in {\mc M}$ is defined by
\begin{equation} \label{eq: S}
{\bb S} (\rho) = - \frac{1}{2} \int_{-1}^1 s(\rho(x)) dx.
\end{equation}

The interval composed of the $E \in [0,+\infty)$ such that $S(E) \ne -\infty$ is called the \textit{energy band}. The bottom and the top of the energy band are defined respectively by
\begin{equation}\label{energy band}
E^-:=\inf_{\rho\in\mathcal{M}}\{\mathbb{S}(\rho)+V(\rho)\} \quad \textrm{and}\quad E^+:=\sup_{\rho\in\mathcal{M}}\{\mathbb{S}(\rho)+V(\rho)\}
\end{equation}

The non-equilibrium free energy is the large deviation function of the empirical density under $\mu_{ss,N}$. Its value does not depend on $p$ nor $q$ but only on the sign of $p-q$, and
 we denote it by $V^+$ if $p-q>0$ and by $V^-$ if $p-q<0$. The explicit computation of this functional has been obtained first in \cite{DLS3} and generalized to other systems in \cite{B2}.   Similarly, the entropy (resp. pressure) of the competitive TASEP is denoted by $S^+$ (resp. $P^+$) and the entropy (resp. pressure) of the cooperative TASEP by $S^-$ (resp. $P^-$). It follows easily from (\ref{eq:pressure-def}) and (\ref{eq:1.9}) that
\begin{equation*}
P(\theta)= \inf_{\rho \in {\mc M}} \left\{ \theta (V (\rho) +{\bb S} (\rho) ) -{\bb S} (\rho) \right\}.
\end{equation*}
This formula can also be obtained starting from (\ref{eq:pressure-cumul}) and using the local equilibrium statement as it is done in \cite{BL} for the entropy function.

In this paper we compute explicitly  $S^+$ and $S^-$ (resp. Theorem \ref{th:ent1} and Theorem \ref{th:ent2}) and  $P^+$ and $P_-$ (resp. Theorem \ref{th:pressure-competitive} and Theorem  \ref{th:pressure-cooperative}). From those results we deduce several interesting consequences (see Theorem \ref{th:cons1} and Theorem \ref{th:cons2}):
\begin{itemize}
\item We recover some results of \cite{B} for the TASEP, showing that the Gibbs-Shannon entropy of the non-equilibrium stationary state of the TASEP is the same, in the thermodynamic limit, as the Gibbs-Shannon  of the local Gibbs equilibrium measure, see Theorems \ref{th:cons1} and \ref{th:cons2}. In this case, the local Gibbs equilibrium measure is $\nu_{\bar\rho}$, namely, the Bernoulli product measure $\otimes_{x=-N}^N {\mc B}({\bar \rho}(x))$  where ${\mc B} ( r )$ is the one-site Bernoulli measure on $\{0,1\}$ with density $r$ and $\bar{\rho}$ is the stationary profile.
\item For the competitive TASEP, contrarily to what happens for the boundary driven symmetric simple exclusion  process (\cite{BL,DLS}),  the fluctuations are Gaussian with the same variance as the one given by the local equilibrium  state.
\item For the cooperative TASEP, the same occurs if $\rho_-,\rho_+ \le 1/2$ or if $\rho_-,\rho_+ \ge 1/2$. But in the MC phase $\rho_- \le 1/2\le \rho_+$, the fluctuations are not Gaussian. This is reminiscent of \cite{DEL, DLS3} where it is shown that the fluctuations of the density are non-Gaussian {\footnote{The non-Gaussian part of the fluctuations can be described in terms of the statistical properties of a Brownian excursion (\cite{DEL}).}}. 
\end{itemize}
Our last results concern the presence of phase transitions {\footnote{We refer the interested reader to \cite{T} for more informations about the implications of these facts from a physical viewpoint.}} for the competitive and the cooperative TASEP.  For the cooperative TASEP the function $S^-$ is a continuously differentiable concave function on its energy band but has linear parts. As a consequence the pressure function $P^-$ is a concave function with a discontinuous derivative. The function $P^-$ may also have a linear part due to the fact that the entropy $S^-$ does not necessarily vanish at the boundaries of the energy band.
If $\rho_- \le 1/2 \le \rho_+$, then the function $S^+$ is a smooth concave function on its energy band, but does not vanish at the top of the energy band. Consequently the pressure function $P^+$ is concave with a linear part on an infinite interval. If $\rho_-, \rho_+ \le 1/2$ or $\rho_-,\rho_+ \ge 1/2$, the entropy function $S^+$ has a discontinuity of its derivative at some point in the interior of the energy band but vanishes at the boundaries of the energy band. Then, the pressure function $P^+$ has a linear part on a finite interval.

It would be interesting to see how these results extend to other asymmetric systems for which the quasi-potential has been explicitly computed (\cite{B2}). The form of the entropy function obtained for the TASEP is relatively simple but follows from long computations. We did not succeed in giving a simple intuitive explanation to the final formulas obtained. We also notice that extending these results to a larger class of systems would require to prove the strong form of local equilibrium for them in order to get (\ref{eq:1.9}). This seems to be a difficult task.

The paper is organized as follows. In Section \ref{sec:compT} we obtain the entropy and the pressure functions for the competitive TASEP and deduce some consequences of these computations. In Section \ref{sec:coopT} we obtain similar results for the cooperative TASEP. The local equilibrium statement is proved in Section \ref{sec:ls}. Technical parts are postponed to the Appendix.

\section{Competitive TASEP}
\label{sec:compT}
In this section we derive the variational formula for the entropy function (\ref{eq:1.9}) for the competitive TASEP. Denote by $\chi(\rho)$, the mobility of the system,
that is $\chi:[0,1] \to [0,1]$ is defined by $\chi (\rho) =\rho (1-\rho)$. The chemical potential corresponding to $\rho_\pm$ is denoted by $\varphi_{\pm}$ and satisfies $\rho_{\pm} = e^{\varphi_{\pm}}/ (1+e^{\varphi_{\pm}})$, see \eqref{chemical potential}.

We consider the set ${\bb L}^{\infty} ([-1,1])$ equipped with the weak{$\star$} topology and ${\mc M}$ as the set
\begin{equation}\label{set: profiles}
{\mc M} = \left\{ \rho \in {\bb L}^{\infty} ([-1,1])\; : \; 0 \le \rho \le 1\right\}
\end{equation}
which is equipped with the relative topology. Denote by $\bar\rho$ the stationary density profile. We recall that $\bar \rho = \rho_-$ for $\rho_+ < 1-\rho_-$, i.e. $\varphi_+ < - \varphi_-$, $\bar \rho = \rho_+$ for $\rho_+ >1-\rho_-$, i.e. $\varphi_+ > - \varphi_-$ and
 ${\bar \rho} (x)= \rho_- {\bf 1}\{x \le 0\} + \rho_+ {\bf 1}\{x \ge 0\}$ if $\rho_- + \rho_+ =1$. Let $\bar \varphi= \sup(\varphi_+, - \varphi_-)$ so that
 $\bar\rho = e^{\bar\varphi}/({1+e^{\bar\varphi}})$ if $\rho_-+\rho_+ \ne 1$. Let
\[
{\Phi}=\Big\{ \varphi:=\varphi_y : x \in [-1,1] \to \varphi_{-} {\bf 1}\{-1\leq x<y\} + {\varphi_+} {\bf 1}\{y\leq x\leq 1\} \; ; \; y \in [-1,1]  \Big\}.
 \]
For $(\rho, \varphi) \in {\mc M} \times \Phi$ we define the functional
\begin{equation}\label{H function}
{\mc H} (\rho, \varphi) = \frac{1}{2}\int_{-1}^{1} \left[ (1-\rho(x)) \varphi(x) - \log (1+e^{\varphi(x)}) \right]dx.
\end{equation}
Then the quasi-potential of the competitive TASEP is given (\cite{DLS3}, \cite{B2}) by
\begin{equation*}
V^+ (\rho) = -{\bb S} (\rho) + \inf_{\varphi \in \Phi}{\mc H} (\rho, \varphi) - {\bar V}^+
\end{equation*}
where
\begin{equation*}
{\bar V}^+  =  -{\bb S} ({\bar \rho}) + \inf_{\varphi \in \Phi}  {\mc H} ({\bar \rho}, \varphi)
=\log\Big(\min_{\rho\in[\rho_{-},\rho_{+}]}\chi( \rho)\Big).
\end{equation*}

 Let us also introduce $\varphi_0 = \sup \{ |\varphi_-|, |\varphi_+|\}$ and $\rho_0 =e^{\varphi_0}/(1+e^{\varphi_0})$.\\

For each $E\ge 0$, $m \in [0,2]$ and $\varphi_-,\varphi_+$ we define
\begin{equation}\label{chi0}
\xi_0 := \cfrac{\log (1+e^{\varphi_+}) - \log (1+e^{\varphi_-}) }{\varphi_+ - \varphi_-} \in (0,1), \hspace{0.3cm} {\hat \xi}_0 :=\cfrac{\log (1+e^{\varphi_+}) + \log (1+e^{\varphi_-}) }{\varphi_+ - \varphi_-}.
\end{equation}

\subsection{Energy bands}

In this section we determine the energy band of the competitive TASEP. This is summarized in Figure \ref{fig:ebcomp}.

\begin{proposition} \label{EBcompetitive}
The bottom  of the energy band is given by
\begin{equation*}
E_{+\infty}^- = -{\bar V}^+ -\log (1+e^{\varphi_0}),
\end{equation*}
where $\varphi_0 = \sup \{ |\varphi_-|, |\varphi_+|\}$ and the top  of the energy band is given by
\begin{equation*}
E_{+\infty}^+ = -{\bar V}^+ +
\begin{cases}
\vspace{0.2cm}
\cfrac{\varphi_- \log (1+e^{\varphi_+}) - {\varphi_+} \log (1+e^{\varphi_-})}{\varphi_+ -\varphi_-}, \quad {\rho_- \le \frac{1}{2} \le \rho_+},\\
\vspace{0.2cm}
 - \log (1+e^{\varphi_+}), \quad \rho_- < \rho_+ \le \frac{1}{2},\\
\vspace{0.2cm}
 -\log (1+e^{- \varphi_-}), \quad \frac{1}{2} \le \rho_- < \rho_+.\\
\end{cases}
\end{equation*}
\end{proposition}

\begin{center}
\begin{figure}[h!]
\label{fig:eb-competitive}
\begin{tabular}{ccc}
\begin{tikzpicture}[scale=0.3]
\draw[->,>=latex] (0,0) -- (13,0);
\draw[-,>=latex,dashed] (0,10) -- (10,10);
\draw[-,>=latex,dashed] (0,10) -- (5,5);
\draw[->,>=latex] (0,0) -- (0,13);
\draw (0,12) node[left]{$\rho_+$};
\draw (12,0) node[below]{$\rho_-$};
\draw (0,0) node[below]{$0$};
\draw (0,0) node[left]{$0$};
\draw (10,0) node[below]{$1$};
\draw (0,10) node[left]{$1$};
\fill (0,0) -- (10,0) -- (10,10) -- cycle;
\draw (1.25, 5) node[right] {${\underline A}$};
\draw (6.25, 7.5) node[left] {${\underline B}$};
\draw (14, 12) node[left] {${ E^-_{+\infty}}$};
\end{tikzpicture}

&
\phantom{aaa}
&
\begin{tikzpicture}[scale=0.3]
\draw[->,>=latex] (0,0) -- (13,0);
\draw[-,>=latex,dashed] (0,10) -- (10,10);
\draw[-,>=latex,dashed] (0,5) -- (5,5) -- (5,10);
\draw[-,>=latex,dashed] (0,10) -- (5,5);
\draw[->,>=latex] (0,0) -- (0,13);
\draw (0,12) node[left]{$\rho_+$};
\draw (12,0) node[below]{$\rho_-$};
\draw (0,0) node[below]{$0$};
\draw (0,0) node[left]{$0$};
\draw (10,0) node[below]{$1$};
\draw (0,10) node[left]{$1$};
\fill (0,0) -- (10,0) -- (10,10) -- cycle;
\draw (1.25, 3.75) node[right] {${\bar A}$};
\draw (1.25, 6.25) node[right] {${\bar B}$};
\draw (2.5, 8.75) node[right] {${\bar C}$};
\draw (6.25, 8.75) node[right] {${\bar D}$};
\draw (11, 12) node[right] {${E^+_{+\infty}}$};
\end{tikzpicture}

\end{tabular}
\caption{The phase diagram for the bottom of the energy band (left) and the top of the energy band (right) for the competitive TASEP. We have ${\underline A}=-\log (1-\rho_-)$, $\underline B= -\log (\rho_+)$ and $\bar A = -\log (\chi(\rho_-)) +\log (1-\rho_+)$, $\bar B= \frac{\varphi_-}{\varphi_+ -\varphi_-} \log \left( \frac{1-\rho_-}{1-\rho_+}\right) -\log \rho_-$, $\bar C=  \frac{\varphi_+}{\varphi_+ -\varphi_-} \log \left( \frac{1-\rho_-}{1-\rho_+}\right) -\log \rho_+$, $\bar D= -\log (\chi(\rho_+)) +\log (\rho_-)$. All transitions are of first order.}
\label{fig:ebcomp}
\end{figure}
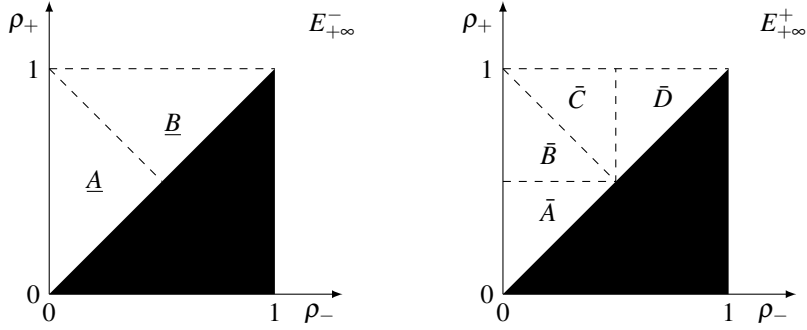
\end{center}

\subsection{Entropy}

Now we compute the entropy function. We introduce
\begin{equation}\label{W function dep rho}
W(\rho_-, \rho_+) =\cfrac{\log(\rho_+)\log(1-\rho_-)-\log(\rho_-)\log(1-\rho_+) }{\log\Big(\frac {\rho_+(1-\rho_-)}{\rho_-(1-\rho_+)}\Big)}
\end{equation}
which corresponds to
\begin{equation} \label{W function}
W(\varphi_-, \varphi_+) =\cfrac{\varphi_- \log (1+e^{\varphi_+}) -\varphi_+ \log (1+e^{\varphi_-}) }{\varphi_+ - \varphi_-}
\end{equation}
and coincides with the first coordinate of one of the (possible) two intersection points of the curves $x \to S_{\rho_+,\rho_+} (-x)$ and $x \to S_{\rho_-,\rho_-} (-x)$, where $S_{\rho,\rho}(\cdot)$ is defined in \eqref{eq:entropysimple}.\\

\begin{theorem}
\label{th:ent1}
The restriction of the entropy function $S^+$ on the energy band $[E_{+\infty}^- \, ; \, E_{+\infty}^+]$ is given by
\begin{equation*}
S^+ (E) =
\begin{cases}
&S_{\rho_0, \rho_0} (-(E+{\bar V}^+)), \quad   \rho_-\leq{\frac{1}{2}}\leq \rho_+,\\
&\\
&S_{\rho_-, \rho_-} (-(E+{\bar V}^+)) \, {\bf 1}\{(E+{\bar V}^+) \le W(\rho_-, \rho_+)\} \\
&+ S_{\rho_+, \rho_+} (-(E+{\bar V}^+)) \,  {\bf 1}\{(E+{\bar V}^+) > W(\rho_-, \rho_+)\}, \quad \rho_- < \rho_+ \leq \frac{1}{2},\\
&\\
&S_{\rho_-, \rho_-} (-(E+{\bar V}^+)) \, {\bf 1}\{(E+{\bar V}^+) \ge W(\rho_-, \rho_+)\}  \\
 &+ S_{\rho_+, \rho_+} (-(E+{\bar V}^+)) \,  {\bf 1}\{(E+{\bar V}^+) < W(\rho_-, \rho_+)\}, \quad \frac{1}{2}\leq  \rho_- < \rho_+,\\
\end{cases}
\end{equation*}
and is a concave function.  Therefore, when $\rho_-\leq{\frac{1}{2}}\leq \rho_+$, its derivative $(S^+)'$ is continuous on the energy band, but in the remaining cases $(S^+)'$ is  continuous except where $E+{\bar V}^+= W(\rho_-, \rho_+)$.

The supremum in the definition of $S^+(E)$, see \eqref{eq:1.9}, for $E\in [E_{+\infty}^- \, ; \, E_{+\infty}^+]$ is realized for a unique profile whose value is given by
\begin{equation*}
\begin{cases}
&u_{\bar{\rho}}, \quad   \rho_-\leq{\frac{1}{2}}\leq \rho_+,\\
&\\
&u_{\rho_-} \, {\bf 1}\{(E+{\bar V}^+) \le W(\rho_-, \rho_+)\} \\
&+ u_{\rho_+} \,  {\bf 1}\{(E+{\bar V}^+) > W(\rho_-, \rho_+)\}, \quad \rho_- < \rho_+ \leq \frac{1}{2},\\
&\\
&u_{\rho_-} \, {\bf 1}\{(E+{\bar V}^+) \ge W(\rho_-, \rho_+)\}  \\
 &+ u_{\rho_+}  {\bf 1}\{(E+{\bar V}^+) < W(\rho_-, \rho_+)\}, \quad \frac{1}{2}\leq  \rho_- < \rho_+,\\
\end{cases}
\end{equation*}
where for any $\rho$, the profile $u_\rho$ is the constant profile equal  to $\frac{\log(\rho)-(E+\bar {V}^+)}{\log(\rho)-\log(1-\rho)}$.
\end{theorem}

\begin{center}
\begin{figure}[h!]
\begin{tabular}{cc}
 {\includegraphics[scale=0.4]{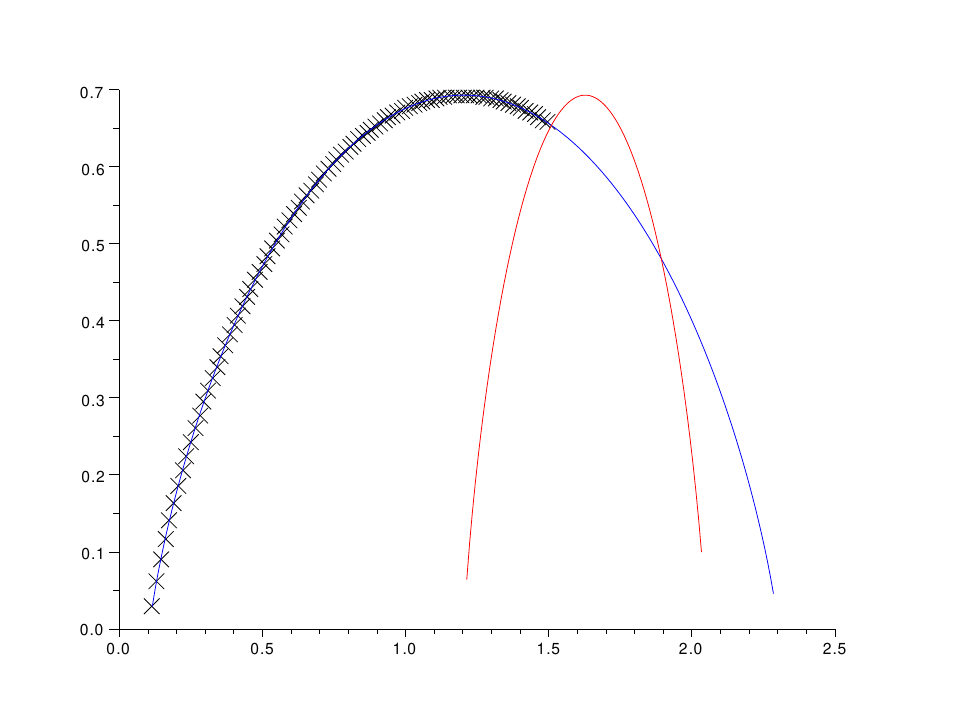}}
&
{ \includegraphics[scale=0.4]{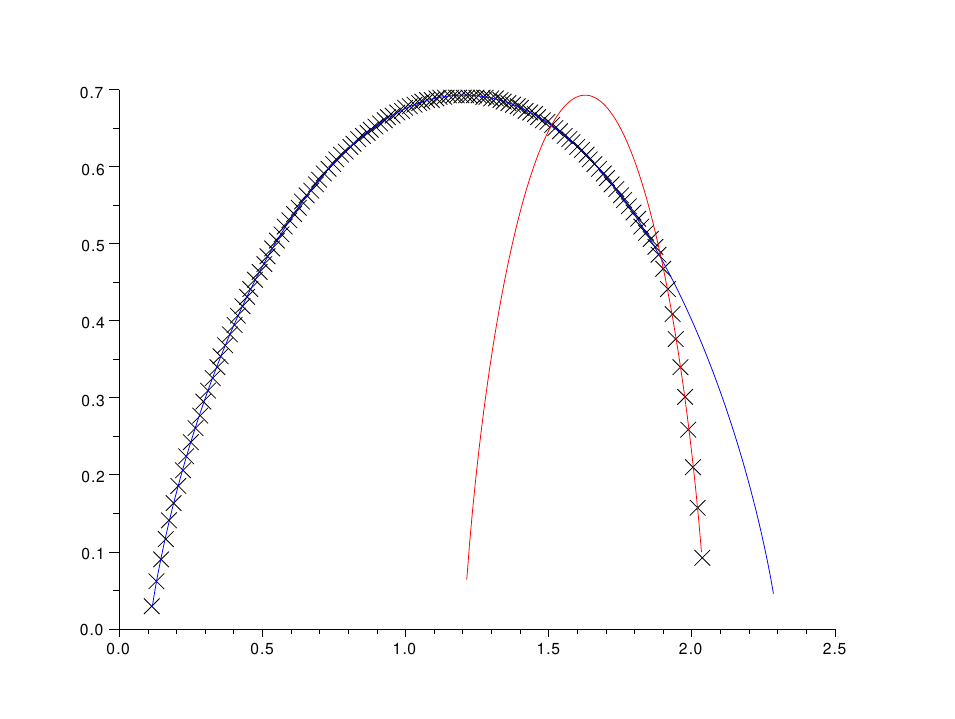}}
\end{tabular}
\caption{Graph of the function $S^+$ (crosses) and graphs of the functions $E \to S_{\rho_\pm, \rho_\pm} (-(E+{\bar V}^+))$ (red and blue) for $\rho_-=0.1$, $\rho_+=0.7$ (left) and for $\rho_-=0.1$, $\rho_+=0.3$ (right). The graph of the  function $S^+$ for $1/2\leq \rho_-<\rho_+$ is similar to the one at the right hand side of the previous figure, since the entropy function in those cases has the same expression when exchanging $\rho_-$ with $\rho_+$.}
\end{figure}
\end{center}

\subsection{Pressure}

We recall that the pressure function $P^+$ is defined as the Legendre transform of the entropy function $S^+$:
\begin{equation*}
P^+ (\theta) = \inf_{E \ge 0} \left\{ \theta E -S^+ (E) \right\}.
\end{equation*}

We introduce the two parameters $\theta_0^\pm:= -\frac{1}{\varphi_\pm} s' (\xi_0)$, where $\xi_0$ is defined in \eqref{chi0}.

\begin{theorem}
\label{th:pressure-competitive}
The pressure function $P^+$ is given by:
\begin{itemize}

\item If $\rho_+ < 1-\rho_-\le \frac{1}{2}$ then
\begin{equation*}
P^+ (\theta) =
\begin{cases}
\vspace{0.2cm}
P(\varphi_-, -\theta) -\theta {\bar V}^+, \quad \theta \ge \theta_0^-,\\
\vspace{0.2cm}
P^+ (\theta_0) + E_{+\infty}^+ (\theta- \theta_0^-), \quad \theta <\theta_0^-.
\end{cases}
\end{equation*}

\item If $\frac{1}{2} \le \rho_+ < 1-\rho_-$ then
\begin{equation*}
P^+ (\theta) =
\begin{cases}
\vspace{0.2cm}
P(\varphi_+, -\theta) -\theta {\bar V}^+, \quad \theta \ge \theta_0^+,\\
\vspace{0.2cm}
 P^+ (\theta_0) + E_{+\infty}^+ (\theta- \theta_0), \quad \theta <\theta_0^+.
\end{cases}
\end{equation*}

\item If $\rho_- < \rho_+ \le \frac{1}{2}$ then
\begin{equation*}
P^+ (\theta) =
\begin{cases}
\vspace{0.2cm}
P(\varphi_-, -\theta) -\theta {\bar V}^+, \quad \theta \ge \theta_0^-,\\
\vspace{0.2cm}
 P(\varphi_+, -\theta) -\theta {\bar V}^+,  \quad \theta \le \theta_0^+,\\
\vspace{0.2cm}
P^+ (\theta_0^+) + \cfrac{P^+(\theta_0^-) - P^+ (\theta_0^+)}{\theta_0^- - \theta_0^+}(\theta-\theta_0^+), \quad \theta \in (\theta_0^+, \theta_0^-).
\end{cases}
\end{equation*}
\item If $\frac{1}{2}\le \rho_- < \rho_+$ then
\begin{equation*}
P^+ (\theta) =
\begin{cases}
\vspace{0.2cm}
P(\varphi_+, -\theta) -\theta {\bar V}^+, \quad \theta \ge \theta_0^+,\\
 \vspace{0.2cm}
 P(\varphi_-, -\theta) -\theta {\bar V}^+,  \quad \theta \le \theta_0^-,\\
\vspace{0.2cm}
P^+ (\theta_0^-) + \cfrac{P^+(\theta_0^+) - P^+ (\theta_0^-)}{\theta_0^+ - \theta_0^-}(\theta-\theta_0^-), \quad \theta \in (\theta_0^-, \theta_0^+),
\end{cases}
\end{equation*}
\end{itemize}
where $P(\rho,\varphi)$ is given by \eqref{Pfunction}.

It follows that the function $P^+$ is a concave continuously differentiable function with some linear parts.
\end{theorem}

The proof of this theorem is postponed to Appendix \ref{sec:A3}.

\begin{center}
\begin{figure}[h!]
\begin{tabular}{cc}
 {\includegraphics[scale=0.4]{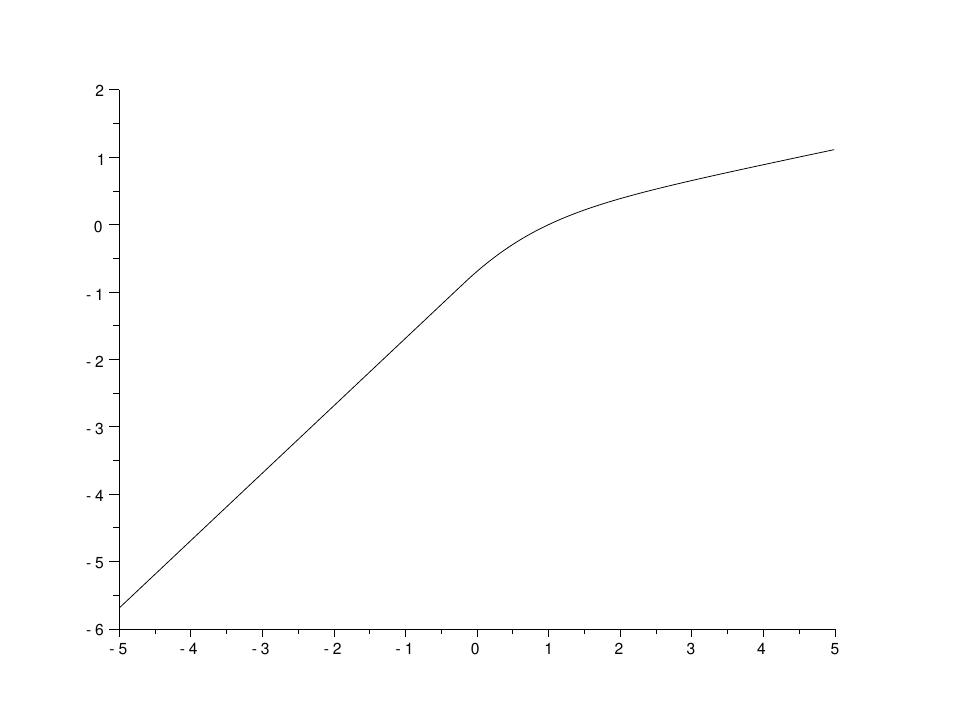}}
&
{ \includegraphics[scale=0.4]{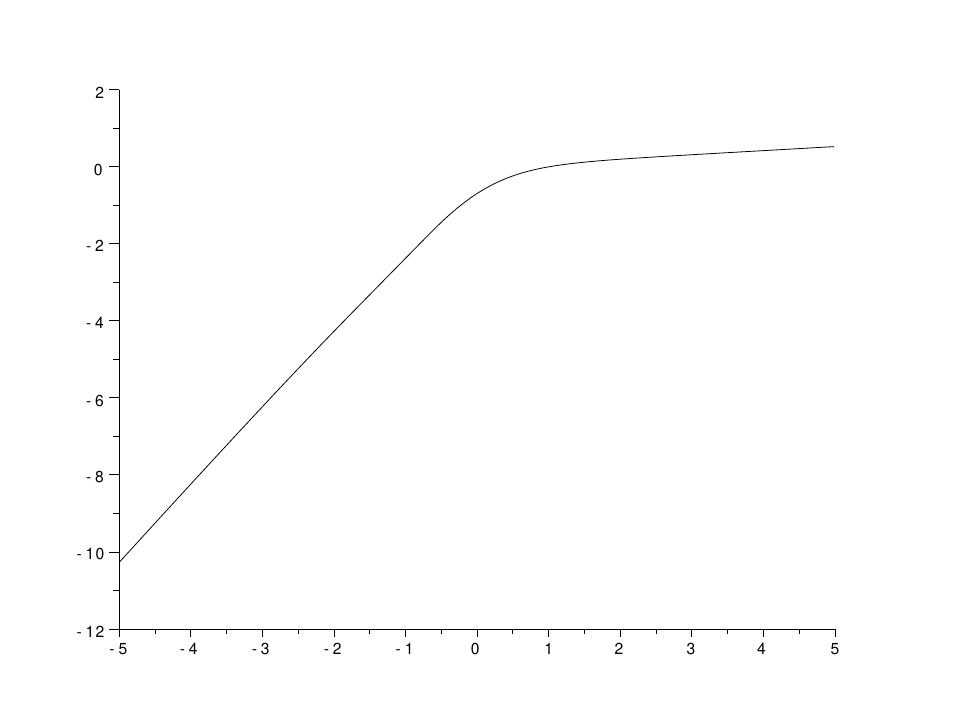}}
\end{tabular}
\caption{Graph of the function $P^+$ for $\rho_-=0.3$, $\rho_+=0.8$ (left) and for $\rho_-=0.1$, $\rho_+=0.3$ (right).}
\end{figure}
\end{center}

\subsection{Consequences}

Let $\rho \in (0,1)$ and $\varphi$ be the associated chemical potential, see \eqref{chemical potential}. Let us first observe that the equation of the tangent to the curve of $E \to S_{\rho,\rho} (E)$ at $E=E_0 (\rho):= -\rho \log (1-\rho) - (1-\rho) \log (\rho)$ is given by
\begin{equation*}
Y=-E -\log (\chi (\rho)).
\end{equation*}
This is the unique point where the tangent has a slope equal to $-1$. Since $S_{\rho, \rho}$ is a concave function, the curve of $S_{\rho,\rho}$ is strictly below the tangent apart from the point $(E_{0}, S_{\rho,\rho} (E_0))=(E_0, (-s) (\rho))$. Moreover, $E_0 (\rho) \ge -\frac{\varphi}{2} +\log(1+e^{\varphi})$, i.e. $E_0$ is to the right of the point where the function $S_{\rho,\rho}$ has its maximum.

This permits to show that $J^+ (E)= E-S^+ (E)$ is a non negative convex function which vanishes for a unique value of $E$ equal to  ${\bb S} ({\bar \rho})$. We recall that $J^+$ is the large deviations function of the random variables $\{Y_N \}_N$ under $\mu_{ss,N}$.

From this we recover the result of Bahadoran (\cite{B}) in the case of the TASEP. We also extend some of the results of \cite{DLS} to the asymmetric simple exclusion process.

\begin{theorem}
\label{th:cons1}
In the thermodynamic limit, the Gibbs-Shanonn entropy of the non-equilibrium stationary state defined by
\begin{equation*}
{\mc S} (\mu_{ss,N}) =  \sum_{\eta \in \Omega_N} \left[ - \mu_{ss, N} (\eta) \log (\mu_{ss,N} (\eta)) \right]
\end{equation*}
is equal to the Gibbs-Shanonn entropy of the local equilibrium state, i.e.
\begin{equation*}
\lim_{N \to +\infty} \cfrac{{\mc S} (\mu_{ss,N})}{2N+1} = {\bb S} ({\bar \rho}).
\end{equation*}
Moreover, the corresponding fluctuations are Gaussian with a variance $\sigma$ equal to the one provided by a local equilibrium statement, i.e. $$\sigma=S_{{\bar \rho},{\bar \rho}}^{''} ({\bb S} ({\bar\rho})) = \frac{1}{2} \int_{-1}^1 \, \chi ({\bar \rho}) \left[ (-s)' ({\bar \rho}(u)) \right]^2\, du.$$
\end{theorem}

\section{Cooperative TASEP}
\label{sec:coopT}

In this section we present the main results of the article in the case of the cooperative TASEP. We start by deriving the variational formula for the entropy function \eqref{eq:1.9} for the cooperative TASEP.
Let ${\mc F}$ be the set
\begin{equation*}
{\mc F} := \left\{ \varphi \in C^{1} ([-1,1]) \; : \; \varphi (\pm1) =\varphi_{\pm}\, \,, \; \varphi' > 0 \right\}.
\end{equation*}
and recall that $\chi:[0,1] \to [0,1]$ represents the mobility and  is given by $\chi(\rho)=\rho (1-\rho)$.

The quasi-potential $V^{-}$ of the cooperative TASEP (\cite{DLS3}, \cite{B2}) is defined by
\begin{equation}
\label{eq:Vminus}
V^{-} (\rho)=  - {\bb S} (\rho) + \sup_{\varphi \in {\mc F}}  \mc H (\rho, \varphi) -{\bar V}^-
\end{equation}
where $\mathbb{S}(\cdot)$ is defined in (\ref{eq: S}), $\mc H$ is defined in (\ref{H function}) and
\begin{equation*}
{\bar V}^- =\log\Big(\max_{ \rho \in [\rho_-, \rho_+]} \chi (\rho)\Big).
\end{equation*}

\subsection{Energy bands}

In this section we determine the energy band of the cooperative TASEP. This is summarized in Figure \ref{fig:ebcoop}.

\begin{proposition}
\label{lem:3123}
The bottom  of the energy band is given by
\begin{equation*}
\begin{split}
&E^{-}_{-\infty}=-{\bar V}^-
+\begin{cases}
\vspace{2mm}
&-\log (2), \quad \rho_- \leq \frac{1}{2} \leq \rho_+,\\
\vspace{2mm}
&- \log (1+e^{-\varphi_+}), \quad \rho_- < \rho_+ \leq \frac{1}{2},\\
\vspace{2mm}
&-\log (1+e^{\varphi_-}), \quad \frac{1}{2}\leq \rho_- < \rho_+,
\end{cases}
\end{split}
\end{equation*}
and the top  of the energy band by is given by
\begin{equation*}
E^{+}_{-\infty}=-{\bar V}^- -\log (1+e^{-\varphi_0})
\end{equation*}
where $\varphi_0=\sup\{|\varphi_-|,|\varphi_+|\}$.

\end{proposition}

\begin{center}
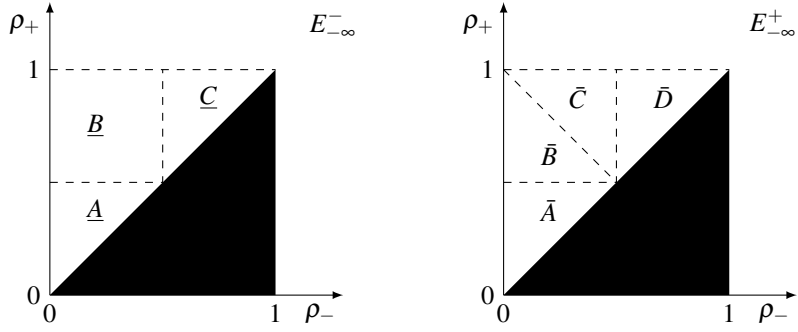
\begin{figure}[h!]
\label{fig:eb-cooperative}
\begin{tabular}{ccc}
\begin{tikzpicture}[scale=0.3]
\draw[->,>=latex] (0,0) -- (13,0);
\draw[-,>=latex,dashed] (0,10) -- (10,10);
\draw[-,>=latex,dashed] (0,5) -- (5,5) -- (5,10);
\draw[->,>=latex] (0,0) -- (0,13);
\draw (0,12) node[left]{$\rho_+$};
\draw (12,0) node[below]{$\rho_-$};
\draw (0,0) node[below]{$0$};
\draw (0,0) node[left]{$0$};
\draw (10,0) node[below]{$1$};
\draw (0,10) node[left]{$1$};
\fill (0,0) -- (10,0) -- (10,10) -- cycle;
\draw (1.25, 3.75) node[right] {${\underline A}$};
\draw (1.25, 7.5) node[right] {${\underline B}$};
\draw (6.25, 8.75) node[right] {${\underline C}$};
\draw (14, 12) node[left] {${E^-_{-\infty}}$};
\end{tikzpicture}
&
\phantom{aaa}
&
\begin{tikzpicture}[scale=0.3]
\draw[->,>=latex] (0,0) -- (13,0);
\draw[-,>=latex,dashed] (0,10) -- (10,10);
\draw[-,>=latex,dashed] (0,10) -- (5,5);
\draw[-,>=latex,dashed] (0,5) -- (5,5) -- (5,10);
\draw[->,>=latex] (0,0) -- (0,13);
\draw (0,12) node[left]{$\rho_+$};
\draw (12,0) node[below]{$\rho_-$};
\draw (0,0) node[below]{$0$};
\draw (0,0) node[left]{$0$};
\draw (10,0) node[below]{$1$};
\draw (0,10) node[left]{$1$};
\fill (0,0) -- (10,0) -- (10,10) -- cycle;
\draw (1.25, 3.75) node[right] {${\bar A}$};
\draw (1.25, 6.25) node[right] {${\bar B}$};
\draw (2.5, 8.75) node[right] {${\bar C}$};
\draw (6.25, 8.75) node[right] {${\bar D}$};
\draw (10.5, 12) node[right] {${E^+_{-\infty}}$};
\end{tikzpicture}
\end{tabular}
\caption{The phase diagram for the bottom of the energy band (left) and the top of the energy band (right) for the cooperative TASEP. We have ${\underline A}=-\log(1-\rho_+)$, $\underline B= \log (2)$, $\underline C = -\log(\rho_-)$ and $\bar A= -\log(\chi(\rho_+)) +\log (1-\rho_-)$, $\bar B = 2\log(2) + \log (1-\rho_-)$, $\bar C= 2\log (2) + \log (\rho_+)$, $\bar D = -\log (\chi(\rho_-)) +\log(\rho_+)$. All transitions are of first order.}
\label{fig:ebcoop}
\end{figure}
\end{center}

\subsection{Entropy} \label{sec entropy 2}

We are now in position to state the main result of this section.

\begin{theorem}
\label{th:ent2}
The restriction of the entropy function $S^-$ on the energy band $[E_{-\infty}^- \, ; \, E_{-\infty}^+]$ is given by
\begin{equation*}
S^- (E) =
\begin{cases}
&-(E+{\bar V}^-){\bf 1}\{(E+{\bar V}^-) \le s(\rho_0)\}\\
&+S_{\rho_0, \rho_0} (-(E+{\bar V}^-)){\bf 1}\{(E+{\bar V}^-) > s(\rho_0)\}, \quad \rho_- \leq \frac{1}{2}\leq  \rho_+,\\
&\\
&S_{\rho_+, \rho_+} (-(E+{\bar V}^-)){\bf 1}\{(E+{\bar V}^-) < s(\rho_+)\}\\
&-(E+{\bar V}^-){\bf 1}\{s(\rho_+)\leq (E+{\bar V}^-) \le s(\rho_-)\}\\
&+S_{\rho_-, \rho_-} (-(E+{\bar V}^-)) \, {\bf 1}\{(E+{\bar V}^-) > s(\rho_-)\}, \quad \rho_- < \rho_+ \leq \frac{1}{2},\\
&\\
&S_{\rho_-, \rho_-} (-(E+{\bar V}^-)) \, {\bf 1}\{(E+{\bar V}^+) < s(\rho_-)\} \\
&-(E+{\bar V}^-){\bf 1}\{s(\rho_-)\leq (E+{\bar V}^-) \le s(\rho_+)\}\\
&+ S_{\rho_+, \rho_+} (-(E+{\bar V}^-)) \,  {\bf 1}\{(E+{\bar V}^-) > s(\rho_+)\}, \quad \frac{1}{2}\leq  \rho_- < \rho_+.
\end{cases}
\end{equation*}
Moreover, the function $S^-$ is concave, its derivative $(S^-)'$ is continuous on the energy band but its second derivative $(S^-)''$ is not continuous.

The supremum in the definition of $S^-(E)$, see \eqref{eq:1.9}, for $E\in [E_{-\infty}^- \, ; \, E_{-\infty}^+]$ is realized by the profiles:
\begin{equation*}
\begin{cases}
\mathcal{M}^E_-,& \quad \textrm{if}\quad E+{\bar V}^- \le s(\rho_0) \\
\{u_{\rho_0}\},& \quad \textrm{if}\quad E+{\bar V}^- > s(\rho_0), \quad \quad \quad \quad \quad  \quad\rho_- \leq \frac{1}{2}\leq  \rho_+, \\
&\\
\{u_{\rho_+}\},& \quad \textrm{if}\quad E+{\bar V}^-< s(\rho_+),\\
\mathcal{M}^E_-,& \quad \textrm{if}\quad s(\rho_+)\leq {E+{\bar V}^-} \leq{ s(\rho_-)}, \\
\{u_{\rho_-}\},& \quad \textrm{if}\quad E+{\bar V}^- > s(\rho_-), \quad \quad \quad \quad \quad  \quad\rho_- < \rho_+ \leq \frac{1}{2},\\
&\\
\{u_{\rho_-}\},& \quad \textrm{if}\quad E+{\bar V}^+ < s(\rho_-), \\
\mathcal{M}^E_-,& \quad \textrm{if}\quad  s(\rho_-)\leq E+{\bar V}^- \le s(\rho_+),\\
\{u_{\rho_+}\},& \quad \textrm{if}\quad E+{\bar V}^- > s(\rho_+)\}, \quad \quad \quad \quad \quad  \quad\frac{1}{2}\leq  \rho_- < \rho_+, \\
\end{cases}
\end{equation*}
where for any $\rho$, the profile $u_\rho$ is the constant profile  equal  to $\frac{\log(\rho)-(E+\bar {V}^-)}{\log(\rho)-\log(1-\rho)}$ and  $\mathcal{M}^E_-$ is the set of non-increasing profiles $\rho:[-1,1]\to{[1-\rho_+,1-\rho_-]}$ such that ${\bb S} (\rho)=-(E+\bar{V}^-)$.

\end{theorem}

\begin{center}
\begin{figure}[h!]
\begin{tabular}{cc}
{ \includegraphics[scale=0.4]{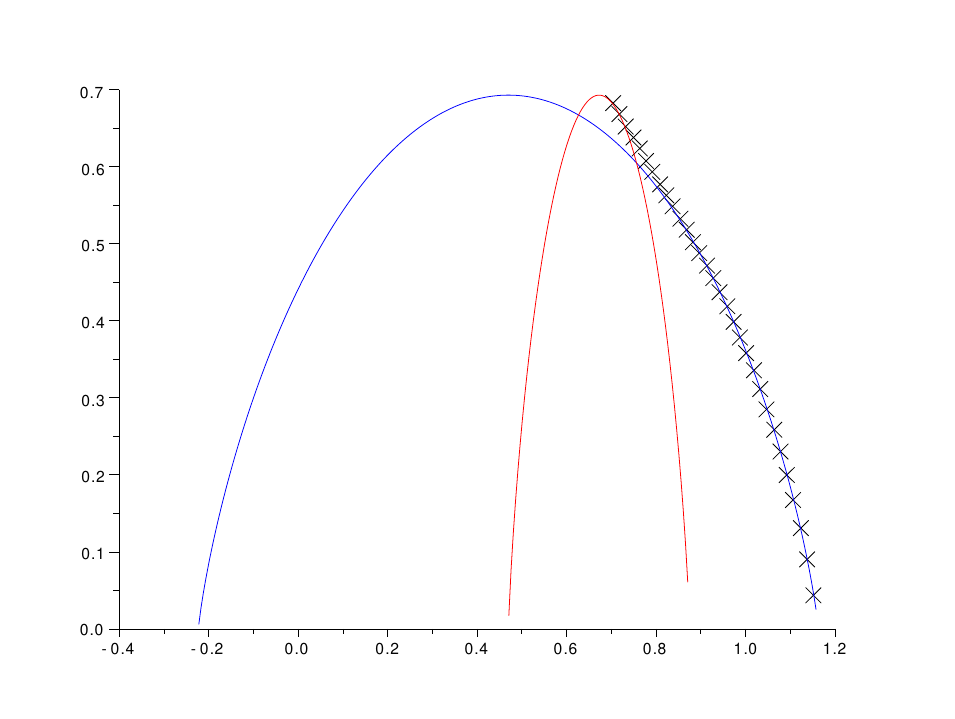}}
&
 {\includegraphics[scale=0.4]{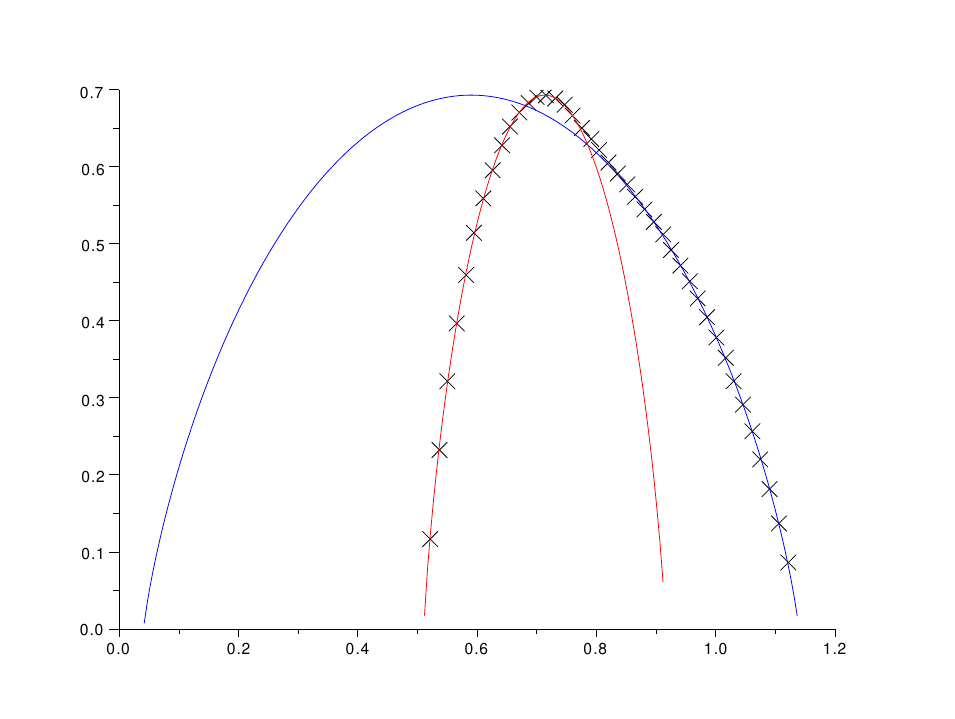}}
\end{tabular}
\caption{Graph of the function $S^-$ (cross) and graphs of the functions $E \to S_{\rho_\pm, \rho_\pm} (-(E+{\bar V}^-)$ (red and blue) for $\rho_-=0.2$, $\rho_+=0.6$ (left) and for $\rho_-=0.25$, $\rho_+=0.4$ (right).}
\end{figure}
\end{center}

\subsection{Pressure}

Let us define the function $m: \RR \to (0,+\infty)$ by
\[
m(\theta)= \inf \{f_{\theta}(t): {t\in [\varphi_-, \varphi_+]}\}, \quad \theta \in \RR,
\]
where for any $\theta \in \RR$, $f_{\theta}: \RR \to (0, +\infty)$ is given by
\begin{equation*}
f_{\theta} (t)=
\begin{cases}
\vspace{1mm}
\exp \Big\{ \cfrac{1}{\theta} \log (1+e^{\theta t} ) +\log (1+e^{-t})\Big\}, \quad \theta \ne 0,\\
\vspace{1mm}
\exp \left\{ t/2 +\log (1+e^{-t})\right\}, \quad \theta=0.
\end{cases}
\end{equation*}

It is easy to check that $f_{\theta}$ is increasing (resp. decreasing)  on $(-\infty,0)$, decreasing (resp. increasing) on $(0,+\infty)$ if $\theta<-1$ (resp. $\theta >-1$) and is constant equal to $1$ if $\theta=-1$. It follows that
\begin{equation*}
m(\theta)=
\begin{cases}
\vspace{1mm}
\min (f_{\theta} (\varphi_-), f_{\theta} (\varphi_+)), \quad 0 \notin [\varphi_-, \varphi_+]\\
\vspace{1mm}
\min (f_{\theta} (\varphi_-), f_{\theta} (\varphi_+), f_{\theta} (0)), \quad 0 \in [\varphi_-\; ; \; \varphi_+]
\end{cases}.
\end{equation*}

\begin{theorem}
\label{th:pressure-cooperative}
The pressure function $P^-$ is given by
$$P^{-} (\theta) = -\theta\,  (\log (m(\theta)) +{\bar V}^{-}). $$
Moreover, the function $P^-$ is concave, has a linear part if $\rho_- \le 1/2 \le \rho_+$, and is differentiable except for $\theta=-1$.
\end{theorem}

The proof of this theorem is postponed to Appendix \ref{sec:A4}.

\begin{center}
\begin{figure}[h!]
\begin{tabular}{cc}
 {\includegraphics[scale=0.4]{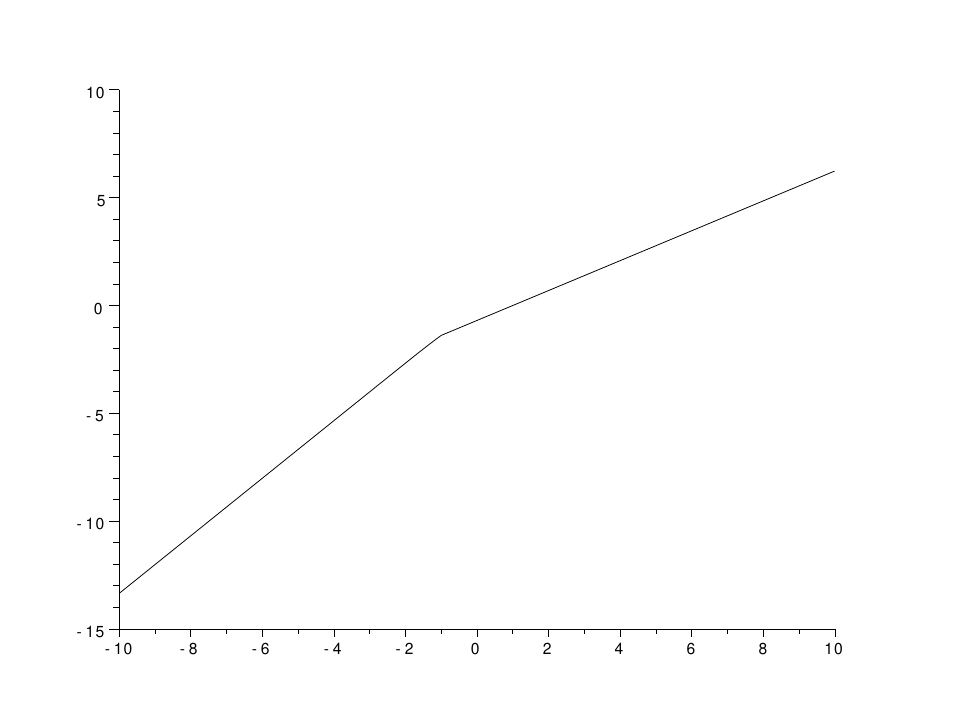}}
&
{ \includegraphics[scale=0.4]{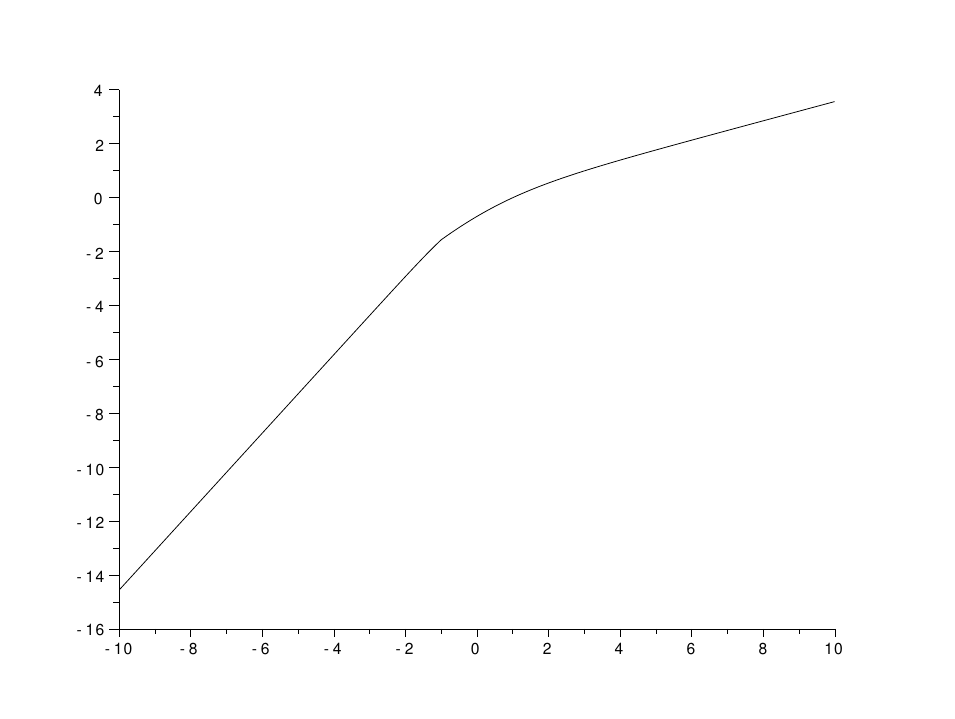}}
\end{tabular}
\caption{Graph of the function $P^-$ for $\rho_-=0.15$, $\rho_+=0.95$ (left) and for $\rho_-=0.1$, $\rho_+=0.3$ (right).}
\end{figure}
\end{center}

\subsection{Consequences}
\begin{theorem}
\label{th:cons2}
In the thermodynamic limit, the Gibbs-Shanonn entropy of the non-equilibrium stationary state defined by
\begin{equation*}
{\mc S} (\mu_{ss,N}) =  \sum_{\eta \in \Omega_N} \left[ - \mu_{ss, N} (\eta) \log (\mu_{ss,N} (\eta)) \right]
\end{equation*}
is equal to the Gibbs-Shanonn entropy of the local equilibrium state, i.e.
\begin{equation*}
\lim_{N \to +\infty} \cfrac{{\mc S} (\mu_{ss,N})}{2N+1} = {\bb S} ({\bar \rho}).
\end{equation*}
Moreover, in the case $\rho_+>\rho_- \ge 1/2$ or $\rho<\rho_+ \le1/2$, the fluctuations are Gaussian with a variance $\sigma=\chi(\bar \rho)[(-s)'(\bar \rho)]^2$. In the case $\rho_-\leq1/2\leq \rho_+$ the fluctuations are not Gaussian.
\end{theorem}
\section{Proofs}
\subsection{Proof of Proposition \ref{EBcompetitive}.}

\quad

Recall from \eqref{energy band} that $E^{-}_{+\infty}= \inf_{\rho \in {\mc M}} \left\{ {\bb S} (\rho) +V^{+} (\rho) \right\}$. The computation of the bottom of the energy band is easy since we have
\begin{equation*}
\begin{split}
E^{-}_{+\infty}&=\inf_{\rho \in {\mc M}} \inf_{\varphi \in \Phi} {\mc H} (\rho, \varphi) -{\bar V}^+ \\&=  \inf_{\varphi \in \Phi} \inf_{\rho \in {\mc M}}{\mc H} (\rho, \varphi) -{\bar V}^+\\
&= \inf_{y \in [-1,1]} \frac{1}{2}\Big\{ (y+1) \left( \varphi_- \wedge 0 -\log(1+e^{\varphi_-}) \right) + (1-y) \left( \varphi_+ \wedge 0 -\log(1+e^{\varphi_+}) \right)\Big\} -{\bar V}^+\\
&= \inf \Big\{  \left( \varphi_- \wedge 0 -\log(1+e^{\varphi_-}) \right) \, ,\, \left( \varphi_+ \wedge 0 -\log(1+e^{\varphi_+}) \right) \Big\} -{\bar V}^+.
\end{split}
\end{equation*}

We  compute now the top of the energy band. Recall from \eqref{energy band} that $$E^{+}_{+\infty}=  \sup_{\rho \in {\mc M}} \left\{ {\bb S} (\rho) +V^{+} (\rho) \right\}.$$ For each profile $\rho \in {\mc M}$, we introduce the non-decreasing continuous and almost everywhere differentiable function
\begin{equation}\label{Hfunction}
H_{\rho} (x) = \int_{-1}^x (1-\rho (z)) dz
 \end{equation}
and the constant $m_{\rho}= H_{\rho} (1) \in [0,2]$. Let $y_{\rho} \in [-1,1]$ the infimum of the points of $[-1,1]$ where the infimum of the continuous function $y \to y \xi_0 -H_{\rho}(y)$ is attained. Then,
\begin{equation*}
\begin{split}
E_{+\infty}^{+}&=\sup_{\rho \in {\mc M}} \inf_{\varphi \in \Phi} {\mc H}(\rho, \varphi) -{\bar V}^+ \\
&= \frac{(\varphi_+ -\varphi_-) }{2} \sup_{m \in [0,2]} \left\{ \cfrac{\varphi_+}{\varphi_+ -\varphi_-} m - \cfrac{{2\bar V}^+}{\varphi_+ -\varphi_-} -{\hat \xi}_0 + {\bb U} (m) \right\}
\end{split}
\end{equation*}
with
$${\bb U}(m) = \sup_{\substack{\rho \in {\mc M},\\ H_{\rho} (1) =m}} \inf_{y \in [-1,1]} \left\{ y \xi_0 -H_{\rho} (y) \right\}.$$
We claim that ${\bb U} (m) =\inf\{ -\xi_0\, , \, \xi_0-m\}.$ It is trivial that ${\bb U} (m) \le \inf\{ -\xi_0, \xi_0-m\}$ (take $y=1,-1$ in the variational formula). For $m \ne \xi_0$, the supremum can be obtained by taking the piecewise linear function $H_{\rho}$ such that $H_{\rho}$ is linear on $[-1,0]$ and on $[0,1]$, with $H_{\rho} (-1)=0$, $H_{\rho} (1)=m$ and
\begin{equation*}
H_{\rho} (0) =
\begin{cases}
\vspace{0.1cm}
\xi_0, \quad - \xi_0 \le \xi_0 -m, \quad m > \xi_0,\\
\vspace{0.1cm}
0, \quad -\xi_0 \le \xi_0 -m, \quad m < \xi_0,\\
\vspace{0.1cm}
1, \quad -\xi_0 \ge \xi_0 -m, \quad m-\xi_0 \ge 1,\\
\vspace{0.1cm}
m-\xi_0, \quad  -\xi_0 \ge \xi_0 -m, \quad m-\xi_0 \le 1.
\end{cases}
\end{equation*}
In the case $m=\xi_0$, there are two profiles $H_{\rho}$ for which the supremum is obtained, one with $H_{\rho} (0)= \xi_0$, the other one with $H_{\rho} (0)=0$. It follows that
\begin{equation*}
\begin{split}
E_{+\infty}^{+} &= \frac{(\varphi_+ -\varphi_-)}{2} \sup_{m \in [0,2]} \left\{ \cfrac{\varphi_+}{\varphi_+ -\varphi_-} m - \cfrac{{2\bar V}^+}{\varphi_+ -\varphi_-} -{\hat \xi}_0 + \inf\{ \xi_0 -m, -\xi_0\}\right\}\\
&= \frac{(\varphi_+ -\varphi_-)}{2} \left[ \sup_{m \in [2\xi_0, 2]} \left\{ \cfrac{\varphi_-}{\varphi_+ - {\varphi_-}} m +(\xi_0 -{\hat \xi}_0) \right\} \bigvee \sup_{m \in [0,2\xi_0]} \left\{ \cfrac{\varphi_+}{\varphi_+ - {\varphi_-}} m -(\xi_0 +{\hat \xi}_0) \right\} \right]\\
&\quad -{\bar V}^+.
\end{split}
\end{equation*}
We have now to optimize a piecewise linear function and we get the result.

\subsection{Proof of Theorem \ref{th:ent1}.}

\quad

We notice that once the form of $S^+$ is obtained its concavity is easy to establish. The computation of ${S}^+$ is accomplished in several steps. For any $E \ge 0$, let $D:=D(E)$ be the  (possibly empty) compact convex domain of $\RR^2$ defined by
\begin{equation}
\label{eq:condm}
(y,m)\in D \Leftrightarrow \begin{cases}
\vspace{1mm}
&-1 \le y \le 1, \quad 0 \le m \le 2,\\
\vspace{1mm}
&\sup\{\xi_0, m-\xi_0\} \le E(m) \le m- (m-1)\xi_0,\\
\vspace{1mm}
&0 \le {y  \xi_0 + E(m) }\le {y +1}, \quad 0 \le {m-(y  \xi_0 +E(m)) } \le {1-y},
\end{cases}
\end{equation}
where
\begin{equation*}
E(m) := \cfrac{\varphi_+ m -2{\bar V}^+ -2E}{\varphi_+ -\varphi_-} -{\hat \xi}_0.
\end{equation*}
The fact that $D$ is a convex compact domain follows from the fact that $E(m)$ is a linear function of $m$ so that $D$ is the intersection of half-planes of $\RR^2$.\\

Define now the function $F:{D} \to \RR$  on $(y,m) \in D$ by
\begin{equation*}
F (y,m) = - (y +1) s\left( \cfrac{y\,  \xi_0 + E(m)}{y +1}\right) - (1-y )s\left( \cfrac{m- (y \, \xi_0 + E(m))}{1-y}\right).
\end{equation*}
 It is understood here that if $y=\pm 1$ then the indefinite terms are equal to $0$. The function $F$ is continuous on $D$ and smooth on $\mathring{D}$.

\begin{proposition}
\label{prop:ent1}
The entropy function ${S}^+$ is given by
\begin{equation*}
{S}^{+} (E) = \frac{1}{2} \; \sup_{(y,m) \in D } F (y,m).
\end{equation*}
\end{proposition}

\begin{proof}
With the notations introduced above, we have
\begin{equation*}
{S}^+ (E) = \sup_{m \in [0,2]} \sup_{\rho \in {\mc M}} \left\{ {\bb S} (\rho) \; ; \; m_\rho=m, \, \inf_{y \in [-1, 1]} \left\{ y\xi_0 -H_{\rho} (y) \right\} = -E(m)\right\}.
\end{equation*}

Assume that there exists a profile $\rho_m \in {\mc M}$ such that  $\inf_{y \in [-1, 1]} \left\{ y\xi_0 -H_{\rho_m} (y) \right\} = -E(m)$ and $ H_{\rho_m} (1)=m$. Then, by taking $y=\pm 1$ in the infimum we see that this is possible only if $-E(m) \le \inf\{-\xi_0, \xi_0 -m\}=-\sup\{\xi_0, m-\xi_0\}$. Moreover, the existence of $\rho_m$ implies that
$$-E(m) = \inf_{y \in [-1, 1]} \left\{ y\xi_0 -H_{\rho_m} (y) \right\} \ge \inf_{\substack{\rho \in {\mc M},\\ H_{\rho} (1) =m}} \inf_{y \in [-1, 1]} \left\{ y\xi_0 -H_{\rho} (y) \right\}.$$

By inverting the two infimums, the right hand side of the previous inequality can  be rewritten as
$$\inf_{y \in [-1, 1]} \Big\{y \xi_0 -\sup_{\substack{\rho \in {\mc M},\\ H_{\rho} (1) =m}} H_{\rho} (y) \Big\}.$$
Since $0 \le H^{\prime}_\rho \le 1$, we have that $\sup_{\substack{\rho \in {\mc M},\\ H_{\rho} (1) =m}} H_{\rho} (y)$ is equal to $y+1$ if $y+1 \le m$ or $m$ if $y+1 \ge m$. It follows that
$$\inf_{\substack{\rho \in {\mc M},\\ H_{\rho} (1) =m}} \inf_{y \in [-1, 1]} \left\{ y\xi_0 -H_{\rho} (y) \right\} =(m-1)\xi_0 -m.$$ Thus the existence of $\rho_m$ is only possible if $\sup\{\xi_0, m-\xi_0\} \le E(m) \le m- (m-1)\xi_0.$


Let us denote $\rho_m$ by $\rho$. Recall that $y_{\rho}$ is the smallest point in $[-1,1]$ such that the infimum in $\inf_{y\in[-1,1]} \{y\xi_0 -H_{\rho} (y) \}$ is realized. If $y_\rho \in (-1,1)$ then we have $y_{\rho} \xi_0 -H_{\rho} (y_\rho) =-E(m)$ which implies that
\begin{equation*}
\begin{split}
&\cfrac{y_{\rho} \xi_0 + E(m) }{y_{\rho} +1}= \cfrac{H_{\rho} (y_\rho) -H_{\rho} (-1) }{y_{\rho} +1} \, \in [0,1],\\
& \cfrac{m-(y_{\rho} \xi_0 +E(m)) }{1-y_\rho} = \cfrac{H_{\rho} (1) -H_{\rho} (y_\rho)}{1-y_\rho}\,  \in [0,1].
\end{split}
\end{equation*}
If $y_{\rho} =-1$, then we have $\xi_0 =E(m)$ and for any $z\ge-1$, $H_{\rho} (z) \le z \xi_0 +E(m)$ and in particular, for $z=1$, we get $m \le 2\xi_0$. If $y_{\rho} =1$, similarly, we have $\xi_0 +E (m) =m$ and $2\xi_0 \le m$.\\

Thus, if a profile $\rho$ is such that $\bar{V}^+(\rho)+{\bb S}(\rho)=E$ then $(y_\rho,m_\rho)$ belongs to the set composed of couples
 $(y,m) \in [-1,1] \times [0,2]$ satisfying

\begin{equation*}
\label{eq:condm1}
\begin{cases}
\vspace{2mm}
&\sup\{\xi_0, m-\xi_0\} \le E(m) \le m- (m-1)\xi_0,\\
\vspace{2mm}
& \cfrac{y  \xi_0 + E(m) }{y +1} \in [0,1], \quad  \cfrac{m-(y  \xi_0 +E(m)) }{1-y} \in [0,1], \quad {\text{ if }} y \in (-1,1),\\
\vspace{2mm}
&-\xi_0 + E(m) = 0, \quad m \le 2\xi_0, \quad {\text{if }} y=-1,\\
\vspace{2mm}
&\xi_0 + E(m) =m, \quad m \ge 2\xi_0, \quad {\text{if }} y=1.
\end{cases}
\end{equation*}\\

These conditions are equivalent to $(y,m) \in D$. By concavity of the function $-s$ together with Jensen's inequality, if $y_{\rho} \in [-1,1]$, we have
 \begin{equation}
 \label{eq:in123}
 \begin{split}
 {\bb S} (\rho)={\bb S} (H_{\rho}^{\prime}) &= - \frac{1}{2} \int_{-1}^{y_\rho} s(H_{\rho}^{\prime}(x)) dx - \frac{1}{2} \int_{y_\rho}^1 s(H_{\rho}^{\prime}(x))dx\\
 &\le - \frac{(y_{\rho} +1)}{2} s\left( \cfrac{y_{\rho} \, \xi_0 + E(m)}{y_{\rho} +1}\right) - \frac{(1-y_{\rho})}{2} s\left( \cfrac{m- (y_{\rho}\,  \xi_0 + E(m))}{1-y_{\rho}}\right)
 \end{split}
 \end{equation}
 with the convention that if $y_{\rho}=\pm 1$ the indefinite terms have to be replaced by $0$. This proves that
 \begin{equation*}
 {S}^+ (E) \le \frac{1}{2} \; \sup_{(y,m) \in D} \; F(y,m).
 \end{equation*}

To prove the opposite inequality, consider any $(y,m) \in D$ and let $H^y$ be the continuous function, linear on $[-1,y]$ and on $[y, 1]$ such that $H^y (-1)=0, H^{y} (1)=m, H^{y} (y) =y\xi_0 + E(m)$. Since $(y,m) \in D$, the profile $\rho$ such that $H_{\rho} =H^y$ belongs to ${\mc M}$ and satisfies $H_{\rho} (1)=m$, $y_{\rho} =y$ and  $\inf_{z} \{z \xi_0 -H_{\rho} (z)  \}=-E(m)$, i.e. ${\bar V}^+(\rho)+{\bb S}(\rho)=E$. Observe now that the equality in (\ref{eq:in123}) is satisfied for $H_{\rho}:=H^{y_{\rho}}$. This shows that $F(y,m) =2 {\bb S }(\rho) \le 2 {S}^+ (E)$ and finishes the proof.

\end{proof}

To prove Theorem \ref{th:ent1} it remains to compute the supremum appearing in the statement of Proposition \ref{prop:ent1}. This is done in Appendix \ref{sec:A1}. The last part of Theorem \ref{th:ent1} concerning the values of the maximizers is also postponed to Appendix \ref{sec:A1}.

\subsection{Proof of Theorem \ref{th:cons1}.}

\quad

We just give the proof in the case $1/2<1-\rho_-<\rho_+$ which corresponds to $0 < -\varphi_- < \varphi_+$ since the other cases are similar. Then, we have $\bar \rho =\rho_+$ and ${\bar V}^+ = \log ( \rho_+ (1-\rho_+))$. Since, in the energy band, $J^+ (E) = E+{\bar V}^+ - S_{\rho_+, \rho_+} (-(E+{\bar V}^+)) -{\bar V}^+$, we conclude that $J^+ (E) \ge 0$ with equality if and only if $-(E+{\bar V}^+)=E_0 (\rho_+)$ and $-E_{0} (\rho_+) -{\bar V}^+ $ belongs to $[E_{+\infty}^- \, , \, E_{+\infty}^+]$. This last condition is equivalent to
\begin{equation}
\label{eq:e0000}
-E_{0} (\rho_+)  \in [-\log (1+e^{\varphi_+}) \, , \, W(\varphi_-,\varphi_+)].
\end{equation}
Since $-E_0 (\rho_+)= -(1-\rho_+ ) \log (1+e^{-\varphi_+})  - \rho_+ \log (1+e^{\varphi_+})$ we get easily that $-E_{0} (\rho_+)  \ge -\log (1+e^{\varphi_+})$. To prove the other inequality we write $W(\varphi_-, \varphi_+) + E_{0} (\rho_+) = g(\varphi_-)$,
where the function $g:[-\varphi_+, 0] \to \RR$ is defined by
\begin{equation*}
g(x)= (1-\rho_+) \log (1+e^{-\varphi_+ }) + \rho_+ \log (1+e^{\varphi_+}) + \cfrac{x \log (1+e^{\varphi_+}) -\varphi_+ \log (1+e^{x})}{\varphi_+ - x}.
\end{equation*}
 Since
 $$g' (x) = \frac{\varphi_+} {(\varphi_+ -x)^2} \left\{ \log (1+e^{\varphi_+}) + \log (1+e^x) - (\varphi_+ -x) \frac{e^{x}}{1+e^x}  \right\}$$
there exists ${\tilde \varphi}_x \in [x,\varphi_+]$ such that
$$g'(x)= \frac{\varphi_+} {(\varphi_+ -x) }\left\{ \frac{e^{{\tilde \varphi}_x}}{1+e^{{\tilde \varphi}_x}} - \frac{e^{x}}{1+e^x}  \right\} \ge 0.$$
The last inequality follows from the fact that the function $t \to e^t /(1+e^t)$ is increasing. Thus $g(\varphi_-) \le g(-\varphi_+)$ and $g(-\varphi_+)= (\rho_+ -1/2) \varphi_+ \ge 0$ which proves (\ref{eq:e0000}). Thus, $J^+$ is a non-negative convex function vanishing only for $-E_0 (\rho_+) -{\bar V}^+= (-s)(\rho_+)={\bb S} (\bar \rho)$. The function $J^+$ is smooth around ${\bb S} (\bar \rho)$ and $J^+ (E)=E+{\bar V}^+ - S_{\rho_+, \rho_+} (-(E+{\bar V}^+)) -{\bar V}^+$. By performing a second Taylor expansion of $J^+$ around ${\bb S}({\bar \rho})$ we can determine the value of the variance of the Gaussian fluctuations and we get the desired result.

\subsection{Proof of Proposition \ref{lem:3123}.}

\quad

Let us first compute the top of the energy band.
Recall from \eqref{energy band} that $$E^{+}_{-\infty}=  \sup_{\rho \in {\mc M}} \left\{ {\bb S} (\rho) +V^{-} (\rho) \right\}.$$ We have
\begin{equation*}
\sup_{\rho \in {\mc M}} \sup_{\varphi \in {\mc F}} {\mc H}(\rho, \varphi) = \sup_{\varphi \in {\mc F}} \sup_{\rho \in {\mc M}} {\mc H} (\rho,\varphi).
\end{equation*}
For any $\varphi \in {\mc F}$ we define $x_{\varphi} =\sup \left\{ x \in [-1,1] \; : \; \varphi (x) \le 0 \right\}$. Then, $\sup_{\rho \in {\mc M}} {\mc H} (\rho,\varphi)$ is realized for $\rho(x)={\bf 1}\{-1\leq x\leq x_{\varphi}\}$. It follows that
\begin{equation*}
\begin{split}
&\sup_{\rho \in {\mc M}} \sup_{\varphi \in {\mc F}} {\mc H}(\rho, \varphi) =\sup_{x \in [-1,1]} \sup_{ \varphi \in {\mc F}_x} \Big\{ \frac{1}{2}\int_x^1  (\varphi(u) -\log(1+e^{\varphi(u)}))du - \frac{1}{2}\int_{-1}^x  \log(1+e^{\varphi(u)}) du\Big\}
\end{split}
\end{equation*}
where ${\mc F}_x$ is the set of functions $\varphi \in {\mc F}$ such that ${\varphi} (x) =0$ if $x\in (-1,1)$, $\varphi (x) \le 0$ if $x=1$ and $\varphi (x) \ge 0$ if $x=-1$. Assume for example that $\varphi_- \leq 0 \leq  \varphi_+$ (the other cases are similar). By using the fact that the function $t \to t -\log (1+e^t)$ is increasing, we see that
\begin{equation*}
\sup_{\rho \in {\mc M}} \sup_{\varphi \in {\mc F}} {\mc H}(\rho, \varphi)
=\sup_{x \in [-1,1]}  \Big\{ -\frac{(x+1)}{2} \log (1+e^{\varphi_-}) + \frac{(1-x)}{2} \left( \varphi_+ -\log(1+e^{\varphi_+})\right)   \Big\}
\end{equation*}
because the supremum over ${\mc F}_x$ is realized by a sequence of functions in ${\mc F}_x$ converging to the step function $\varphi_{-} {\bf 1}\{-1\leq u\leq x\} + \varphi_+ {\bf 1} \{ x\leq u\leq1\}$. The last supremum is equal to $- \log (1+ e^{-\varphi_0})$ which concludes the computation of the top of the energy band.\\

We compute now the bottom of the energy band. Recall from \eqref{energy band} that $$E^{-}_{-\infty}= \inf_{\rho \in {\mc M}} \left\{ {\bb S} (\rho) +V^{-} (\rho) \right\}.$$ We first recall some results of  \cite{DLS3}. Recall the definition of $H_\rho$ from \eqref{Hfunction} and let $G_{\rho}$ be the convex envelop of $H_{\rho}$, i.e. the biggest convex function $G$ such that $G \leq H_{\rho}$. We recall that any convex function is almost everywhere differentiable. Then, the supremum in (\ref{eq:Vminus}) is given by $- {\bb S} (\rho) +{\mc H} (\rho, \varphi_{G_\rho})+ {\bar V}^-$ where
\begin{equation*}
\varphi_{G_\rho} (x) =
\begin{cases}
\vspace{2mm}
&\varphi_-, \quad G^{\prime}_{\rho} (x) \le \rho_-, \\
\vspace{2mm}
&\log \left( \cfrac{G^{\prime}_{\rho} (x) }{1- G_{\rho}^{\prime} (x)} \right), \quad \rho_- \le G^{\prime}_{\rho} (x) \le \rho_+,\\
\vspace{2mm}
&\varphi_+, \quad G^{\prime}_{\rho} (x) \ge \rho_+.
\end{cases}
\end{equation*}
Moreover, by \eqref{H function} we have that $${\mc H} (\rho, \varphi_{G_\rho})=  \frac{1}{2}\int_{-1}^{1} [ G^{\prime}_{\rho}(x)  \varphi_{G_\rho}(x)  - \log (1+e^{\varphi_{G_\rho}(x)})\,] dx \,.$$ This shows that ${\mc H} (\rho, {\varphi_{G_{\rho}}})$ does not depend on $\rho$ but only on $G_{\rho}$. Since as $\rho$ describes the set ${\mc M}$, $G_{\rho}$ describes exactly the set of non-decreasing convex functions $G$ such that $G  (-1)=0$ and $0 \le G^{\prime} \le 1$, then we have
\begin{equation*}
\inf_{\rho \in {\mc M}} \sup_{\varphi \in {\mc F}} {\mc H} (\rho, \varphi) =\inf_{\rho \in {\mc M}} {\mc H} (\rho, \varphi_{G_{\rho}})= \inf_{G} \left\{  \frac{1}{2}\int_{-1}^{1} [ G^{\prime}(x)  \varphi_{G}(x)  - \log (1+e^{\varphi_{G}(x)})\,] dx \,\right\}
\end{equation*}
where the last infimum is carried over the set of non-decreasing convex functions $G$ such that $G (-1)=0$ and $0 \le G^{\prime} \le 1$.

Let us now consider the set of non-decreasing functions $g$ such that $g \in [0,1]$ and for such $g$ let $T (g):[-1,1]\to{[\varphi_-,\varphi_+]}$ be defined by
\begin{equation*}
T(g) (x) =
\begin{cases}
\vspace{2mm}
&\varphi_-, \quad g (x) \le \rho_-, \\
\vspace{2mm}
&\log \left(\cfrac{ g (x) }{1- g (x)} \right), \quad \rho_- \le g (x) \le \rho_+,\\
\vspace{2mm}
&\varphi_+, \quad g(x) \ge \rho_+.
\end{cases}
\end{equation*}
We have then
\begin{equation}
\label{eq:infenergyband435}
\begin{split}
&\inf_{\rho \in {\mc M}} \sup_{\varphi \in {\mc F}} {\mc H} (\rho,\varphi)= \inf_g \left\{ \frac{1}{2}\int_{-1}^1  [g(x) T(g)(x) -\log (1+e^{T(g)(x)}) ] dx\right\}
\end{split}
\end{equation}
where the infimum is taken over the set of non-decreasing functions $g:[-1,1] \to [0,1]$.
To each non-decreasing function $g: [-1,1] \to [0,1]$ we associate $-1 \le x_- \le x_+ \le 1$ and $0 \le y_- \le \rho_- < \rho_+ \le y_+ \le 1$ defined by
\begin{equation*}
\begin{split}
&x_- = \sup \{ x \in [-1,1] \; : \; g(x) \le \rho_-\}, \quad x_+ = \inf\{ x \in [-1,1] \; : \; g(x) \ge \rho_+\},\\
&y_- = \cfrac{1}{x_- +1} \int_{-1}^{x_-} g(x)dx, \quad y_+ = \cfrac{1}{1-x_+}\int_{x_+}^1 g(x) dx.
\end{split}
\end{equation*}
In the case $x_- =-1$ (resp. $x_+ =1$) we adopt the convention that $y_-=\rho_-$ (resp. $y_+=\rho_+$). With these definitions we can write
\begin{equation*}
\begin{split}
&\int_{-1}^1  [g(x) T(g)(x) -\log (1+e^{T(g)(x)})] dx\\=
 &\Big\{ (x_- +1) \left( \varphi_- y_- -\log(1+e^{\varphi_-}) \right) + (1-x_+) \left( \varphi_+ y_+ -\log(1+e^{\varphi_+}) \right)  \Big\}+ \int_{x_-}^{x_+} s(g(x)) dx.
\end{split}
\end{equation*}
The infimum can be computed by fixing first $x_-, x_+,y_-, y_+$, optimizing separately over the restrictions of $g$ to $[-1,x_-]$, $[x_-, x_+]$ and $[x_+, 1]$ and then taking the infimum over $x_-, x_+,y_-, y_+$. These parameters shall satisfy $-1 \le x_- \le x_+ \le 1$ and $0 \le y_- \le \rho_- < \rho_+ \le y_+ \le 1.$

   By convexity of the function $s(\cdot)$ and by Jensen's inequality we get that the infimum of $\int_{x_-}^{x_+} s(g(x))dx$ is given by $(x_+ -x_-) \inf_{\rho\in [\rho_-, \rho_+]} s(\rho)$. It follows that (\ref{eq:infenergyband435}) is equal to
\begin{equation*}
\begin{split}
\inf_{\substack{-1 \le x_- \le x_+ \le 1\\ 0 \le y_- \le \rho_- < \rho_+ \le y_+ \le 1}} &\frac{1}{2}\Big\{ (x_- +1) \left( \varphi_- y_- -\log(1+e^{\varphi_-}) \right) + (1-x_+) \left(\varphi_+ y_+ -\log(1+e^{\varphi_+}) \right) \\
&+ \left. (x_+ -x_-)\left(\inf_{\rho \in[\rho_-, \rho_+]} s(\rho) \right)\right\}
\end{split}
\end{equation*}
\begin{equation*}
\begin{split}
=\inf_{\substack{-1 \le x_- \le x_+ \le 1}} & \frac{1}{2}\Big\{ (x_- +1) \left( (\varphi_- \wedge 0) \rho_{-} -\log(1+e^{\varphi_-}) \right) + (1-x_+) \left( \varphi_+ \wedge (\varphi_+ \rho_+) \right) \\
& -\log(1+e^{\varphi_+})+ (x_+ -x_-)\left(\inf_{\rho \in[\rho_-, \rho_+]} s(\rho) \right)\Big\}\\
\end{split}
\end{equation*}
We observe now that $\rho_{\pm} = e^{\varphi_{\pm}}/ (1+e^{\varphi_{\pm}})$ and that the function $t \to t e^t /(1+e^t) - \log(1+e^t)$ is even, increasing on $[0,+\infty)$ and negative. The result then follows.

\subsection{Proof of Theorem \ref{th:ent2}.}

\quad

We notice that once the form of $S^-$ is known, the fact that it is concave and continuously differentiable on its energy band is trivial. For completeness, we prove here that for $\rho_+\geq1/2\geq\rho_-$ and for $\rho_0=\rho_+$ the entropy function $S^-$ is continuously differentiable but not twice continuously differentiable. The rest of the cases is completely similar. By the expression for  $S^-$, it is enough to check that $(S^-)'(-s(\rho_+))=1$ and that $(S^-)''(-s(\rho_+))\neq 0$. But this follows from a simple computation using \eqref{eq:entropysimple} and the expression for  $s(\cdot)$.

 In order to  obtain the form of $S^-$ we start to reduce the computation of $S^{-}$ to a $4$ dimensional optimization problem. Some notations shall be introduced. Let $m = \min_{\rho \in [\rho_-,\rho_+]} (-s (\rho))$ and $M = \max_{\rho \in [\rho_-,\rho_+]} (-s (\rho))$. We define the linear functions $\gamma_{\pm} : \RR \to \RR$ by
$$\gamma_\pm(y)= y \varphi_{\pm} -\log (1+e^{\varphi_{\pm}}), \quad y \in \RR.$$

For any $E \ge 0$, let $K:=K(E)$ be the (possibly empty) compact convex domain of $\RR^4$ composed of $4$-tuples $(x_-,x_+,y_-,y_+)$ such that the following conditions are satisfied
\begin{equation*}
\begin{cases}
\vspace{1mm}
&0 \le y_- \le \rho_- < \rho_+ \le y_+ \le 1,\\
\vspace{1mm}
&-1 \le x_- \le x_+ \le 1, \\
\vspace{1mm}
& \frac{1}{2}\Big((x_- +1) (\gamma_- (y_-) +m)+ (1-x_+) (\gamma_+(y_+) +m)\Big) \ge (E +{\bar V}^-) +m,\\
\vspace{1mm}
&\frac{1}{2}\Big( (x_- +1) (\gamma_- (y_-) +M)+ (1-x_+) (\gamma_+ (y_+) +M) \Big)\le (E +{\bar V}^-) +M.
\end{cases}
\end{equation*}

Let $F:K \to \RR$ be the function defined by
\begin{equation*}
F (x_-,x_+,y_-,y_+)=\frac{1}{2}\Big((x_- +1) \left( (-s)(y_-) +\gamma_-(y_{-})\right) + (1-x_+) \left( (-s)(y_+) +\gamma_+(y_{+})\right)\Big).
\end{equation*}

\begin{proposition}
\label{prop:s-}
For any $E \ge 0$, we have that

\begin{equation*}
S^- (E) = \sup_{k \in K} F(k)  \, - \, (E+ {\bar V}^-).
\end{equation*}
Moreover $K\ne \emptyset$ if and only if $E \in [E_{-\infty}^-\,;\, E_{-\infty}^+]$.
\end{proposition}

\begin{proof}
The last part of the proposition follows directly from the computations performed during the determination of the energy band. We assume now that $E \in [E_{-\infty}^+ \, ; \, E_{-\infty}^+]$.

We use the notations introduced in the proof of Proposition \ref{lem:3123}. Then we have
\begin{equation*}
{ S}^{-} (E) = \sup_{\rho \in {\mc M}} \left\{ {\bb S} (\rho)\; ;  \; {\mc H} (\rho, \varphi_{G_\rho}) =E+{\bar V}^-\right\}.
\end{equation*}

Since, by convexity of the function $s$, we have ${\bb S} (\rho) \le {\bb S} (G^{\prime}_{\rho})$, we get
\begin{equation*}
{ S}^{-} (E) \le \sup_{G} \left\{ \; {\bb S} (G^{\prime})\; ;  \; \frac{1}{2}\int_{-1}^{1}   G^{\prime}(x) \varphi_G(x) - \log (1+e^{\varphi_G(x)}) \, dx \, =E+{\bar V}^-\right\}
\end{equation*}
where the supremum is carried over the set of non-decreasing convex functions $G$ such that $G(-1)=0$, $0 \le G' \le 1$. On the other hand, given any non-decreasing convex function $G$ such that $G(-1)=0$ and $0 \le G' \le 1$, let $\rho = 1-G' \in {\mc M}$. We have $H_{\rho}=G=G_{\rho}$ and ${\bb S} (\rho)= {\bb S} (H_{\rho}^{\prime})={\bb S} (G)$. It follows that
\begin{equation*}
{ S}^{-} (E) = \sup_{G} \left\{ \; {\bb S} (G^{\prime})\; ;  \; \frac{1}{2}\int_{-1}^{1}  G^{\prime}(x) \varphi_G(x) - \log (1+e^{\varphi_G(x)}) \, dx \,=E+{\bar V}^-\right\}.
\end{equation*}

This can be written as
\begin{equation}
\label{eq:s-g}
{ S}^{-} (E) = \sup_{g} \left\{ \; {\bb S} (g)\; :  \; \frac{1}{2}\int_{-1}^{1}  g(x) T(g)(x) - \log (1+e^{T(g)(x)}) \,dx \, =E+{\bar V}^-\right\}
\end{equation}
where the supremum is taken over the set of non-decreasing functions $g$ such that $g \in [0,1]$. Then,  the constraint in \eqref{eq:s-g} is given by
\begin{equation}
\label{eq:cons1}
\begin{split}
&\frac{1}{2}\int_{-1}^{1} g(x) T(g)(x) - \log (1+e^{T(g)(x)}) \, dx \, \\
=&\frac{1}{2}\Big( (x_- +1) \gamma_-(y_-) + (1-x_+) \gamma_+(y_{+}) + \int_{x_-}^{x_+} s(g(x)) dx\Big)\\
=& E +{\bar V}^-.
\end{split}
\end{equation}

Fix $x_-,x_+, y_-,y_+$. We decompose the integral appearing in ${\bb S} (g)$ into the three integrals corresponding to the intervals $[-1,x_-]$, $[x_-,x_+]$ and $[x_+,1]$ so that we can optimize independently over the restrictions of $g$ to $[-1,x_-]$ and $[x_+,1]$. Then the value of the integral of $(-s) (g)$ over $[x_-, x_+]$ is fixed by the constraint (\ref{eq:cons1}).

By using that $\sup_g \int_{-1}^{x_-} (-s) (g(x)) dx$ over the constraint that $\int_{-1}^{x_-} g(x)dx = (x_- +1) y_-$ is given by $(x_- +1) (-s) (y_-)$ (and similarly for  $\sup_g \int_{x_+}^{1} (-s) (g(x)) dx$), we conclude that ${ S}^- (E)$ is given by
\begin{equation*}\label{exp entropy}
{ S}^- (E)=\sup\Big\{F(x_-,x_+,y_-,y_+)-(E+{\bar V}^-):(x_-,x_+,y_-,y_+)\in{}\mathcal{A}\Big\}.
\end{equation*}
Above ${\mc A}$ is the set of $4$-tuples $(x_-,x_+,y_-,y_+)$ such that
\[
-1 \le x_- \le x_+ \le 1, \quad 0 \le y_- \le \rho_- < \rho_+ \le y_+ \le 1
\]
and that there exists a non-decreasing function $h:[x_-,x_+] \to [\rho_-,\rho_+]$ satisfying
 \begin{equation*}
\frac{1}{2} \int_{x_-}^{x_+} (-s) (h(x))dx = -(E+{\bar V}^-) +\frac{1}{2}\Big( (x_- +1) \gamma_-(y_{-}) + (1-x_+) \gamma_+(y_{+})\Big).
 \end{equation*}
This last condition can be stated as
\begin{equation}
\label{eq:cons2}
-(E+{\bar V}^-) + \frac{1}{2}\Big((x_- +1) \gamma_-(y_{-}) + (1-x_+)\gamma_+(y_{+})\Big) \in   \left[ \frac{x_+ -x_-}{2} m\, ; \, \frac{x_+ -x_-}{2} M  \right].
\end{equation}
It is easy to see that ${\mc A}=K$ and we have proved the proposition.
\end{proof}

Assume from now on that $E$ belongs to the energy band $[E_{-\infty}^- \, ; \, E_{-\infty}^+]$. We have to compute the supremum of the function $F$ over the non-empty convex compact set  $K$. To do this we first fix $y_- \in [0,\rho_-]$ and $y_+ \in [\rho_+,1]$ and optimize over the couples $(x_-,x_+)$ such that  $(x_-,x_+,y_-,y_+) \in K$.

Observe now that writing
$$(-s)(y_\pm) +\gamma_{\pm}(y_{\pm}) = (-s) (y_{\pm}) - \left[ (-s) (\rho_\pm) + (-s') (\rho_\pm) (y_{\pm} -\rho_{\pm})\right]$$
and using the concavity of $-s$, we get that $F(x_-,x_+,y_-,y_+) \le 0$ for all $(x_-,x_+,y_-,y_+)$; and $F(x_-,x_+,y_-,y_+)=0$ if and only if $(x_ \pm \mp 1) (y_{\pm} -\rho_{\pm}) =0$.

 If $-(E+{\bar V}^-)$ belongs to $[m, M]$ then by taking $x_\pm =\pm 1$, we conclude that $\sup_{k \in K} F (k) =0$
 and consequently that
\begin{equation}\label{linear entropy}
{ S}^- (E) = - (E+{\bar V}^-).
\end{equation}

 Consider now the case where $E$ belongs to the energy band but $ -(E+{\bar V}^-) \notin [m, M]$. Fix first $y_{\pm}$. In order to keep notation simple and since $y_{\pm}$ are fixed, we use the notation $\gamma_{\pm}:=\gamma_{\pm}(y_{\pm})$. We have first to maximize the function $F(\cdot,\cdot, y_-,y_+)$ in the compact convex domain $D:=D(y_-, {y_+})$ composed of $(x_-,x_+)$ such that
\begin{equation}
\label{eq:cons3}
D=\begin{cases}
\vspace{1mm}
&-1 \le x_- \le x_+ \le 1, \\
\vspace{1mm}
& \frac{1}{2}\Big((x_- +1) (\gamma_- +m)+ (1-x_+) (\gamma_++m)\Big) \ge (E +{\bar V}^-) +m,\\
\vspace{1mm}
&\frac{1}{2}\Big( (x_- +1) (\gamma_- +M)+ (1-x_+) (\gamma_+ +M) \Big)\le (E +{\bar V}^-) +M.
\end{cases}
\end{equation}
The two last conditions are obtained from \eqref{eq:cons2}.

%
%
%
%

Since $F(\cdot,\cdot,y_-,y_+)$ is an affine function, the supremum of $F(\cdot,\cdot,y_-,y_+)$ is attained at one of the extremal points of the domain $D$. Consider the lines $D_m$ and $D_M$ defined by
\begin{equation*}
\begin{split}
& D_m = \left\{ (x_-, x_+) \in \RR^2 \, ; \,\frac{1}{2}\Big( (x_- +1) \left( \gamma_- +m \right) + (1-x_+) \left( \gamma_+ +m \right)\Big) = (E +{\bar V}^-) +m\right\},\\
& D_M =\left \{(x_-, x_+) \in \RR^2 \, ; \, \frac{1}{2}\Big((x_- +1) \left( \gamma_- +M \right) + (1-x_+) \left( \gamma_+ +M \right)\Big) =
(E +{\bar V}^-) +M \right\}.
\end{split}
\end{equation*}
There are $3,4$ or $5$ of such extremal points.
The line $D_m$ intersects the lines $x_-=x_+$, $x_-=-1$ and $x_+=1$ at the points
\begin{equation*}
\begin{split}
&X_0:=\left( \cfrac{2(E+{\bar V}^-) - (\gamma_- +\gamma_+)}{\gamma_- - \gamma_+}, \cfrac{2(E+{\bar V}^-) - (\gamma_- +\gamma_+)}{\gamma_- -\gamma_+}\right),\\
&X_m:=\left( -1, 1- \cfrac{2(E+{\bar V}^-) +2m }{\gamma_+ +m}\right)\\
&Y_m:=\left(\cfrac{2(E+{\bar V}^-) +2m }{\gamma_- +m}-1,1\right),
\end{split}
\end{equation*}
respectively.
The line $D_M$ intersects the same lines at the points $X_0$,
\begin{equation*}
 \begin{split}
 &X_M:=\left( -1, 1- \cfrac{2(E+{\bar V}^-) +2M}{\gamma_+ + M}\right)\\
 &Y_M:=\left(\cfrac{2(E+{\bar V}^-) +2M }{\gamma_- +M}-1,1\right),
\end{split}
\end{equation*}
respectively.
Observe that the point $(x_-,x_+)=(-1,1)$ does not belong to the domain $D$ because $-(E+{\bar V}^-) \notin [m,M]$.

The rest of the proof consists in determining what are the extremal points of $D$ according to the position of $E$ in the energy band, find what is the supremum of $F(\cdot,\cdot,y_-,y_+)$ among these extremal points, and then to maximize over $y_-,y_+$. This is accomplished in Appendix \ref{sec:A2}.
The proof of the last statement of the theorem is also postponed to the Appendix.

\subsection{Proof of Theorem \ref{th:cons2}.}

\quad

We start by giving the proof in the case $1/2 \leq 1-\rho_- < \rho_+$, which corresponds to $0\leq{-\varphi_-}< \varphi_+$. The case $1/2 \leq \rho_+ <  1-\rho_-$ is similar and for that reason is omitted.

By the definition of ${\bar V}^-$, we have that ${\bar V}^- = \log (\chi(1/2))=-2\log(2)$. Also, by the results of the previous sections, defining $J^- (E) = E-S^-(E)$, we have for $E+{\bar V}^-\in I_1=[-\log(2);-s(\rho_+)]$, that $J^- (E)=E +  (E+{\bar V}^-)=2E+{\bar V}^-$. Since $J^-$ is linear and increasing, we conclude that  for $E+{\bar V}^-\in I_1$ is holds that $J^- (E) \ge 0$ with equality if and only if $E_0=-\frac{1}{2}{\bar V}^-=\log(2)$,
which satisfies $E_0+\bar{V}^-=-\log(2)\in I_1$. Now, for $E+{\bar V}^-\in I_2=(s(\rho_+); -\log(1+e^{-\varphi_+})]$, $J^- (E)=E - S_{\rho_+,\rho_+}(-(E+{\bar V}^-))$. As in the previous chapter we conclude that  $J^- (E) \ge 0$ with equality if and only if
$E_0:=E_0(\rho_+)=-\rho_+\log(1-\rho_+)-(1-\rho_+)\log(\rho_+)$ and $-E_0\in I_2$.
Now, we notice that by a simple computation $-E_0$ can be written as $-\log(1+e^{\varphi_+})+\frac{\varphi_+}{1+e^{\varphi_+}}$. Since $s(\rho_+)=-\log(1+e^{\varphi_+})+\frac{\varphi_+}{1+e^{-\varphi_+}}=-\log(1+e^{\varphi_+})+\frac{\varphi_+e^{\varphi_+}}{1+e^{\varphi_+}}$ and since $\varphi_+>0$ we easily conclude that $-E_0\leq{s(\rho_+)}$ and as a consequence $-E_0\notin{I_2}$. Then $J^-$ vanishes for a unique value $E_0:=\log(2)=\mathbb{S}(\bar\rho)$ for $\bar\rho=1/2$. Thus, the function $J^-$ is linear around $\log(2)$ and the fluctuations are not Gaussian.

Now we consider the case $\rho_+\leq{1/2}$, which corresponds to $\varphi_- \le \varphi_+<0$.
By the definition of ${\bar V}^-$, we have that ${\bar V}^- = \log (\chi(\rho_+))$.
By the previous  results, for $E+{\bar V}^-\in I_1=[s(\rho_+);s(\rho_-)]$, we have that $J^- (E)=E +  (E+{\bar V}^-)=2E+{\bar V}^-$. We conclude that for $E+{\bar V}^-\in I_1$ it holds that $J^- (E) \ge 0$ with equality if and only if $E_0=-\frac{1}{2}{\bar V}^-=\frac{1}{2}\log (\chi(\rho_+))$. But in this case a simple computation shows that $E_0\notin{I}_1$.
On the other hand for $E+{\bar V}^-\in I_2=[-\log(1+e^{-\varphi_+});s(\rho_+))$ we have that $J^- (E)=E - S_{\rho_+,\rho_+}(-(E+{\bar V}^-))$. As above we conclude that for $E+{\bar V}^-\in I_2$ it holds that  $J^- (E) \ge 0$ with equality if and only if
$E_0:=E_0(\rho_+)=-\rho_+\log(1-\rho_+)-(1-\rho_+)\log(\rho_+)$.
Repeating the same computations as above, one shows that $-E_0<s(\rho_+)$ so that $-E_0\in{I_2}$. In the remaining case, namely $E+{\bar V}^-\in I_3=(s(\rho_-); -\log(1+e^{\varphi_-})]$ we have $J^- (E)=E - S_{\rho_-,\rho_-}(-(E+{\bar V}^-))$ and
$J^-(E_0)=0$ for $E_0:=E_0(\rho_-)=-\rho_-\log(1-\rho_-)-(1-\rho_-)\log(\rho_-)$ but in this case $E_0\notin{I_3}$. Thus, $J^-$ vanishes for a unique value of $-E_0-\bar{V}^-:=-\rho_+\log(\rho_+)-(1-\rho_+)\log(1-\rho_+)=\mathbb{S}(\bar\rho)$ for $\bar\rho=\rho_+$.

The case $0\le \varphi_- < \varphi_+$ is analogous to the previous one and for that reason we omitted its proof.

\section{A local equilibrium statement}
\label{sec:ls}

In this section we give a derivation of the strong form of local equilibrium that we need in order to establish the variational formula (\ref{eq:1.9}). For any $\ve>0$, we split the set $\{-N,\ldots,N\}$ into $K=\ve^{-1}$ boxes of size $\ve N$ (we assume $\ve N$ to be an integer to simplify). To each configuration $\eta \in \Omega_N$, let $\bM(\eta)=(M_1(\eta), \ldots, M_K(\eta))$ with $M_i(\eta)$ begin the number of particles in the $i^{\rm{th}}$ box in the configuration $\eta$. For $\bM=(M_1, \ldots, M_K)$ fixed, we denote by $\Omega_N (\bM)$ the configurations $\eta$ such that for any $i \in \{1,\ldots,K\}$, the number of particles in the $i^{\rm{th}}$ box is $M_i$ and by $Z_N (\bM)$ its cardinal.

The strong form of local equilibrium is the following statement:
\begin{equation*}
\tag*{({\bf H})}
\lim_{\ve \to 0} \,\limsup_{N \to \infty}\, \sup_{\bM}
\,\sup_{\eta \in \Omega_N (\bM)} \, \Big| \log \big(Z_{N} (\bM) \, \mu_{ss,N} (\eta|\bM)\big)
\,\Big|\;=\; 0\; .
\end{equation*}

The stationary state $\mu_{ss,N}$ can be expressed in terms of a product of (infinite) matrices (\cite{D}). We consider the TASEP with $p=1$ but we do not assume in this section that $\rho_- < \rho_+$. Thus, the case $\rho_- < \rho_+$ corresponds to the competitive TASEP and the case $\rho_- > \rho_+$ to the cooperative TASEP (up to a trivial left-right symmetry). Moreover, to have notations consistent with \cite{D}  we consider the boundary driven TASEP on the lattice $\{1, \ldots, N\}$ rather than on $\{-N, \ldots,N\}$. Let $\Sigma_N = \{0,1\}^{\{1, \ldots,N\}}$. \\

By \cite{D}, there exist matrices $D$, $E$ and vectors $|V\rangle,
\langle W|$ such that $\langle W \, | \, V \rangle =1$,
\begin{equation*}
\begin{split}
& DE \;=\;  D\;+\; E \; , \\
&(1- \rho_+) \,  D \,|\, V \rangle \;=\; |\, V \rangle\;, \\
&\langle W \,|\, \rho_-\,  E \;=\; \langle W \, |
\end{split}
\end{equation*}
and
\begin{equation*}
\label{a01}
\mu_{ss,N} (\eta) \;=\; \cfrac{\omega_N (\eta)}
{\langle W | (D+E)^{N} |V \rangle}\;,
\end{equation*}
where the weight $\omega_N (\eta)$ is given by
\begin{equation*}
\omega_N (\eta) \;=\; \langle W | \prod_{x=1}^{N}
\left\{ \eta(x) D + [1-\eta(x)] E\right\} | V \rangle\;.
\end{equation*}

\begin{lemma}
\label{lem:45678}
For any $N \ge 2$, any  $\eta \in \Sigma_N$ such there exists a site $x\in \{1, \ldots,N-1 \}$ for which $\eta_x =1$, $ \eta_{x+1} =0$, we have that
\begin{equation*}
s_N (x,\eta)= \omega_N (\eta) - \omega_N (\sigma^{x,x+1} \eta)
\end{equation*}
has the same sign as $\rho_- -\rho_+$.
\end{lemma}

\begin{proof}
Let  us define $s=\frac{1}{\rho_{-}} +\frac{1}{1-\rho_+} -\frac{1}{(1- \rho_+)\rho_-}$ which has the same sign as $\rho_- -\rho_+$. We prove the lemma by induction. A configuration of $\Sigma_N$ is identified with a sequence of $0$'s and $1$'s of length $N$. For example $011$ is the configuration $\eta \in \Sigma_3$ such that $\eta_{1} =0, \eta_2= 1, \eta_3 =1$.  For $N=2$, the induction hypothesis is trivial since
\begin{equation*}
\begin{split}
\omega_2 (10) -\omega_2 (01) = \langle W \, | \, DE-ED \, | \, V \rangle = s.
\end{split}
\end{equation*}
Assume that the induction hypothesis is valid for $N-1$. Consider a configuration $\eta \in \Sigma_N$ such that $\eta_x=1$, $\eta_{x+1} =0$, $x \in \{ 1, \ldots,N-1\}$. We write $\eta$ in the form $\eta=\alpha 10 \beta$ where the $1$ is at position $x$.
If $\alpha=\alpha' 1$ then by using the relation $DE= D+E$, we have
\begin{equation*}
\begin{split}
s_N (x,\eta)&= \omega_{N} (\alpha' 110 \beta) - \omega_N (\alpha' 101 \beta)\\
&= \omega_{N-1} (\alpha' 11 \beta) + \omega_{N-1} ( \alpha' 10 \beta) - \omega_{N-1} (\alpha' 11 \beta) - \omega_{N-1} ( \alpha' 01 \beta)\\
&= s_{N-1} (x-1, \eta')
\end{split}
\end{equation*}
where $\eta' = \alpha' 10 \beta$. Thus, by the induction hypothesis applied to $\eta'$, $s_N (x,\eta)$ has the same sign as $s$. If $\beta=0\beta'$, the same conclusion holds. Thus we can assume that $\eta$ is in the form $\alpha 0101\beta$. If $\beta= 0 \beta'$, by using the relation $DE=D+E$ applied at position $(x+2, x+3)$, we get
\begin{equation*}
\begin{split}
s_N (x, \eta)&= \omega_{N-1} (\alpha0101\beta') -\omega_{N-1} (\alpha0011\beta')+\omega_{N-1}(\alpha0100\beta') -\omega_{N-1} (\alpha0010\beta')\\
&= s_{N-1}(x,\eta^1) + s_{N-1} (x,\eta^0)
\end{split}
\end{equation*}
where the $\eta^1= \alpha0101\beta'$ and $\eta^0=\alpha0100\beta'$. By the induction hypothesis, this has the same sign as $s$. The same conclusion holds if $\alpha=\alpha' 1$. By iterating this procedure, one can prove that $s_N (x,\eta)$ has the same sign as $s$ if there is a $1$ to the left of $x-1$ or a $0$ to the right of $x+2$. The only remaining case is if $\eta$ is in the form $\eta=0\ldots 0\, 10\, 1 \ldots 1$ with $m$ zeroes to the left of the leftmost one and $n$ ones to the right of the rightmost zero. But in this case we have
\begin{equation*}
\begin{split}
s_N (x,\eta) = \langle W \, | \, E^m (DE -ED) D^n \, | \, V \rangle &= \frac{1}{\rho_-^m (1-\rho_+)^n} \langle W \, | \, (DE -ED) \, | \, V \rangle\\
& = \frac{s}{\rho_-^m (1-\rho_+)^n}
\end{split}
\end{equation*}
which has the same sign as $s$ and the lemma is proved.
\end{proof}

This lemma is sufficient to prove the local equilibrium statement as in \cite{BL}.

\appendix

\section{Proof of Theorem \ref{th:ent1}}
\label{sec:A1}
Recall the definition of $D$ from \eqref{eq:condm}. We can rewrite the set $D$ in a more convenient form by introducing
\begin{equation*}
a_m = \cfrac{E(m)-1}{1-\xi_0} \quad \textrm{and}\quad \quad b_m= \cfrac{m-E(m)}{\xi_0}.
\end{equation*}
It is clear that if $E(m) \ge \sup\{\xi_0, m-\xi_0\}$ then $a_m \ge -1$ and $b_m \le 1$. Also,  if $E(m) \le m+(1-m) \xi_0$ then $a_m \le b_m$. We have $a_m=-1$ (resp. $b_m=1$) if and only if $E(m)=\xi_0$ (resp. $m-E(m)=\xi_0$).

We have
\begin{equation*}
(y,m) \in D \Leftrightarrow
\begin{cases}
\vspace{1mm}
&0\leq m\leq2, \quad \sup\{\xi_0, m-\xi_0\} \le E(m) \le m- (m-1)\xi_0,\\
\vspace{1mm}
&a_m \le y \le b_m.
\end{cases}
\end{equation*}

It is easy to check that $D \ne \emptyset  $ is equivalent to $\{ m \in [0,2]\, ;\, \sup\{\xi_0, m-\xi_0\} \le E(m) \le m- (m-1)\xi_0  \}\ne \emptyset$ which is equivalent to $E\in [E_{+\infty}^-, E_{+\infty}^+]$.

Assume from now on that $E\in [E_{+\infty}^-, E_{+\infty}^+]$.

 We denote by $\alpha:=\alpha(E)$ (resp. $\beta:=\beta(E)$, resp. $\gamma:=\gamma (E)$) the solution to the linear equation $E(m)=\xi_0$ (resp. $m-E(m)=\xi_0$, resp. $E(m)-m+(m-1)\xi_0=0$). We have that $\{ m \in [0,2]\, ;\, \sup\{\xi_0, m-\xi_0\} \le E(m) \le m- (m-1)\xi_0  \}$ is given by $m\in [0,2]$ such that
\begin{equation*}
\begin{cases}
\vspace{1mm}
E(\alpha) \le E(m)\\
\vspace{1mm}
m-E(m) \le \beta -E(\beta)\\
\vspace{1mm}
E(m)-m+(m-1)\xi_0\le E(\gamma)-\gamma+(\gamma-1)\xi_0.
\end{cases}
\end{equation*}
Since $E(m)$ is a linear function of $m$, $\{ m \in [0,2]\, ;\, \sup\{\xi_0, m-\xi_0\} \le E(m) \le m- (m-1)\xi_0  \}$ is a closed interval $[m_-, m_+]$ (with $m_-:=m_-(E)$ and $m_+:=m_+(E)$) of $[0,2]$ and it is easy to show by inspection of the different cases that we have:

\begin{table}[htbm]
\renewcommand{\arraystretch}{1.8}
\label{eq:valuesalphabetagamma}
\begin{center}
\begin{tabular}{|c|c|c|c|}
\hline
$\frac{1}{2} \le \rho_-<\rho_+ $& $\rho_- < \rho_+ \le \frac{1}{2}$ & $\frac{1}{2}\le 1-\rho_-\le  \rho_+$& $\frac{1}{2}\le \rho_+\le 1-\rho_-$\\
\hline
 $m_- =\sup\{ \alpha,\beta\}$&$m_- =\gamma$&$ m_-=\alpha$&$ m_-=\gamma$\\
\hline
$m_+ =\gamma$&$m_+ =\inf\{ \alpha,\beta\}$&$m_+=\gamma$&$m_+=\beta$ \\
\hline
$\inf \{\alpha,\beta\} \notin (m_-,m_+)$& $\sup\{\alpha,\beta\} \notin (m_-,m_+)$ & $\beta \notin (m_-,m_+)$ & $\alpha \notin (m_-,m_+)$\\
\hline
\end{tabular}
\end{center}
\vspace{0.2cm}
\caption{Values of $m_-$ and $m_+$ in terms of $\alpha,\beta$ and $\gamma$.}
\end{table}

This shows that $\gamma$ is always equal to $m_{-}$ or $m_+$, that $\alpha$ or $\beta$ is the other boundary of the interval $[m_-, m_+]$ and that the remaining point among $\{\alpha,\beta,\gamma\}$ does not belong to $(m_-,m_+)$.

Let $f: [m_-, m_+]\to \mathbb{R}$ be defined  by
\begin{equation}
\label{eq:defoff}
f(m)=\sup_{y \in [a_m, b_m]} F(y,m)
\end{equation}
so that
\begin{equation}
\label{eq:sfg0}
\begin{split}
{S}^+ (E) &=\frac{1}{2} \, \sup_{m \in [m_-, m_+] } f(m).
\end{split}
\end{equation}

Observe that $\alpha \ne \beta\ne \gamma$ apart from a finite (at most three) number of explicit values of $E$. If $E$ is different from these values we say that $E$ is a regular value of the energy. For simplicity we restrict the study to the case where $E$ is a regular value but the same analysis could be performed for the non-regular values of $E$. Since $\alpha=\beta$ is equivalent to $\alpha=\beta=2\xi_0$ we will always assume that it is not the case.

%
%

\begin{lemma}\label{lemma use}
Let $E \in [E_{+\infty}^-\,;\, E_{+\infty}^+]$ be a regular value of the energy.\\
For any $m\in (m_-,m_+)$, we have that $f(m)=F(y(m),m)$ for a unique $y(m) \in (a_m,b_m)$ which is solution to the equation $\partial_{y} F(y(m),m)=0$.\\
If $\alpha$ (resp. $\beta$) belongs to $\{m_-, m_+\}$  then $f(\alpha)= -2 s(\alpha/2)$ (resp. $f(\beta)= -2 s (\beta /2)$) and the supremum appearing in the definition of $f$ is uniquely realized for $y=-1$ (resp. $y=1$). \\
We have $f(\gamma)=0$ and the supremum in the definition of $f$ is uniquely realized for $y=\gamma-1$.
\end{lemma}
\begin{proof}

We notice that if $m=\gamma$ then $a_m=b_m=m-1$ and $f(\gamma)=0$. This shows the last sentence of the lemma.

Now, let $m\in (m_-,m_+)$.
Then $a_m < b_m$. For any $y \in (a_m,b_m)$, we have
\begin{equation*}
\begin{split}
\partial_{y} F (y,m) &= -\cfrac{\xi_0 -E(m)}{y+1} s^{\prime} \left( \cfrac{y\,  \xi_0 + E(m)}{y +1}\right) - s  \left( \cfrac{y\,  \xi_0 + E(m)}{y +1}\right)\\
&-  \cfrac{m-E(m)-\xi_0}{1-y}s^{\prime} \left( \cfrac{m- (y \, \xi_0 + E(m))}{1-y}\right) +  s\left( \cfrac{m- (y \, \xi_0 + E(m))}{1-y}\right),\\
\partial_{y}^2 F (y,m) &=- \cfrac{(\xi_0 -E(m))^2}{(y+1)^3} s^{\prime\prime} \left( \cfrac{y\,  \xi_0 + E(m)}{y +1}\right) \\
&- \cfrac{(m-E(m)-\xi_0)^2}{(1-y)^3}s^{\prime\prime} \left( \cfrac{m- (y \, \xi_0 + E(m))}{1-y}\right).
\end{split}
\end{equation*}
Since $s$ is a strictly convex function, $\partial_{y}^2 F (y,m)<0$, i.e. $F (\cdot,m)$ is strictly concave on $(a_m,b_m)$ so that $\sup_{[a_m,b_m]} F (\cdot,m)$ is attained for a unique point $y(m)$ of $[a_m,b_m]$. If $m \ne \alpha$ then $a_m \ne -1$ and if $m \ne \beta$ then $b_m \ne 1$.  Noticing  that $(y\xi_0+E(m))/(y+1)$ (resp. $(m-(y\xi_0+E(m)))/(1-y)$) goes to $1$ (resp. $0$) as $y$ goes to $a_m$ (resp. $b_m$), we conclude that $\partial_{y} F (y,m)$ goes to $+\infty$ (resp. $-\infty$). This implies that $y(m) \in (a_m,b_m)$ and $(\partial_y F)(y(m),m)=0$. If $m=\alpha \in [m_-,m_+]$, then
$$F(y,\alpha)= -(y+1)s (\xi_0) -(1-y) s\left( \cfrac{\alpha- \xi_0 (y + 1)}{1-y}\right)$$
for $y\in [-1,b_\alpha] \subset [-1,1)$, $F(\cdot,\alpha)$ is strictly concave on$[-1,b_\alpha]$ and
\begin{equation*}
\begin{split}
\lim_{ y \to -1} \partial_y F(y,\alpha) = (-s)(\xi_0) - (-s) (\alpha/2) - (\xi_0 -\alpha/2) (-s)' (\alpha/2) < 0
\end{split}
\end{equation*}
because $s$ is strictly convex and $\alpha \ne 2\xi_0$ (otherwise $E$ is not regular). The maximum of $F(\cdot, \alpha)$ is uniquely attained for $y(\alpha):=-1$ and $F(y(\alpha),\alpha)= -2 s(\alpha/2)$.

Similarly, if $m=\beta \in [m_-,m_+]$, the supremum is uniquely attained for $y(\beta)=1$ and $F(y(\beta),\beta)=-2s(\beta/2)$.

\end{proof}

\begin{lemma}
\label{lem:C1func}
Let $E \in [E_{+\infty}^-\,;\, E_{+\infty}^+]$ be a regular value of the energy. The function $f:[m_-,m_+] \to [0,2\log(2)]$ is a continuously differentiable function on $(m_-,m_+)$ and, when $\alpha,\beta \in [m_-,m_+]$,
\begin{equation*}
f'(\alpha)= - s'( \alpha /2)\quad and \quad \quad f'(\beta)= - s'( \beta /2).
\end{equation*}
\end{lemma}

\begin{proof}
It is clear that $f$ is smooth on $(m_-,m_+)$ and that the implicit function theorem applies. Thus the continuity and differentiability problems are only around the points $m_-$ and $m_+$.

Let us prove that $f(m)$ goes to $f(\alpha)$ and that $f' (m)$ has a limit as $m$ goes to $\alpha$ (assuming that $\alpha\in\{m_-,m_+\}$) equal to $-s' (\alpha/2)$. The other case can be treated similarly. We will show that
\begin{equation}
\label{eq:limyuv}
\lim_{m \to \alpha} y(m)=-1,\quad \lim_{m \to \alpha} \cfrac{y(m)\,  \xi_0 + E(m)}{y(m) +1}=\lim_{m \to +\infty}  \frac{m- (y(m) \, \xi_0 + E(m))}{1-y(m)} = \alpha/2.
\end{equation}
By the implicit function theorem, for any $m \ne \alpha, \beta$ we have
\begin{equation*}
\begin{split}
f' (m) &= \partial_m F (y(m),m) \\
&= - \frac{\varphi_+}{\varphi_+ - \varphi_-} s' \left( \cfrac{y(m)\,  \xi_0 + E(m)}{y(m) +1} \right) + \frac{\varphi_-}{\varphi_+ - \varphi_-} s' \left(\cfrac{m- (y(m) \, \xi_0 + E(m))}{1-y(m)} \right).
\end{split}
\end{equation*}
From (\ref{eq:limyuv}) we deduce that $\lim_{m \to \alpha} f(m)= f(\alpha)$ and $\lim_{m \to \alpha} f' (m) =-s' (\alpha/2)$.

Let $(m_n)_{n \ge 0}$ be a sequence in $(m_-,m_+)$ ($m_n \ne \alpha,\beta,\gamma$ for any $n$) converging to $\alpha$. Since for any $m$,
$$y(m) \in [-1,1], \quad \cfrac{y(m)\,  \xi_0 + E(m)}{y(m) +1} \in [0,1] ,\quad \cfrac{m- (y(m) \, \xi_0 + E(m))}{1-y(m)} \in [0,1] ,$$
up to a subsequence we can assume that $y_n:=y(m_n)$ converges to some $a\in [-1,1]$ and that $\frac{y_n\,  \xi_0 + E(m_n)}{y_n +1}$ converges to $u\in [0,1]$ and  $\frac{m_n- (y_n \, \xi_0 + E(m_n))}{1-y_n}$ converges to $v\in [0,1]$.

If $a \in (-1,1)$, then $u= \xi_0 \in (0,1)$, $v= \xi_0 + \frac{\alpha -2 \xi_0}{1-a}$ and $v\ne u$ since $\alpha \ne 2\xi_0$. By continuity of the functions involved and taking into account that $\partial_{y} F(y_n, m_n)=0$, we get that
\begin{equation*}
-s (u) -\frac{\alpha -2\xi_0}{1-a} s' (v) +s(v) =0.
\end{equation*}
The term on the left hand side of the previous equality can be written as $s(v) -s(u) -(v-u) s'(v)$ which is negative, by the convexity of the function $s$ (if $v=0$ or $v=1$ then $s'(v) =- \infty$ or $s'(v)=+\infty$ and the inequality is still valid).
Therefore $a \in \{-1,1\}$.

If $a=1$, then since $\frac{m_n- (y_n \, \xi_0 + E(m_n))}{1-y_n}$ converges to $v\in [0,1]$ and $y_n \to a=1$, $m_n \to \alpha$ it implies that $\alpha=2\xi_0$ that is in contradiction with our assumptions.

It follows that $a= -1$, so that $v=\alpha/2$. Observe that $\frac{E(m_n) -\xi_0}{y_n +1}$ converges to $u -\xi_0$. Using the fact that $\partial_{y} F(y_n, m_n)=0$, we get
\begin{equation*}
\begin{split}
0&=(u-\xi_0) s' (u) -s(u) -\frac{\alpha -2\xi_0}{2} s' (\alpha /2) + s(\alpha/2).
\end{split}
\end{equation*}
Since $s$ is convex, the function $z \to (z-\xi_0)s'(z) -s(z)$ is monotone, so that the equality is uniquely satisfied for $u=\alpha/2$. This proves (\ref{eq:limyuv}).
\end{proof}

\begin{lemma}
Let $E \in [E_{+\infty}^-\,;\, E_{+\infty}^+]$ be a regular value of the energy. The function $f$ is strictly concave on $(m_-,m_+)$.
\end{lemma}

\begin{proof}
On $\mathring D$ we have
\begin{equation*}
\begin{split}
\partial^{2}_m F (y,m) &= - \left( \cfrac{\varphi_+}{\varphi_+ -\varphi_-} \right)^2 \cfrac{1}{y+1} s^{\prime \prime} \left( \cfrac{\xi_0 -E(m)}{y+1} \right)\\
& - \left( \cfrac{\varphi_-}{\varphi_+ -\varphi_-} \right)^2 \cfrac{1}{1-y} s^{\prime \prime} \left( \cfrac{m- (y \, \xi_0 + E(m))}{1-y} \right) \\
\partial^2_{y,m} F (y,m) &= - \left( \cfrac{\varphi_+}{\varphi_+ -\varphi_-} \right) \cfrac{\xi_0 -E(m)}{(y+1)^2} s^{\prime \prime} \left( \cfrac{\xi_0 -E(m)}{y+1} \right) \\
&+ \left( \cfrac{\varphi_-}{\varphi_+ -\varphi_-} \right) \cfrac{(m-E(m)-\xi_0)}{(1-y)^2} s^{\prime \prime} \left( \cfrac{m- (y \, \xi_0 + E(m))}{1-y} \right).
\end{split}
\end{equation*}

One easily checks that
\begin{equation*}
\begin{split}
&(\partial_{y}^2 F)(\partial_m^2 F) - (\partial^2_{y,m} F)^2 \\
&= \cfrac{s^{\prime\prime} (a) s^{\prime \prime } (b)}{(\varphi_+ -\varphi_-)^2(y^2-1)^2} \left[ \varphi_+ \sqrt{\cfrac{1+y}{1-y}} (m-E(m)-\xi_0) + \varphi_- \sqrt{\cfrac{1-y}{1+y}} (\xi_0 -E(m))\right]^2,
\end{split}
\end{equation*}
where

$$ a= \cfrac{y\xi_0 -E(m)}{y+1} \quad \textrm{and} \quad b= \cfrac{m- (y \, \xi_0 + E(m))}{1-y}.$$
By convexity of the function $s$ it follows that $(\partial_{y}^2 F)(\partial_m^2 F) - (\partial^2_{y,m} F)^2 >0 $ if $m \ne \alpha,\beta$.

By the implicit function theorem, the function $m \to F(y(m),m)$ is smooth on $(m_-, m_+)$ and
\begin{equation*}
\begin{split}
&f'' (m)= \partial^2_m F(y(m),m) + \partial^2_{y,m} F ( y(m),m)\,  y' (m)\\
&\cfrac{d}{dm} \left[ \partial_y F  (y(m),m)\right]=0=(\partial_{m,y}^2 F)(y(m),m)+(\partial_y^2 F)(y(m),m) \, y' (m),
\end{split}
\end{equation*}
so that
\begin{equation*}
f'' (m) = \frac{\left[ (\partial_{y}^2 F)(\partial_m^2 F) - (\partial^2_{y,m} F)^2\right](y(m),m) }{\partial_y^2 F(y(m),m)} .
\end{equation*}
Recalling from the proof of lemma \ref{lemma use} that $\partial_y^2 F <0$, we get  that $f'' (m) < 0$.

\end{proof}

We have to compute $\sup_{m \in [0,2]} f(m)$. Since $f$ is strictly concave there exists a unique $m_0 \in [m_-,m_+]$ for which the supremum of $f$ is attained.

The point $m_0$ belongs to $(m_-,m_+)$ if and only if there exists $m \in (m_-,m_+)$ such that $f'(m)=0$. This is equivalent to the existence of $m \in (m_-,m_+)$ and $y \in (a_m, b_m)$ such that
\begin{equation}
\label{eq:Fym}
\partial_m F(y,m) =0, \quad \partial_{y} F (y,m)=0.
\end{equation}
To simplify notations we introduce
\begin{equation*}
\tilde{a}= s^{\prime} \left( \cfrac{y\xi_0 +E(m)}{y+1} \right)\quad \textrm{and} \quad \tilde{b}= s^{\prime} \left( \cfrac{m- (y \, \xi_0 + E(m))}{1-y}\right),
\end{equation*}
so that $\tilde{a}:=s'(a)$ and $\tilde{b}:=s'(b)$, where $a$ and $b$ were introduced above.
Then (\ref{eq:Fym}) is equivalent to
\begin{equation*}
\tilde b=\cfrac{\varphi_+}{\varphi_-} \;\;\tilde a\quad \textrm{and} \quad \xi_0 \left( 1-\cfrac{\varphi_-}{\varphi_+}\right) \tilde b + \log \left( \cfrac{1+e^{\frac{\varphi_-}{\varphi_+}\tilde b }}{1+e^{\tilde b}}\right) =0.
\end{equation*}
There are two solutions to the second equation, $\tilde b=0$ and $\tilde b=\varphi_+$.\\

If $\tilde b=0$ then from $a=b=1/2$ we get that $m=1$. As a consequence we obtain that $y=(1/2-E(1))/(\xi_0 -1/2)$. The condition $y \in (a_1,b_1)$ implies that $E(1) < \sup \{1-\xi_0, \xi_0\}$,
which is in contradiction with the fact that $m=1$ shall satisfy $\sup\{ \xi_0, m-\xi_0\} \le E(m) \le m - (m-1) \xi_0$.\\

 If $\tilde b=\varphi_+$ then
\begin{equation*}
a=\cfrac{y\xi_0 +E(m)}{y+1}=\rho_- \quad \textrm{and} \quad b=\cfrac{m-(y \xi_0 +E(m))}{1-y} =\rho_+.
\end{equation*}
As a consequence we obtain that $y= \cfrac{m-\rho_+ - \rho_- }{\rho_- -\rho_+}.$ Then we get that
\begin{equation*}
E(m) - \xi_0 = (\rho_- - \xi_0) (y+1), \quad (m- E(m) -\xi_0) =(\rho_+ -\xi_0) (1-y).
\end{equation*}
Thus the conditions $\sup \{ \xi_0, m-\xi_0\} \le E(m)$ and $y \in (-1,1)$ imply $\rho_- \ge \xi_0 \ge \rho_+$. But we assumed $\rho_- < \rho_+$ and we have a contradiction.

Therefore $m_0 \in \{m_-, m_+\}$. Consequently, for any regular value $E \in [E_{+\infty}^-\,;\, E_{+\infty}^+]$,
\begin{equation*}
{S}^+ (E) = \frac{1}{2} \; \sup \Big\{ \, f(m_-) \, , \, f(m_+) \, \Big\}.
\end{equation*}

Recall that the set $\{m_-,m_+\}$ is equal to $\{\alpha,\gamma\}$ or to $\{\beta, \gamma\}$ and  that $f(\gamma)=0$, $f(\alpha)=-2 s(\alpha /2)$ and $f(\beta)= -2 s(\beta/2)$. Thus, by using the results in Table 1, we have
\begin{equation*}
{S}^+ (E)=
\begin{cases}
\vspace{1mm}
- s (\sup\{\alpha,\beta\}/2), \quad 0 <\varphi_- < \varphi_+,\\
\vspace{1mm}
- s (\inf\{\alpha,\beta\}/2), \quad \varphi_- < \varphi_+<0,\\
\vspace{1mm}
- s (\alpha/2), \quad 0 < -\varphi_- \le \varphi_+,\\
\vspace{1mm}
- s(\beta/2), \quad 0 < \varphi_+ \le -\varphi_-.
\end{cases}
\end{equation*}

By definition of $\alpha$ and $\beta$ we have that
\begin{equation*}
\frac{\alpha}{2}= \frac{1}{\varphi_+} \left\{ (E+{\bar V}^+) + \log(1+e^{\varphi_+})\right\}\quad \textrm{and} \quad \frac{\beta}{2} =\frac{1}{\varphi_-} \left\{ (E+{\bar V}^+) + \log (1+e^{\varphi_-}) \right\}.
\end{equation*}
Recall from \eqref{W function} the definition of $W(\varphi_-, \varphi_+)$ and let $\varphi =\log (\rho /(1-\rho))$. Observe that the function $x \to S_{\rho, \rho} (-x)$ is a concave function equal to $- \infty$ outside $[i_1,i_2]:=[ - \log (1+e^{-|\varphi|}) \, ; \, - \log (1+ e^{|\varphi|}) ]$, positive inside, vanishing at the boundaries of the interval, attaining its maximum equal to $\log(2)$ for $x_0:=(i_1+i_2)/2= \varphi/2 -\log (1+e^{\varphi})$. It is increasing on $[i_1,x_0]$ and decreasing on $[x_0,i_2]$.

\subsection{The case $\frac{1}{2}\le 1-\rho_- \le \rho_+ $:}

\vspace{2mm}

\quad

\vspace{2mm}

Recall that the energy band is given by
$$-{\bar V}^+ + [ -\log (1+e^{\varphi_+}) \, ; \, W(\varphi_-, \varphi_+)].$$
The entropy function ${S}^+$, in the energy band, is given by ${S}^+ (E)  =S_{\rho_+, \rho_+} (-(E+{\bar V}^+)).$
Remark that we have $W(\varphi_-, \varphi_+) + \log (1+e^{\varphi_-}) = \varphi_- \xi_0 <0$ since $\xi_0 \in (0,1)$. Thus, $W(\varphi_-,\varphi_+) < -\log (1+e^{-\varphi_+})$ and the function ${S}^+$ is concave, smooth in the interior of the energy band, but does not vanish at the top of the energy band.

\subsection{The case $\frac{1}{2}\le  \rho_+ \le 1-\rho_- $:}

\vspace{2mm}

\quad

\vspace{2mm}

Recall that the energy band is given by
$$-{\bar V}^+ + [ -\log (1+e^{-\varphi_-}) \, ; \, W(\varphi_-, \varphi_+)].$$
The entropy function ${S}^+$, in the energy band, is given by ${S}^+ (E)  =S_{\rho_-, \rho_-} (-(E+{\bar V}^+)).$
The function ${S}^+$ is concave, smooth in the interior of the energy band, but does not vanish at the top of the energy band.

\vspace{2mm}

It remains now to prove the last statement of Theorem \ref{th:ent1}. Let us assume that $\rho$ is a maximizer of $S^{+} (E)$, $E$ belonging to the energy band. We use the notations of the proof of Proposition  \ref{prop:ent1}. In the proof of this proposition, we have seen that $\rho$ being a maximizer of $S^+ (E)$ is equivalent to the fact that $(y_{\rho}, H_{\rho} (1)) \in D(E)$ being a maximizer of the function $F$ over $D(E)$ and $\rho$ is such that $H_{\rho}$ is linear on $[-1,y_{\rho}]$ and on $[y_{\rho},1]$ with $H_{\rho} (y_\rho) = y_\rho \xi_0 +E(m_{\rho})$. Moreover, we have seen above that such a maximizer $(y,m)$ satisfies $y = \pm 1$ and
\begin{enumerate}[i)]
\item $m= \sup (\{\alpha,\beta\})=
\begin{cases}
\beta \text{ if } E+{\bar V}^+ > W(\rho_-, \rho_+),\\
\alpha \text{ if } E+{\bar V}^+ \le  W(\rho_-, \rho_+),\\
\end{cases}
, \quad 1/2 \le \rho_-<\rho_+$,
\item $m= \inf (\{\alpha,\beta\})=
\begin{cases}
\beta \text{ if } E+{\bar V}^+ < W(\rho_-, \rho_+),\\
\alpha \text{ if } E+{\bar V}^+ \ge  W(\rho_-, \rho_+),\\
\end{cases}
, \quad \rho_-<\rho_+\le 1/2$,
\item $m=\alpha \quad 1/2 \le 1-\rho_-\le \rho_+$,
\item $m=\beta \quad 1/2 \le \rho_+ \le 1-\rho_-$.
\end{enumerate}

 This implies in particular that if $\rho$ is a maximizer of $S^+ (E)$ then $H_{\rho}$ is linear with a slope equal to  $m_{\rho}/2$, i.e. $\rho$ is constant equal to $1\, -\, m_{\rho}/2$.  Since, by definition, we have
 \begin{equation*}
 \alpha= 2\,  \frac{\log(1+e^{\varphi_+}) + (E+{\bar V}^+)}{\varphi_+}\quad \textrm{and} \quad \beta= 2 \, \frac{\log(1+e^{\varphi_-}) + (E+{\bar V}^+)}{\varphi_-},
 \end{equation*}
we get the result.

\subsection{ The case $\rho_- < \rho_+ \le \frac{1}{2}$: }

\vspace{2mm}

\quad

\vspace{2mm}

 Recall that the energy band is given by
$$-{\bar V}^+ + [-\log(1+e^{-\varphi_-}) \, ; \, -\log (1+e^{\varphi_+})].$$
The condition $\alpha/ 2 \le \beta /2$ is equivalent to
\begin{equation*}
E +{\bar V}^+ \ge W(\varphi_-, \varphi_+).
\end{equation*}
Observe that $W(\varphi_-, \varphi_+) + \log (1+e^{\varphi_+}) = \varphi_+ \xi_0$ and $\xi_0 \in (0,1)$,  so that
\begin{equation*}
-\log (1+e^{-\varphi_-}) < -\log(1+e^{-\varphi_+}) \le W(\varphi_-, \varphi_+) \le -\log (1+e^{\varphi_+}) < -\log (1+ e^{\varphi_-}).
\end{equation*}

Moreover, we have $\xi_0 <1/2$ because there exists ${\tilde \varphi} \in [\varphi_-, \varphi_+]$ such that $\xi_0= \frac{e^{\tilde \varphi}}{1+e^{\tilde \varphi}}.$
This implies that $W(\varphi_-, \varphi_+) \ge  \frac{\varphi_+}{2} - \log (1+e^{\varphi_+}) $. We recall that $ \frac{\varphi_+}{2} - \log (1+e^{\varphi_+})$ is the first coordinate of the point for which the concave function $x \to S_{\rho_+,\rho_+} (-x)$ attains its maximum given by $\log(2)$. It follows that ${S}^+ $ is a concave function. On the energy band it is given by
\begin{equation*}
\begin{split}
{S}^+ (E) &= S_{\rho_-, \rho_-} (-(E+{\bar V}^+)) \, {\bf 1}\{(E+{\bar V}^+) \le W(\varphi_-,\varphi_+)\}\\&+ S_{\rho_+, \rho_+} (-(E+{\bar V}^+)) \,  {\bf 1}\{(E+{\bar V}^+) > W(\varphi_-,\varphi_+)\}.
\end{split}
\end{equation*}
The function ${S}^+$ is not differentiable at the point $W(\varphi_-,\varphi_+) - {\bar V}^+$.

\subsection{The case $\frac{1}{2}\le \rho_- < \rho_+$:}

\vspace{2mm}

\quad

\vspace{2mm}

Recall that the energy band is given by
$$-{\bar V}^+ + [-\log(1+e^{\varphi_+}) \, ; \, -\log (1+e^{-\varphi_-}) ].$$
The condition $\alpha/ 2 \le \beta /2$ is equivalent to $E +{\bar V}^+ \ge W(\varphi_-, \varphi_+).$ Observe that $W(\varphi_-, \varphi_+) + \log (1+e^{\varphi_-}) = \varphi_- \xi_0$ and  $\xi_0 \in (0,1)$,  so that
\begin{equation*}
-\log (1+e^{\varphi_+}) \le -\log (1+e^{\varphi_-}) \le W(\varphi_-, \varphi_+) \le -\log (1+e^{-\varphi_-}) \le -\log (1+e^{-\varphi_+}).
\end{equation*}
Moreover, we have $\xi_0 >1/2$ because there exists ${\tilde \varphi} \in [\varphi_-, \varphi_+]$ such that $\xi_0= \frac{e^{\tilde \varphi}}{1+e^{\tilde \varphi}}.$
This implies that $W(\varphi_-, \varphi_+) \ge \frac{\varphi_-}{2} - \log (1+e^{\varphi_-}) $. We recall that $ \frac{\varphi_-}{2} - \log (1+e^{\varphi_-})$ is the first coordinate of the point for which the concave function $x \to S_{\rho_-,\rho_-} (-x)$ attains its maximum given by $\log(2)$. It follows that ${S}^+$ is a concave function. On the energy band it is given by
\begin{equation*}
\begin{split}
{S}^+ (E) &= S_{\rho_-, \rho_-} (-(E+{\bar V}^+)) \, {\bf 1}\{(E+{\bar V}^+) \ge W(\varphi_-,\varphi_+)\} \\
&+ S_{\rho_+, \rho_+} (-(E+{\bar V}^+)) \,  {\bf 1}\{(E+{\bar V}^+) < W(\varphi_-,\varphi_+)\}.
\end{split}
\end{equation*}
 The function ${S}^+$ is not differentiable at the point $W(\varphi_-,\varphi_+) - {\bar V}^+$.

\section{Proof of Theorem \ref{th:pressure-competitive}}\label{sec:A3}

In order to prove the theorem, we have simply to compute the Legendre transform of $S^+$ whose explicit form is given in Theorem \ref{th:ent1}. Recall also that the Legendre transform of the function $S_{\rho,\rho}$ defined by (\ref{eq:entropysimple}) is given by the function $P(\varphi,\cdot)$ defined by (\ref{Pfunction}).

\subsection{The case $\frac{1}{2}\leq 1-\rho_-< \rho_+ $:}

\vspace{2mm}

\quad

\vspace{2mm}

For any $E \in (E_{+\infty}^- \, ; \, E_{+\infty}^+)$ we have that $\frac{d S^+}{dE}$ is a decreasing function and
$$\lim_{E \to E_{+\infty}^-} \frac{d S^+}{dE} =+\infty, \quad  \lim_{E \to {E}_{+\infty}^+} \frac{d S^+}{dE}=\theta_0^-.$$

We get that
\begin{equation*}
P^+ (\theta) =
\begin{cases}
\vspace{1mm}
&P(\varphi_-, -\theta) -\theta {\bar V}^+, \quad \theta \ge \theta_0^-,\\
\vspace{1mm}
& P^+ (\theta_0) + E_{+\infty}^+ (\theta- \theta_0^-), \quad \theta <\theta_0^-.
\end{cases}
\end{equation*}

\subsection{The case $\frac{1}{2}\leq \rho_+ < 1-\rho_- $:}

\vspace{2mm}

\quad

\vspace{2mm}

For any $E \in (E_{+\infty}^- \, ; \, E_{+\infty}^+)$ we have that $\frac{d S^+}{dE}$ is a decreasing function and
$$\lim_{E \to E_{+\infty}^-} \frac{d S^+}{dE} =+\infty, \quad  \lim_{E \to {E}_{+\infty}^+} \frac{d S^+}{dE}=\theta_0^+.$$

We get that
\begin{equation*}
P^+ (\theta) =
\begin{cases}
\vspace{1mm}
&P(\varphi_+, -\theta) -\theta {\bar V}^+, \quad \theta \ge \theta_0^+,\\
\vspace{1mm}
& P^+ (\theta_0) + E_{+\infty}^+ (\theta- \theta_0), \quad \theta <\theta_0^+.
\end{cases}
\end{equation*}

\subsection{ The case $\rho_- < \rho_+\leq \frac{1}{2}$:}

\vspace{2mm}

\quad

\vspace{2mm}

In this case we have that $\theta_0^+ \le  \theta_0^-$ because $\xi_0 \in (0,1/2)$. We get similarly that
\begin{equation*}
P^+ (\theta) =
\begin{cases}
\vspace{1mm}
&P(\varphi_-, -\theta) -\theta {\bar V}^+, \quad \theta \ge \theta_0^-,\\
\vspace{1mm}
& P(\varphi_+, -\theta) -\theta {\bar V}^+,  \quad \theta \le \theta_0^+,\\
\vspace{1mm}
&P^+ (\theta_0^+) + \cfrac{P^+(\theta_0^-) - P^+ (\theta_0^+)}{\theta_0^- - \theta_0^+}(\theta-\theta_0^+), \quad \theta \in (\theta_0^+, \theta_0^-).
\end{cases}
\end{equation*}

\subsection{The case $\frac{1}{2} \le \rho_- < \rho_+$:}

\vspace{2mm}

\quad

\vspace{2mm}

The function $S^+$ is differentiable everywhere in the interior of the energy band apart from the point $W(\varphi_-, \varphi_+) -{\bar V}^+$. For $E=W(\varphi_-, \varphi_+) -{\bar V}^+$, ${S^+}$ has a left-tangent and a right-tangent.  Moreover, $\frac{d S^+}{dE}$ is decreasing on $(E_{+\infty}^-\, ; \,  W(\varphi_-,\varphi_+) -{\bar V}^+)$ and increasing on $(W(\varphi_-,\varphi_+) -{\bar V}^+\, ; \, E_{+\infty}^+)$. We have
$$\lim_{E \to E_{+\infty}^\pm} \frac{d S^+}{dE} =\mp\infty \quad \textrm{and} \quad  \lim_{E \to [W(\varphi_-, \varphi_+) -{\bar V}^+]^\pm} \frac{d S^+}{dE}=\theta_0^\pm. $$

Observe that $\theta_0^- \le  \theta_0^+$ because $\xi_0 \in (1/2 ,1)$. We get easily that
\begin{equation*}
P^+ (\theta) =
\begin{cases}
\vspace{1mm}
&P(\varphi_+, -\theta) -\theta {\bar V}^+, \quad \theta \ge \theta_0^+,\\
\vspace{1mm}
& P(\varphi_-, -\theta) -\theta {\bar V}^+,  \quad \theta \le \theta_0^-,\\
\vspace{1mm}
&P^+ (\theta_0^-) + \cfrac{P^+(\theta_0^+) - P^+ (\theta_0^-)}{\theta_0^+ - \theta_0^-}(\theta-\theta_0^-), \quad \theta \in (\theta_0^-, \theta_0^+).
\end{cases}
\end{equation*}

\section{Proof of Theorem \ref{th:ent2}}
\label{sec:A2}

In this section we determine the extremal points of the domain $\D$ according to the position of $E$ along the energy band, we find the supremum of $F(\cdot,\cdot,y_-,y_+)$ among those points and then we maximize over $y_-$ and $y_+$.

\subsection{The case $\rho_- \leq \frac{1}{2}\leq \rho_+$:}

\vspace{2mm}

\quad

\vspace{2mm}

This case corresponds to $\varphi_- \leq 0\leq \varphi_+$, therefore $m=(-s)(\rho_0)$ and $M=(-s) (1/2)=\log(2)=-E_{-\infty}^--\bar{V}^- $. Since $\gamma_+'(y_+)=\varphi_+>0$, then the function $\gamma_+$  is increasing.
On the other hand, since $\gamma_-'(y_-)=\varphi_-<0$, the function $\gamma_-$ is decreasing.  Since
$\gamma_+(\rho_+):=\varphi_+/(1+e^{-\varphi_+})-\log(1+e^{\varphi_+})$ and the function $t\to  t/(1+e^{-t})-\log(1+e^t)$ is increasing in $(0,+\infty)$,
we obtain that $-M\leq{\gamma_+(\rho_+)}<\gamma_+(1) $. On the other hand, since
$\gamma_-(\rho_-)=\varphi_-/(1+e^-\varphi_-)-\log(1+e^{\varphi_-})$ and the function $t\to t/(1+e^{-t})-\log(1+e^t)$ is decreasing
in $(-\infty,0)$ we obtain that $-M\leq{\gamma_-(\rho_-)}\leq{\gamma_-(0)}$.

\vspace{3mm}

We only consider the case $\frac{1}{2} \leq 1-\rho_-< \rho_+$ (which corresponds to $\varphi_0=\varphi_+$), the case $\frac{1}{2} \leq \rho_+< 1-\rho_-$ (which corresponds to $\varphi_0=-\varphi_-$) being similar.

\vspace{2mm}

Since, $-m=\gamma_+(\rho_+)$ we have that  $-M<-m<\gamma_+(1)=E_{-\infty}^++\bar{V}^-$. As a consequence, $\gamma_+(y_+)\geq{-m}$ for all $y_+\in[\rho_+,1]$.
Notice that the function
$t\to t/(1+e^{-t})-\log(1+e^{t})$ is even and increasing in $(0,+\infty)$. Therefore, $\gamma_-(\rho_-)\leq{\gamma_+(\rho_+)}$.
On the other hand, $\gamma_-(0)\leq{\gamma_+(1)}$, which implies that  $-M\leq{\gamma_-(\rho_-)}\leq{-m}=\gamma_+(\rho_+)\leq{\gamma_+(1)}$.
Now, two things can happen, either $\gamma_-(0)>-m$ or $\gamma_-(0)<-m$. We start by the former.

\vspace{0.3cm}

 \textbf{(a) $\gamma_-(0)>-m$:} Since we do not know the sign of $\gamma_--\gamma_+$ we split again into two cases:  $\gamma_-<\gamma_+$ and $\gamma_->\gamma_+$. We  start by the former.

\quad

\textbf{(a.1) Case $\gamma_->\gamma_+$:} Recall the intersection points of the lines $D_m$ and $D_M$ from Section \ref{sec entropy 2}. In this case $X_0$ is in the square $[-1,1]^2$ if and only if $\gamma_+ \le E+{\bar V}^- \le  \gamma_-$. We first restrict to the case $X_0 \in [-1,1]^2$, i.e.  $\gamma_+ \le E+{\bar V}^- \le \gamma_-$.
Now, we check wether $X_m$, $X_M$, $Y_m$ and $Y_M$ are in the domain $D$. We start with $X_m$, and the same computations holds for $X_M$. For that purpose, it is enough to notice that $X_m$ satisfies the third equation in \eqref{eq:cons3}, that is
\begin{equation*}
\cfrac{(E+{\bar V}^-) +m }{\gamma_+ +m}  \le \cfrac{(E +{\bar V}^-) +M}{\gamma_+ +M}.
\end{equation*}
Since the function $t\to (E+\bar{V}^-+t)/(\gamma_++t)$ is decreasing and $m<M$ we conclude that $X_m$ is not in the domain $D$. Analogously one shows that $X_M$ is not in ${D}$. By replacing $\gamma_+$ by $\gamma_-$ in the computations above, one shows that $Y_m$ and $Y_M$ are in the domain $D$.

Now, if $X_0 \notin [-1,1]^2$, i.e. $E+{\bar V}^->\gamma_-$ then the same computation as done above shows that $Y_m$ and $Y_M$ are not in $D$ and as a consequence $D$ is empty; and if $E+{\bar V}^-<\gamma_+$ then $X_m$ and $X_M$ are not in $D$ and as a consequence $D$ is empty.

So, we are restricted to the case $\gamma_+ \le E+{\bar V}^- \le \gamma_-$. It remains to compute
\begin{equation*}
\sup_{(y_-,y_+) \in{\Gamma}}\Big\{F(X_0,y_-,y_+), F(Y_m,y_-,y_+), F(Y_M,y_-,y_+)\Big\},
\end{equation*}
where $\Gamma:=\{(y_-,y_+)\, : \,  \gamma_+\leq{E+\bar{V}^-}\leq{\gamma_-}\}$. Observe now that whatever the value of $y_-$ is, we have that
\begin{equation*}
    \begin{split}
&F(Y_m,y_-,y_+)= \cfrac{(E+{\bar V}^-) +m }{\gamma_- +m} \left( (-s) (y_-) + \gamma_- \right),\\
&F (Y_M,y_-,y_+)= \cfrac{(E+{\bar V}^-) +M }{\gamma_-- +M} \left( (-s) (y_-) + \gamma_- \right).
\end{split}
\end{equation*}
Now we notice that:
\begin{equation*}
\begin{split}
(-s(y_-)+\gamma_-)&=-\Big(s\Big(\frac{\gamma_-+\log(1+e^{\varphi_-})}{\varphi_-}\Big)-\gamma_-)\\
&=-\Big(s\Big(\frac{-\gamma_--\log(1+e^{-\varphi_-})}{\varphi_-}\Big)-\gamma_-\Big)\\
&=-\Big(-S_{1-\rho_-,1-\rho_-}(-\gamma_-)-\gamma_-\Big)\\
&=-J_{\rho_-,\rho_-}(-\gamma_-).
\end{split}
\end{equation*}
 In the second equality above, we wrote $y_-$ in terms of $\gamma_-$, in the third equality we used the fact that $s(\theta)=s(1-\theta)$ for
  all $\theta\in(0,1)$ and in fourth equality we used that for any $\rho \in [0,1]$,
\begin{equation}\label{ident S}
S_{\rho,\rho} =S_{1-\rho,1-\rho}
\end{equation}
 together with the definition of $J_{\rho,\rho}$ given above.
   The same argument also shows that $(-s(y_+)+\gamma_+)=-J_{\rho_+,\rho_+}(-\gamma_+)$. Then we conclude that
\begin{equation*}
\begin{split}
&F(Y_m,y_-,y_+)= - \cfrac{(E+{\bar V}^-) +m }{\gamma_- +m} J_{\rho_-,\rho_-} (-\gamma_-)\\
&F (Y_M,y_-,y_+)= - \cfrac{(E+{\bar V}^-) +M }{\gamma_- +M} J_{\rho_-,\rho_-} (-\gamma_-).
\end{split}
\end{equation*}

 Observe that since the function $t\to - ((E+{\bar V}^-) +t ) / (\gamma_- +t)$ is decreasing, $m \leq{M}$ and $J_{\rho_-,\rho_-} (-\gamma_-) \ge 0$,
 we get that
$F \left(Y_m, y_-, y_+ \right)\geq F \left(Y_M, y_-, y_+ \right).$

Now we recall some properties of the function $J_{\rho,\rho}$, for $\rho \in [0,1]$. At first we notice that by \eqref{ident S},
we have that $J_{\rho,\rho}=J_{1-\rho,1-\rho}$. The function $J_{\rho,\rho}$ is convex and positive apart from the point $-s(\rho)$ where it vanishes. As a consequence the function $J_{\rho_-,\rho_-}$ is convex, non-negative and finite on $[\log(1+e^{\varphi_-})\,;\,\log(1+e^{-\varphi_-})]$ and has a minimum equal to $0$ at the point $-s(\rho_-)=\gamma_-(\rho_-)$. Analogously, the function $J_{\rho_+,\rho_+}$ is convex, non-negative and finite on $[\log(1+e^{-\varphi_+})\,;\,\log(1+e^{\varphi_+})]$ and has a minimum equal to $0$ at the point $-s(\rho_+)= \gamma_+(\rho_+)$. Since $-s(\rho_-)\leq{m}=-s(\rho_+)$ then for all $y\in{I}:=[\log(1+e^{-\varphi_+})\,;\,m]$ we have that $J_{\rho_-,\rho_-}(y) \geq{J_{\rho_+,\rho_+}(y)}$. In particular, since $\gamma_-\in{I}$ we obtain that
$J_{\rho_-,\rho_-}(-\gamma_-)
\geq{J_{\rho_+,\rho_+}(-\gamma_-)}.$ Putting together the previous observations, the fact that $E+{\bar V}^-\geq{\gamma_+}\geq{m} $ and $\gamma_->-m$,
we conclude that
\begin{equation*}
F(Y_m,y_-,y_+)= - \cfrac{(E+{\bar V}^-) +m }{\gamma_- +m} J_{\rho_-,\rho_-} (-\gamma_-)
\leq {- \cfrac{(E+{\bar V}^-) +m }{\gamma_- +m} J_{\rho_+,\rho_+} (-\gamma_-)}.
\end{equation*}
On the other hand,  by computing the derivative of the function
\[
G(\gamma_-)=- \cfrac{(E+{\bar V}^-) +m }{\gamma_- +m} J_{\rho_+,\rho_+} (-\gamma_-)
\]
with respect to $\gamma_-$ and using the fact that the function  $J_{\rho_+,\rho_+} (-\gamma_-)$ is convex at the point $\gamma_-$,
 we conclude that $G(\cdot)$ is decreasing.  Then,
\begin{equation*}
 \begin{split}
\sup_{(y_-,y_+) \in{\Gamma}}F(Y_m,y_-,y_+) & \leq \sup_{ \gamma_-\geq E+\bar{V}^-}G(\gamma_-)\\
&=G(E+{\bar V}^-) \\&=-J_{\rho_+,\rho_+} \Big(-(E+{\bar V}^-)\Big).
\end{split}
\end{equation*}
Now we rewrite  $F \Big(X_0, y_-, y_+ \Big)$ as
\begin{equation*}
 F \Big(X_0, y_-, y_+ \Big)=-\cfrac{(E+{\bar V}^-) -\gamma_+ }{\gamma_--\gamma_+} J_{\rho_-,\rho_-} (-\gamma_-)-\cfrac{\gamma_- -(E+\bar{V}^-)}{\gamma_--\gamma_+}
 J_{\rho_+,\rho_+} (-\gamma_+).
\end{equation*}
By computing the derivative with respect to $\gamma_+$ of $F(X_0,y_-,y_+)$, and noticing that both $J_{\rho_+,\rho_+}$ and $J_{\rho_-,\rho_-}$ are convex, we obtain that $F(X_0,y_-,y_+)$ is increasing as a function of $\gamma_+$.  Then
\begin{equation*}
 \begin{split}
\sup\Big\{ F(X_0,y_-,y_+): \, (y_-,y_+) \in{\Gamma}\Big\} &\leq  \sup\Big\{F(X_0,y_-,y_+): \, \gamma_-\geq E+\bar{V}^-\Big\}\\
&=-J_{\rho_+,\rho_+} \Big(-(E+{\bar V}^-)\Big).
\end{split}
\end{equation*}

Putting together the previous computations we obtain that

 \begin{equation*}
 \begin{split}
\sup_{(y_-,y_+) \in{\Gamma}}\Big\{F(X_0,y_-,y_+), F(Y_m,y_-,y_+), F(Y_M,y_-,y_+)\Big\}=&\sup_{ \gamma_+ \le E+{\bar V}^- \le  \gamma_-}F(X_0,y_-,y_+)\\
 =&\sup_{\gamma_+=E+{\bar V}^-}\Big\{F(X_0,y_-,y_+)\\
=&-J_{\rho_+,\rho_+} \Big(-(E+{\bar V}^-)\Big).
\end{split}
\end{equation*}

\quad

\textbf{(a.2) Case $\gamma_-<\gamma_+$:} A simple computation shows that $X_0$ belongs to the square $[-1,1]^2$ if and only if $\gamma_- \le E+{\bar V}^- \le  \gamma_+$.
As above, we first restrict to $\gamma_- \le E+{\bar V}^- \le  \gamma_+$. A simple computation as performed above, shows that $X_m$ and $X_M$ are in
 $D$ and $Y_m$ and $Y_M$ are not in $D$.

 Now, if $X_0 \notin [-1,1]^2$, i.e. $E+{\bar V}^->\gamma_+$ then the same computation as done above shows that $X_m$ and $X_M$ are not in $D$ and as a consequence $D$ is empty; and if $E+{\bar V}^-<\gamma_-$ then $Y_m$ and $Y_M$ are not in $D$ and as a consequence $D$ is empty.

So, we are restricted to the case $\gamma_- \le E+{\bar V}^- \le \gamma_+$. Then, it remains to compute
\[
\sup_{(y_-,y_+) \in{\Gamma}}\Big\{F(X_0,y_-,y_+), F(X_m,y_-,y_+), F(X_M,y_-,y_+)\Big\},
\]
where $\Gamma:=\{(y_-,y_+): \gamma_-\leq{E+\bar{V}^-}\leq{\gamma_+}\}$.
By inverting the role of $\rho_-$ with $\rho_+$ and of $\gamma_-$ with $\gamma_+$ in the proof of the previous case, we obtain here that the previous supremum equals to $-J_{\rho_+,\rho_+} \Big(-(E+{\bar V}^-)\Big).$
\vspace{0.3cm}

 \textbf{(b) $\gamma_-(0)<-m$:} In this case we have that $\gamma_-<-m<\gamma_+$. As above, we have to check whether  the points $X_0,X_m,X_M,Y_m,Y_M$ are in the domain $D$ or not.

The point $X_0$ belongs to $[-1,1]^2$ if and only if $\gamma_- \le E+{\bar V}^- \le \gamma_+$. If $X_0 \notin [-1,1]^2$, i.e.  $E+{\bar V}^->\gamma_+$, then $X_m$ and $X_M$ are not in the domain $D$ and as a consequence $D$ is empty.

Then we restrict to $\gamma_- \le E+{\bar V}^- \le \gamma_+$. A simple computation shows that $X_m$ and $X_M$ belong to the domain $D$, but $Y_M, Y_m$ are not in $D$. So we have to compute
\[
\sup_{(y_-,y_+) \in{\Gamma}}\Big\{F(X_0,y_-,y_+), F(X_m,y_-,y_+), F(X_M,y_-,y_+)\Big\},
\]
where $\Gamma:=\{(y_-,y_+): \gamma_-\leq{E+\bar{V}^-}\leq{\gamma_+}\}$.

As above easily we can show that $F(X_m,y_-,y_+)\geq F(X_M,y_-,y_+)$. Now we have to compare $F(X_m,y_-,y_+)$ with $F(X_0,y_-,y_+)$.
A simple computation shows that $F \left(X_0, y_-, y_+ \right)$ can be written as
\begin{equation*}
F \Big(X_0, y_-, y_+ \Big)=-\cfrac{\gamma_+-(E+{\bar V}^-) }{\gamma_+-\gamma_-} J_{\rho_-,\rho_-} (-\gamma_-)
-\cfrac{\gamma_- -(E+\bar{V}^-)}{\gamma_--\gamma_+} J_{\rho_+,\rho_+} (-\gamma_+).
\end{equation*}
Since $ J_{\rho_-,\rho_-} $ is a positive function, $E+{\bar V}^- \le  \gamma_-$ and $\gamma_+>\gamma_-$ we have that
\begin{equation*}
F \Big(X_0, y_-, y_+ \Big)\leq -\cfrac{\gamma_- -(E+\bar{V}^-)}{\gamma_--\gamma_+} J_{\rho_+,\rho_+} (-\gamma_+).
\end{equation*}

Now, since the function $t\to -(E+{\bar V}^-+t)/(\gamma_++t)$ is decreasing and $m>-\gamma_-$ we obtain that $F(X_0,y_-,y_+)\leq{F(X_m,y_-,y_+)}$.

It follows that
\begin{equation*}
\begin{split}
&\sup_{(y_-,y_+) \in{\Gamma}}\Big\{F(X_0,y_-,y_+), F(X_m,y_-,y_+), F(X_M,y_-,y_+)\Big\}\\
=&\sup_{(y_-,y_+) \in{\Gamma}} F(X_m,y_-,y_+)\\=& \sup_{\gamma_+ \geq E+ {\bar V}^-}F(X_m,y_-,y_+) \\
 =& -J_{\rho_+,\rho_+} \Big(-(E+{\bar V}^-)\Big).
\end{split}
\end{equation*}

By the conclusions above together with \eqref{linear entropy}, we obtain that the restriction of the entropy function $S^-$ to $[E_{-\infty}^- \, ; \, E_{-\infty}^+]$ is given by
\begin{equation*}
\begin{split}
&S^-(E)=\\
&=\begin{cases}
\vspace{0.2cm}
-(E+\bar{V}^-), -\log(2)\leq{E+\bar{V}^-}\leq{\frac{\varphi_0}{1+e^{-\varphi_0}}-\log(1+e^{\varphi_0})},\\
\vspace{0.2cm}
S_{\rho_0,\rho_0}\Big(-(E+\bar{V}^-)\Big), \frac{\varphi_0}{1+e^{-\varphi_0}}-\log(1+e^{\varphi_0})<E+\bar{V}^-\leq{- \log (1+e^{-\varphi_0})}.
\end{cases}
\end{split}
\end{equation*}
Above we used the equality $S_{\rho_+,\rho_+}(E)=E-J_{\rho_+,\rho_+}(E)$. To conclude, we notice that for any $\rho=e^{\varphi}/ (1+e^{\varphi}) \in [0,1]$,
\begin{equation*}
s(\rho)=\frac{\varphi}{1+e^{-\varphi}}-\log(1+e^{\varphi})=\frac{-\varphi}{1+e^{\varphi}}-\log(1+e^{-\varphi}).
\end{equation*}

Now, we prove the last assertion of the theorem. As above, we consider the case $\varphi_0=\varphi_+$ the other case being similar. We have to split now into two cases, whether  $E+\bar V^-> s(\rho_+)$ or $E+\bar V^-\leq s(\rho_+)$. We start by the later.

Assume  $E+\bar V^-\leq s(\rho_+)$  and let $\rho$ be a profile such that ${\bb S} (\rho) = S^{-} (E)$ and ${\bb S} (\rho) +V^- (\rho) =E+{\bar V}^-$. With the notations of Proposition \ref{prop:s-}, we have ${S}^{-} (E)= {\bb S} (\rho) = {\bb S} (H^{\prime}_{\rho}) \le {\bb S} (G^{\prime}_{\rho})$. Moreover, we have seen in the proof of Proposition \ref{lem:3123} that $({\bb S} +V^-) (G^{\prime}_{\rho}) = ({\bb S} +V^-) (\rho)$. We claim now that $\rho ={G}^{\prime}_{\rho}$. Indeed, let $(a,b) \in [-1,1]$ be a maximal interval where $G_{\rho} < H_{\rho}$ (which implies $H_{\rho} (a) =G_{\rho} (a)$ and $H_{\rho} (b) =G_{\rho}(b)$). Since $G_{\rho}$ is the convex envelope of $H_{\rho}$, it implies that $G_{\rho}$ is linear on $[a,b]$. By Jensen's inequality one has that
\begin{equation*}
\begin{split}
\int_a^b (-s) (H^{\prime}_{\rho} (x)) dx &> (b-a) (-s) \left( \frac{1}{b-a} \int_a^b H_{\rho}^{\prime} (x) dx \right)\\
&= (b-a) (-s) \left( \frac{H_{\rho} (b) -H_{\rho} (a) }{b-a} \right)= (b-a) (-s) \left( \frac{G_{\rho} (b) -G_{\rho} (a) }{b-a} \right)\\
&= \int_a^b (-s) (G^{\prime}_{\rho} (x)) dx.
\end{split}
\end{equation*}
Thus, if $H_{\rho} \ne G_{\rho}$ we can find a profile ${\tilde \rho}$ (i.e. $1-G^{\prime}_{\rho}$) which satisfies the constraint ${\bb S} ({\tilde \rho}) +V^- ({\tilde \rho}) =E+{\bar V}^-$ and such that ${\bb S} ({\tilde \rho}) > {\bb S} (\rho)=S^{-} (E)$. Since this is not possible we get that $H_{\rho} = G_{\rho}$, i.e. $\rho$ is a non-increasing profile. In particular it implies that $1-\rho =g$ where $g$ is a maximizer of  (\ref{eq:s-g}).
In the proof of Proposition \ref{prop:s-} we have seen that $S^{-} (E)= -(E+{\bar V}^-)$ corresponds to the case where the supremum $\sup_{k \in K} F(k) =0$, which is equivalent to $(x_{\pm} \mp 1)(y_{\pm} - \rho_{\pm})=0$. For such a $4$-tuple $(x_-,x_+,y_-,y_+) \in K$, a maximizer $g$ of (\ref{eq:s-g}) is then given by any non-decreasing function on $[x_-, x_+]$ taking values in $[\rho_-, \rho_+]$, constant equal to $\rho_-$ on $[-1,x_-]$ and to $\rho_+$ on  $[x_+,1]$ and such that (\ref{eq:cons1}) is satisfied. Thus the set of maximizers $\rho$ of $S^{-}(E)$, when $E +{\bar V}^- \in [s (\rho_+) , s (\rho_-)]$ is given by the set of non-increasing profiles $\rho$ such that $1-\rho_+ \le \rho \le 1-\rho_-$ and satisfying $ {\bb S} (\rho) =-(E+{\bar V}^-)$.

Now assume that $E+\bar V^-> s(\rho_+)$. In all the cases ($a_1$), ($a_2$) and (b) above, the supremum of $F$ is attained at $X_0,X_0$ and $X_m$, respectively, and  $\gamma_+=E+\bar V^-$, which implies that $x_-=x_+=-1$. Then, the function $g$ realizing the supremum $\sup_{\rho}\int_{-1}^1(-s)(g(x))dx$ with the constraint $\int_{-1}^1g(x)dx=(1-x_+)y_+$ is constant and equals to $y_+$. Therefore, the profile $u_{\rho_+}$ is given by $1-y_+$. Using the definition of $y_+$ and the fact that $\gamma_+=E+\bar V^-$, it follows that $$u_{\rho_+}\equiv\frac{\log(\rho_+)-E+\bar V^-}{\log(\rho_+)-\log(1-\rho_+)}.$$ Finally, putting together \eqref{eq:1.9}, \eqref{eq: S}, \eqref{eq:entropysimple} and the expression for $u_{\rho_+}$ we recover the expression for $S^-(E)$  that is $S^{-} (E)= S_{\rho_+,\rho_+}(-(E+\bar V^-))$.

\vspace{0.1cm}

\subsection{ The case $\rho_- < \rho_+ \leq \frac{1}{2}$:}

\vspace{2mm}

\quad

\vspace{2mm}

This case corresponds to $\varphi_- < \varphi_+ \leq 0$, therefore $m=(-s)(\rho_-)$ and $M=(-s) (\rho_+)$. Since $\gamma'(y_{\pm})=\varphi_{\pm}<0$, then both functions $\gamma_{\pm}$ are decreasing.
Notice that $\gamma_+(1)=E_{-\infty}^-+\bar{V}^-$ and $\gamma_-(0)=E_{-\infty}^++\bar{V}^-$.
Then  $-M=\gamma_+(\rho_+)\geq{\gamma_+(1)}=-\log(1+e^{-\varphi_+})$. Analogously,
$-m=\gamma_{-}(\rho_-)\leq{\gamma_-(0)}=-\log(1+e^{\varphi_-})$.
As a consequence we have the following inequalities:
\begin{equation*}
-\log (1+e^{-\varphi_+})=\gamma_+(1)\leq  \gamma_+(\rho_+)= -M < -m \le \gamma_-(\rho_-)\leq{\gamma_-(0)} =-\log (1 +e^{\varphi_-}).
\end{equation*}

These inequalities imply that $\gamma_+< \gamma_-$.
As in the previous case we have to check whether the intersection points are in the domain $D$.
At first we notice that $X_0$ belongs to $[-1,1]^2$ if and only if $\gamma_+ \le E+{\bar V}^- \le \gamma_-$.
Since  $-(E+{\bar V}^-) \notin [m,M]$ we have only to distinguish two cases: $ \gamma_+ \le E+{\bar V}^- < -M$ and  $ -m\le E+{\bar V}^- < \gamma_-$.

If $X_0$ is not in $[-1,1]^2$, i.e. $ E+{\bar V}^-> \gamma_-$ then $X_m$ and $X_M$ are not in $D$ and as a consequence $D$ is empty; and if
$ E+{\bar V}^-< \gamma_+$ then $Y_m$ and $Y_M$ are not in $D$ and as a consequence $D$ is empty.
So we restrict to $ \gamma_+ \le E+{\bar V}^- < -M$ and  $ -m\le E+{\bar V}^- < \gamma_-$. We start by the former.

\vspace{0.3cm}

 \textbf{(a) $ \gamma_+ \le E+{\bar V}^- < -M$:} In this case, a simple computation shows that $X_m$ and $X_M$ are in $D$, but $Y_m$ and $Y_M$ are not in $D$.
Then we have to compute
\[
\sup_{ (y_-,y_+) \in{\Gamma}}\Big\{F(X_0,y_-,y_+), F(X_m,y_-,y_+), F(X_M,y_-,y_+)\Big\},
\]
where $\Gamma:=\{(y_-,y_+): \gamma_+ \le E+{\bar V}^- < -M\}$.

As above, by noticing that the function $t\to - ((E+{\bar V}^-) +t)/ (\gamma_+ +t)$ is increasing, $m \leq{M}$ and $J_{\rho_+,\rho_+} (-\gamma_+) \ge 0$,
we get that $F \left(X_m, y_-, y_+ \right)\le F \left(X_M, y_-, y_+ \right)$. On the other hand
\begin{equation*}
F \left(X_0, y_-, y_+ \right)= -\cfrac{(E+{\bar V}^-) -\gamma_+ }{\gamma_--\gamma_+} J_{\rho_-,\rho_-} (-\gamma_-)-\cfrac{\gamma_- -(E+\bar{V}^-)}{\gamma_--\gamma_+} J_{\rho_+,\rho_+} (-\gamma_+).
\end{equation*}
Since the function $J_{\rho_-,\rho_-} $ is positive, $\gamma_->\gamma_+$ and  $E+\bar{V}^->\gamma_+$, we have that
\begin{equation*}
F \left(X_0, y_-, y_+ \right)\leq -\cfrac{\gamma_- -(E+\bar{V}^-)}{\gamma_--\gamma_+} J_{\rho_+,\rho_+} (-\gamma_+)=
-\cfrac{(E+\bar{V}^-)-\gamma_- }{\gamma_+-\gamma_-} J_{\rho_+,\rho_+} (-\gamma_+).
\end{equation*}
Finally , since the function $t\to  - ((E+{\bar V}^-) +t) / (\gamma_+ +t)$ is increasing, $-\gamma_-<M$ and $J_{\rho_+,\rho_+} (-\gamma_+)\geq{0}$
we obtain that $F \left(X_0, y_-, y_+ \right)\leq F(X_M,y_-,y_+)$. Since  the function $J_{\rho_+,\rho_+} $ is convex, with a unique minimum at $M$ equal to $0$, it follows that
\begin{equation*}
\begin{split}
\sup_{(y_-,y_+) \in{\Gamma}}\Big\{F(X_0,y_-,y_+), F(X_m,y_-,y_+), F(X_M,y_-,y_+)\Big\}=&\sup_{(y_-,y_+) \in{\Gamma}} F(X_M,y_-,y_+)\\
=&\sup_{\gamma_+ \le (E+ {\bar V}^-)}F(X_M,y_-,y_+)\\
=&-J_{\rho_+,\rho_+} \Big(-(E+{\bar V}^-)\Big).\\
\end{split}
\end{equation*}

\vspace{0.3cm}

 \textbf{(b) $-m < E+{\bar V}^- \le \gamma_-$:} In this case, a simple computation shows that $Y_m$ and $Y_M$ are in $D$, but $X_m$ and $X_M$ are not in $D$.
Then we have to compute
\[
\sup_{(y_-,y_+) \in{\Gamma}}\Big\{F(X_0,y_-,y_+), F(Y_m,y_-,y_+), F(Y_M,y_-,y_+)\Big\},
\]
where $\Gamma:=\{(y_-,y_+): -m < E+{\bar V}^- \le \gamma_-\}$.

As above, by noticing that the function $t\to -((E+{\bar V}^-) +t) / (\gamma_- +t)$ is increasing, $m \leq{M}$ and $J_{\rho_-,\rho_-} (-\gamma_-) \ge 0$,
we get $F \left(Y_m, y_-, y_+ \right)\le F \left(Y_M, y_-, y_+ \right).$

As above we can show that $F(X_0,y_-,y_+)\leq{F(Y_M,y_-,y_+)}$ and as a consequence
\begin{equation*}
\begin{split}
\sup_{(y_-,y_+) \in{\Gamma}}\Big\{F(X_0,y_-,y_+), F(Y_m,y_-,y_+), F(Y_M,y_-,y_+)\Big\}=&\sup_{(y_-,y_+) \in{\Gamma}} F(Y_M,y_-,y_+)\\
=&\sup_{\gamma_- \le (E+ {\bar V}^-)}F(Y_M,y_-,y_+)\\
=&-J_{\rho_-,\rho_-} \Big(-(E+{\bar V}^-)\Big).\\
\end{split}
\end{equation*}

Collecting the previous facts and by \eqref{linear entropy}, we have that the restriction of the entropy function $S^-$ to $[E_{-\infty}^- \, ; \, E_{-\infty}^+]$ is given by

\begin{equation*}
\begin{split}
&S^-(E)=\\
&=\begin{cases}
\vspace{0.2cm}
S_{\rho_+,\rho_+} \Big(-(E+{\bar V}^-)\Big),  -\log(1+e^{-\varphi_+})\leq{E+\bar{V}^-}
<\frac{-\varphi_+}{1+e^{\varphi_+}}-\log(1+e^{-\varphi_+}),\\
\vspace{0.2cm}
-(E+\bar{V}^-),  \frac{-\varphi_+}{1+e^{\varphi_+}}-\log(1+e^{-\varphi_+})\leq{E+\bar{V}^-}\leq{\frac{-\varphi_-}{1+e^{\varphi_-}}-\log(1+e^{-\varphi_-})},\\
\vspace{0.2cm}
S_{\rho_-,\rho_-} \Big(-(E+{\bar V}^-)\Big),  \frac{-\varphi_-}{1+e^{\varphi_-}}-\log(1+e^{-\varphi_-})<E+\bar{V}^-\leq{-\log (1+e^{\varphi_-})}
\end{cases}
\end{split}
\end{equation*}

Now, we prove the last assertion of the theorem. As above, we have to split into several cases, whether  $E+\bar V^-< s(\rho_+)$, $s(\rho_+)\leq E+\bar V^-\leq s(\rho_-)$ or $E+\bar V^->s(\rho_-)$. We start by the first case, the third being completely similar. Analogously to what we have done for $\rho_-\leq 1/2\leq \rho_+$, it is enough to notice that in the cases ($a$) and (b) above, the supremum is attained at $X_M$ and $Y_M$, respectively, with  $\gamma_+=E+\bar V^-$, which implies that $x_-=x_+=-1$. The rest of the argument follows as above. The second case, follows by reasoning as in the case $\rho_- \leq \frac{1}{2}\leq \rho_+$.

\vspace{0.1cm}

\subsection{  The case $\frac{1}{2}\leq\rho_- < \rho_+$:}

\vspace{2mm}

\quad

\vspace{2mm}

Repeating the same computations as performed in the previous situation, we can show that the restriction of the entropy function $S^-$ to $[E_{-\infty}^- \, ; \, E_{-\infty}^+]$ is given by

\begin{equation*}
\begin{split}
&S^-(E)=\\
&=\begin{cases}
\vspace{0.2cm}S_{\rho_-,\rho_-}\Big(-(E+{\bar V}^-)\Big),  -\log(1+e^{\varphi_-})\leq{E+\bar{V}^-}<\frac{-\varphi_-}{1+e^{\varphi_-}}-\log(1+e^{-\varphi_-}),\\
\vspace{0.2cm}
-(E+\bar{V}^-),  \frac{-\varphi_-}{1+e^{\varphi_-}}-\log(1+e^{-\varphi_-})\leq{E+\bar{V}^-}\leq{\frac{-\varphi_+}{1+e^{\varphi_+}}-\log(1+e^{-\varphi_+})}\\
\vspace{0.2cm}
S_{\rho_+,\rho_+}\Big(-(E+{\bar V}^-)\Big),  \frac{-\varphi_+}{1+e^{\varphi_+}}-\log(1+e^{-\varphi_+})
<E+\bar{V}^-\leq{-\log(1+e^{-\varphi_+})}.
\end{cases}
\end{split}
\end{equation*}

We notice that to prove the last assertion of the theorem is is enough to invert the role of $\rho_-$
 with $\rho_+$ in the proof of the previous case.

\section{Proof of Theorem \ref{th:pressure-cooperative}}\label{sec:A4}

To prove this theorem we compute explicitly the Legendre transform of $P^-$, namely, ${\tilde S}^- (E)=\inf_{\theta \in \RR} \left\{ \theta E -P^{-} (\theta) \right\}$ and we show that it coincides with the expression for $S^{-} (E)$ obtained in the previous section. Since the Legendre transform is a one to one correspondence between concave functions, this is sufficient to conclude. We denote by $P_0$ the function defined by ${ P}_0 (\theta) = -\theta \log (m(\theta))$.

\subsection{The case $\rho_- \leq \frac{1}{2} \leq  \rho_+$:}

\vspace{2mm}

\quad

\vspace{2mm}

This case corresponds to $\varphi_- \leq 0 \leq  \varphi_+$. Recall that ${\varphi}_0 = \sup(|\varphi_-|, |\varphi_+|)$.
Since
\begin{equation*}
m(\theta)=\min (f_{\theta} (\varphi_{-}), f_{\theta} ({\varphi_+}))=
\begin{cases}
\vspace{1mm}
f_{\theta} (\varphi_0), \quad \theta <-1,\\
\vspace{1mm}
1, \quad \theta=-1,\\
\vspace{1mm}
f_{\theta} (0), \quad \theta>-1,
\end{cases},
\end{equation*}
we have that
\begin{equation*}
{P} _0 (\theta) =
 \begin{cases}
 \vspace{1mm}
  -\theta \log (1+e^{-{\varphi}_0}) - \log (1+e^{\theta {\varphi}_0}), \quad \theta <-1,\\
 \vspace{1mm}
 0, \quad \theta=-1,\\
 \vspace{1mm}
 - (\theta+1)\log 2 , \quad \theta >-1.
\end{cases}
\end{equation*}
As a consequence
\begin{equation*}
{P}' _0 (\theta) =
 \begin{cases}
 \vspace{1mm}
-\log (1+e^{-{\varphi}_0}) - \cfrac{\varphi_0}{1+e^{-\theta \varphi_0}}, \quad \theta <-1,\\
 \vspace{1mm}
 -\log(2) , \quad \theta >-1.
 \end{cases}
\end{equation*}

The function ${ P}_0$ is differentiable everywhere except for $\theta=-1$.  A simple computation show that $P_0'$ is decreasing in $(-\infty,-1)$.
 We have that
\begin{equation*}
\begin{split}
&P_0'(-1^-)=\lim_{\substack{\theta \to -1\\ \theta <-1}}{P}_0^{\prime} (\theta)=-\log (1+e^{-{\varphi}_0})-\cfrac{\varphi_0}{1+e^{\varphi_0}},\\
 &P_0'(-\infty)=\lim_{\theta \to -\infty} { P}_0^{\prime} (\theta)=- \log(1+e^{-\varphi_0}).\\
\end{split}
\end{equation*}
Observe that $P_0'(-\infty)>P_0'(-1^-)>- \log(2)$,
which is a consequence of the function $t\to -\log(1+e^{-t}) -t/(1+e^{t})$ being decreasing on $(-\infty,0]$. This implies that the function ${P}_0$ is concave. Now we compute ${\tilde S}^-(E)$. By the previous observations we have that
\begin{equation*}
{\tilde S}^-(E)=\inf\Big\{\inf_{\theta>-1}\{\theta (E+\bar{V}^-)-P_0(\theta)\},-(E+\bar{V}^-),\inf_{\theta<-1}\{\theta (E+\bar{V}^-)-P_0(\theta)\}\Big\}.
\end{equation*}
Now, for $\theta>-1$ we have that
\begin{equation*}
\inf_{\theta>-1} \Big\{\theta (E+\bar{V}^-)-P_0(\theta)\}=\log(2)+\inf_{\theta>-1}\{\theta ((E+\bar{V}^-)+\log(2))\Big\}.
\end{equation*}
This equals to $-(E+\bar{V}^-)$ if $(E+\bar{V}^-)>-\log(2)$ and equals $-\infty$ if $(E+\bar{V}^-)<-\log(2)$. On the other hand, for $\theta<-1$ we have that
\begin{equation*}
\inf_{\theta<-1}\Big\{\theta (E+\bar{V}^-)-P_0(\theta)\}=\inf_{\theta<-1}\{\theta (E+\bar{V}^-)-P(-\varphi_0,-\theta)\Big\},
\end{equation*}
where $P(\cdot,\cdot)$ is defined in \eqref{Pfunction}. Let $I_1:= \Big[\frac{-\varphi_0}{1+e^{\varphi_0}}-\log (1+e^{-\varphi_0}) \; ; \; - \log (1+e^{-\varphi_0})\Big]  \subset (-\infty,0).$ We have that
\begin{equation*}
\begin{split}
\inf_{\theta \in \RR} \Big \{\theta (E+\bar{V}^-) -P(-\varphi_0, -\theta) \Big\}&=\frac{1}{2}\inf_{\theta \in \RR} \Big\{ \theta 2(E+\bar{V}^-) -2P(-\varphi_0, -\theta) \Big\}\\
&=\frac{1}{2}\inf_{\theta \in \RR} \Big\{ \theta (-2(E+\bar{V}^-)) -2P(-\varphi_0, \theta) \Big\}\\&=-s\left(\cfrac{-(E+\bar{V}^-) -\log(1+e^{-\varphi_0})}{\varphi_0}\right)\\
&=S_{\rho_0,\rho_0} \Big(-(E+\bar{V}^-)\Big),
\end{split}
\end{equation*}
for any $E+\bar{V}^-\in I_1$. Then, we conclude that
\begin{equation*}
\inf_{\theta>-1}\Big\{\theta (E+\bar{V}^-)-P_0(\theta)\Big\}=S_{\rho_0,\rho_0}\Big(-(E+\bar{V}^-)\Big),
\end{equation*}
for $E+\bar{V}^-\in I_1$.
Now we look for ${\tilde S}^-(E)$ for $E+\bar{V}^-$ outside $I_1$. Let $I_1^-:=\Big(-\infty\, ;\,\frac{-\varphi_0}{1+e^{\varphi_0}}-\log (1+e^{-\varphi_0})\Big) $ and $I_1^+:=\Big(- \log (1+e^{-\varphi_0})\, ;\,+\infty\Big)$. Then, if $E+\bar{V}^-\in{I_1^-}$, we have that $sign(E+\bar{V}^--P_0'(\theta))=sign((E+\bar{V}^-)-P_0'(-\infty))=sign((E+\bar{V}^-)+\log(1+e^{-\varphi_0}))<0$, and as a consequence the infimum is attained at $\theta=-1$ and in this case ${\tilde S}^-(E)=-(E+\bar{V}^-)$. On the other hand if $E+\bar{V}^-\in{I_1^+}$, we have that $sign(E+\bar{V}^--P_0'(\theta))=sign((E+\bar{V}^-)-P_0'(-\infty))>0$, and as a consequence the infimum is attained at $\theta=-\infty$ and in this case ${\tilde S}^-(E)=-\infty$.

Finally we conclude that ${\tilde S}^-$ when restricted to the energy band $[E_{-\infty}^- \, ; \, E_{-\infty}^+]$ is given by
\begin{equation*}
\begin{split}
&{\tilde S}^-(E)=\\
&=\begin{cases}
\vspace{0.2cm}
-(E+\bar{V}^-), \quad -\log(2)\leq{E+\bar{V}^-}\leq{\frac{-\varphi_0}{1+e^{\varphi_0}}-\log (1+e^{-\varphi_0})},\\
\vspace{0.2cm}
S_{\rho_0,\rho_0}\Big(-(E+\bar{V}^-)\Big), \quad \frac{-\varphi_0}{1+e^{\varphi_0}}- \log (1+e^{-\varphi_0})<E+\bar{V}^-\leq{- \log (1+e^{-\varphi_0})}
\end{cases}
\end{split}
\end{equation*}
and equal to $-\infty$ outside the energy band. Thus ${\tilde S}^-$ coincides with ${S}^-$.

\subsection{ The case $\rho_-<\rho_+\leq \frac{1}{2}$:}

\vspace{2mm}

\quad

\vspace{2mm}

 In this case we have that
\begin{equation*}
m(\theta)=\min (f_{\theta} (\varphi_{-}), f_{\theta} ({\varphi_+}))=
\begin{cases}
\vspace{1mm}
f_{\theta} (\varphi_-), \quad \theta <-1,\\
\vspace{1mm}
1, \quad \theta=-1,\\
\vspace{1mm}
f_{\theta} (\varphi_+), \quad \theta>-1,
\end{cases}
\end{equation*}
and as a consequence
\begin{equation*}
{ P}_0'(\theta)=
\begin{cases}
\vspace{1mm}
-\cfrac{\varphi_-}{1+e^{-\theta \varphi_-}} -\log(1+e^{-\varphi_-}),\quad \theta <-1,\\
\vspace{1mm}
-\cfrac{\varphi_+}{1+e^{-\theta \varphi_+}}-\log(1+e^{-\varphi_+}), \quad \theta>-1.
\end{cases}
\end{equation*}
The function ${ P}_0$ is differentiable everywhere except for $\theta=-1$.  We have that ${P}_0^{\prime}$ is decreasing on $(-\infty, -1)$ and
\begin{equation*}
 P_0'(-\infty)>P_0'(-1^-)=\lim_{\substack{\theta \to -1\\ \theta <-1}} P_0^{\prime} (\theta) >  \lim_{\substack{\theta \to -1\\ \theta >-1}} P_0^{\prime}, (\theta)=P_0'(-1^+)>P_0'(+\infty)
\end{equation*}
which is a consequence of the function $t\to \log(1+e^{-t}) +t/(1+e^{t})$ being increasing on $(-\infty,0)$. This implies that the function ${P}_0$ is concave. Now we compute ${{\tilde S}}^-(E)$. From the previous observations it follows that
\begin{equation*}
{\tilde S}^-(E)=\inf\Big\{\inf_{\theta>-1}\{\theta (E+\bar{V}^-)-P_0(\theta)\},-(E+\bar{V}^-),\inf_{\theta<-1}\{\theta (E+\bar{V}^-)-P_0(\theta)\}\Big\}.
\end{equation*}
We start by the case $\theta>-1$. Since $P_0$ is concave, then $P_0'$ is decreasing. A simple computation shows that $P'_0(\theta)=E+\bar{V}^-$ for
\[
\theta(E):=-\frac{1}{\varphi_+}\log\Big(\frac{-\varphi_+-(E+\bar{V}^-)-\log(1+e^{-\varphi_+})}{E+\bar{V}^-+\log(1+e^{-\varphi_+})}\Big).
\]

Since $P_0'$ is decreasing, its image is given by $I_1:=[P_0'(+\infty)\, ; \,P_0'(-1)]$, that is $I_1=[-\log(1+e^{-\varphi_+})\, ;\, \frac{-\varphi_+}{1+e^{\varphi_+}}-\log(1+e^{-\varphi_+})].$
Then, for $E+\bar{V}^-\in{I_1}$ we have that
\begin{equation*}
{\tilde S}^-(E)=\theta(E)(E+\bar{V}^-)-P_0(\theta(E))=S_{\rho_+,\rho_+}\Big(-(E+\bar{V}^-)\Big).
\end{equation*}
Now we look for ${\tilde S}^-(E)$ for $E+\bar{V}^-$ outside $I_1$. Let $I_1^-:=(-\infty\, ; \,-\log(1+e^{-\varphi_+}))$
and $I_1^+:=( \frac{-\varphi_+}{1+e^{\varphi_+}}-\log(1+e^{-\varphi_+})\, ; \, +\infty)$.
Then, if $E+\bar{V}^-\in{I_1^-}$, we have that $sign(E+\bar{V}^--P_0'(\theta))=sign(E+\bar{V}^--P_0'(+\infty))=sign(E+\bar{V}^-+\log(1+e^{-\varphi_+}))<0$, and as a consequence the infimum is attained at $\theta=+\infty$ and in this case ${\tilde S}^-(E)=-\infty$. On the other hand if $E+\bar{V}^-\in{I_1^+}$, we have that $sign(E+\bar{V}^--P_0'(\theta))=sign(E+\bar{V}^--P_0'(+\infty))>0$, and as a consequence the infimum is attained at $\theta=-1$ and in this case ${\tilde S}^-(E)=-\infty$.

Now we look to the case $\theta<-1$. Since $P_0$ is concave, then $P_0'$ is decreasing. A simple computation shows that $P'_0(\theta)=E+\bar{V}^-$ for
\[
\theta(E):=-\frac{1}{\varphi_-}\log\Big(\frac{-\varphi_--(E+\bar{V}^-)-\log(1+e^{-\varphi_-})}{E+\bar{V}^-+\log(1+e^{-\varphi_-})}\Big).
\]
Since $P_0'$ is decreasing, its image is given by $I_2:=[P_0'(-1)\, ; \, P_0'(-\infty)]$, that is $I_2=[\frac{-\varphi_-}{1+e^{\varphi_-}}-\log(1+e^{-\varphi_-})\, ;\, -\log(1+e^{\varphi_-})].$
Then, for $E+\bar{V}^-\in{I_2}$ we have that ${S}^-(E):=\theta(E)(E+\bar{V}^-)-P_0(\theta(E))$ which can be written as  $S_{\rho_-,\rho_-}\Big(-(E+\bar{V}^-)\Big)$.
Now we look for ${\tilde S}^-(E)$ for $E+\bar{V}^-$ outside $I_2$. Let $I_2^-:=(-\infty\, ; \, \frac{-\varphi_-}{1+e^{\varphi_-}}-\log(1+e^{-\varphi_-}))$
and $I_2^+:=(-\log(1+e^{\varphi_-})\, ;\, +\infty)$.
Then, if $E+\bar{V}^-\in{I_2^-}$, we have that $sign(E+\bar{V}^--P_0'(\theta))=sign(E+\bar{V}^--P_0'(-\infty))=sign(E+\bar{V}^-+\log(1+e^{\varphi_-}))<0$, and as a consequence the infimum is attained at $\theta=-1$ and in this case ${\tilde S}^-(E)=-(E+\bar{V}^-)$. On the other hand if $E+\bar{V}^-\in{I_2^+}$, we have that $sign(E+\bar{V}^--P_0'(\theta))=sign(E+\bar{V}^--P_0'(\theta))>0$, and as a consequence the infimum is attained at $\theta=-\infty$ and in this case ${\tilde S}^-(E)=-\infty$.

Now, a simple computation shows that the function $t\to -\log(1+e^{-t}) -t/(1+e^{t})$ is decreasing on $(-\infty,0]$ so that the sets $I_1$ and $I_2$ do not intersect. The restriction of the function ${\tilde S}^-$ to $[E_{-\infty}^- \, ; \, E_{-\infty}^+]$ has the expression:
\begin{equation*}
\begin{split}
&{\tilde S}^-(E)=\\
&=\begin{cases}
\vspace{0.2cm}
S_{\rho_+,\rho_+}\Big(-(E+\bar{V}^-)\Big), \quad -\log(1+e^{-\varphi_+})\leq{E+\bar{V}^-}<\frac{-\varphi_+}{1+e^{\varphi_+}}-\log(1+e^{-\varphi_+}),\\
\vspace{0.2cm}
-(E+\bar{V}^-), \quad \frac{-\varphi_+}{1+e^{\varphi_+}}-\log(1+e^{-\varphi_+})\leq{E+\bar{V}^-}\leq{\frac{-\varphi_-}{1+e^{\varphi_-}}-\log(1+e^{-\varphi_-}),}\\
\vspace{0.2cm}
S_{\rho_-,\rho_-}\Big(-(E+\bar{V}^-)\Big), \quad \frac{-\varphi_-}{1+e^{\varphi_-}}-\log(1+e^{-\varphi_-})<E+\bar{V}^-\leq{-\log(1+e^{\varphi_-})}
\end{cases}
\end{split}
\end{equation*}
and is equal to $-\infty$ outside the energy band. Thus ${\tilde S}^-$ coincides with ${S}^-$.

\subsection{ The case $\frac{1}{2}\leq \rho_-<\rho_{+}$:}

\vspace{2mm}

\quad

 \vspace{2mm}

 In this case, by inverting the role of $\rho_-$ with $\rho_+$ and of $\varphi_-$ with $\varphi_+$ in the previous case, we obtain  that  ${\tilde S}^-$ restricted to $[E_{-\infty}^- \, ; \, E_{-\infty}^+]$ is given by
\begin{equation*}
\begin{split}
&{\tilde S}^-(E)=\\
&=\begin{cases}
\vspace{0.2cm}
S_{\rho_-,\rho_-}\Big(-(E+\bar{V}^-)\Big), \quad -\log(1+e^{\varphi_-})\leq{E+\bar{V}^-}<\frac{-\varphi_-}{1+e^{\varphi_-}}-\log(1+e^{-\varphi_-}),\\
\vspace{0.2cm}
-(E+\bar{V}^-), \quad \frac{-\varphi_-}{1+e^{\varphi_-}}-\log(1+e^{-\varphi_-})\leq{E+\bar{V}^-}\leq{\frac{-\varphi_+}{1+e^{\varphi_+}}-\log(1+e^{-\varphi_+})},\\
\vspace{0.2cm}
S_{\rho_+,\rho_+}\Big(-(E+\bar{V}^-)\Big), \quad \frac{-\varphi_+}{1+e^{\varphi_+}}-\log(1+e^{-\varphi_+})<E+\bar{V}^-\leq{-\log(1+e^{-\varphi_+})}
\end{cases}
\end{split}
\end{equation*}
and is equal to $-\infty$ outside the energy band. Thus ${\tilde S}^-$ coincides with ${S}^-$.

\vspace{0.5cm}
\noindent{\bf Acknowledgments.}\\
The authors are very grateful to Christophe Bahadoran and Bernard Derrida for very useful discussions. They acknowledge the support of  \' Egide (France) and FCT (Portugal) through the research project "Fluctuations of weakly and strongly asymmetric systems" no. FCT/1560/25/1/2012/S. CB acknowledges the support of the French Ministry of  Education through the grants ANR-10-BLAN 0108 (SHEPI).  PG thanks FCT for support through the research project PTDC/MAT/109844/2009 and the Research Centre of Mathematics of the University of Minho, for the financial support provided by "FEDER" through the "Programa Operacional Factores de Competitividade - COMPETE" and FCT through the research project PEst-C/MAT/UI0013/2011.

\end{document}